\documentclass[11pt]{article}

\usepackage{graphicx}
\usepackage{latexsym,amsmath,amsfonts,amscd, amsthm, dsfont}
\usepackage{bm,color}
\usepackage{epsfig,verbatim,epstopdf,graphics}
\usepackage{subfigure}
\usepackage{changebar}
\usepackage{multirow}

\usepackage{algorithmic}% http://ctan.org/pkg/algorithms
\usepackage{yhmath}%¼Ó»¡ÏßÍ·±êµÄ×ÖÌå°ü  \wideparen
\usepackage{booktabs} %\cmidrule(lr){2-5} \cmidrule(lr){6-9}
\usepackage{tikz}
\usepackage{verbatim}
\usetikzlibrary{arrows,backgrounds,snakes,shapes}
\numberwithin{equation}{section}

\graphicspath{{./}{./figure/}}
\allowdisplaybreaks

\newtheoremstyle{plainNoItalics}{}{}{\normalfont}{}{\bfseries}{.}{ }{}

\theoremstyle{plain}
\newtheorem{thm}{Theorem}[section]

\theoremstyle{plainNoItalics}

\newtheorem{lem}[thm]{Lemma}

\newtheorem{exa}[thm]{Example}

\renewcommand{\theequation}{\thesection.\arabic{equation}}
\renewcommand{\thefigure}{\thesection.\arabic{figure}}
\renewcommand{\thetable}{\thesection.\arabic{table}}

\newcommand{\bx}{{\bf x}}

\newcommand{\mD}{{\mathcal D}}
\newcommand{\mP}{{\mathcal P}}
\newcommand{\mO}{{\mathcal O}}

\newcommand{\be}{\begin{eqnarray}}
\newcommand{\ee}{\end{eqnarray}}
\newcommand{\beno}{\begin{eqnarray*}}
	\newcommand{\eeno}{\end{eqnarray*}}

\makeatletter

\newcommand{\Rmnum}[1]{\expandafter\@slowromancap\romannumeral #1@}
\makeatother

\begin{document}

\baselineskip=1.8pc

%\vspace*{.10in}

%=============  title  =========================
\begin{center}
{\bf
A Kernel Based High Order \lq\lq Explicit'' Unconditionally Stable Scheme for Time Dependent Hamilton-Jacobi Equations
}
\end{center}

\vspace{.2in}
\centerline{
Andrew Christlieb\footnote{
 Department of Computational Mathematics, Science and Engineering, Department of Mathematics  and  Department of Electrical Engineering, Michigan State University, East Lansing, MI, 48824. E-mail: christli@.msu.edu
},
Wei Guo \footnote{
	Department of Mathematics and Statistics, Texas Tech University, Lubbock, TX, 70409. Research
	is supported by NSF grant NSF-DMS-1620047. E-mail: weimath.guo@ttu.edu
},
and 
Yan Jiang\footnote{Department of Mathematics, Michigan State University, East Lansing, MI, 48824. E-mail: jiangyan@math.msu.edu}
}

\bigskip
\noindent
{\bf Abstract.}

In this paper,  a class of high order numerical schemes is proposed for solving Hamilton-Jacobi (H-J) equations. This work is regarded as an extension of our previous work for nonlinear degenerate parabolic equations, see Christlieb et al. \cite{christlieb2017kernel},
%\emph{arXiv preprint arXiv:1707.09294},
which relies on a special kernel-based formulation of the solutions and successive convolution. When applied to the H-J equations, the newly proposed scheme attains genuinely high order accuracy in both space and time, and more importantly, it is unconditionally stable, hence allowing for much larger time step evolution compared with other explicit schemes and saving computational cost. A high order weighted essentially non-oscillatory methodology and a novel nonlinear filter are further incorporated to capture the correct viscosity solution. Furthermore, by coupling the recently proposed inverse Lax-Wendroff boundary treatment technique, this method is very flexible in handing complex geometry as well as general boundary conditions.   
We perform numerical experiments on a collection of numerical examples, including H-J equations with linear, nonlinear, convex or non-convex Hamiltonians. The efficacy and efficiency of the proposed scheme in approximating the viscosity solution of general H-J equations  is verified.

\vfill


%\newpage

\section{Introduction}

In this paper, we propose a class of high order non-oscillatory numerical schemes for approximating the viscosity solution to the Hamilton-Jacobi (H-J) equation 
\begin{align}
\label{eq:HJ}
\begin{cases}
%\phi_{t}+H(\phi_{x_{1}},\ldots,\phi_{x_{d}})=0,\quad x\in\Omega\subset \mathbb{R}^{d}\\
\phi_{t}+H(\nabla_{\overline{\bx}} \phi,\overline{\bx},t)=0,\quad \overline{\bx}\in\Omega\subset \mathbb{R}^{d},\\
\phi(\overline{\bx},0)=\phi^{0}(\overline{\bx}),
\end{cases}
\end{align}
with suitable boundary conditions. The H-J equations find applications in diverse fields, including optimal control, seismic waves, crystal growth, robotic navigation, image processing, and calculus of variations. The well-known property of the H-J equations is that the weak solutions may not be unique, and the concept of viscosity solutions was introduced  to single out the unique physically relevant solution \cite{crandall1983viscosity, crandall1984some}. The viscosity solution to the H-J equation \eqref{eq:HJ} is known to be Lipschitz continuous but may develop discontinuous derivatives in finite time, even with a $\mathcal{C}^\infty$ initial condition.

A variety of numerical schemes have been proposed to solve the H-J equation \eqref{eq:HJ} in the literature, including for example, the monotone schemes \cite{crandall1984two, abgrall1996numerical, lafon1996high}, the essentially non-oscillatory (ENO) schemes \cite{osher1988fronts, osher1991high} or the weighted ENO (WENO) schemes \cite{jiang2000weighted, zhang2003high}, the Hermite WENO (HWENO) schemes \cite{qiu2005hermite, qiu2007hermite, tao2017dimension, zheng2017finite, zhu2013hermite}, the discontinuous Galerkin methods \cite{hu1999discontinuous, lepsky2000analysis, cheng2007discontinuous, yan2011local, cheng2014new}. More details about recent development of high-order numerical methods for solving the H-J equations can be found in the review paper \cite{shu2007high}.
Most of these methods are in the Method of Lines (MOL) framework, which means that the discretization is first carried out for the spatial variable, then the numerical solution is updated in time by coupling a suitable time integrator. The most commonly used time evolution methods are the strong-stability-preserving
Runge-Kutta (SSP RK) schemes and SSP multi-step schemes \cite{gottlieb2001strong, shu2002survey, gottlieb2005high}, which can preserve the strong stability in some desired norm and prevent spurious oscillations near spatial discontinuities. However, it is well known that an explicit time discretization does have a restriction on the time step in order to maintain linear stability.  Recently,  optimization algorithms based on the Hopf formula were developed for solving a class of  H-J equations from optimal control and differential games \cite{darbon2016algorithms,chow2017algorithm}. Such algorithms are able to overcome the curse of dimensionality when solving \eqref{eq:HJ} with large $d$ and have attracted a lot of attention. However, the applicability of this type of algorithms depends on the Hopf formula, which is not available for general H-J equations, such as when the Hamiltonian $H$ and the initial condition $\phi^0$ are both non-convex.

 Another framework named the Method of Lines Transpose (MOL$^T$) has been exploited in the literature for solving time-dependent partial differential equations (PDEs). It is also known as Rothes method or transverse method of lines \cite{salazar2000theoretical, schemann1998adaptive, causley2014method}. In such a framework for solving linear PDEs, the temporal variable is first discretized, resulting in a set of linear boundary value problems (BVPs) at  discrete time levels. Each BVP can be inverted analytically in an integral formulation based on a Green's/kernel function and then the numerical solution is updated accordingly. As a notable advantage, the MOL$^T$ approach is able to use an implicit method but avoid solving linear systems at each time step, see \cite{causley2014method}. Moreover, a fast convolution algorithm is developed to reduce the computational complexity of the scheme from $\mathcal{O}(N^2)$ to $\mathcal{O}(N)$ \cite{causley2013method, greengard1987fast, barnes1986hierarchical}, where $N$ is the number of discrete mesh points. Over the past several years, the MOL$^T$ methods have been developed for solving the heat equation \cite{causley2017method, causley2016method, kropinski2011fast, jia2008krylov},  Maxwell's equations \cite{cheng2017asymptotic}, the split Vlasov equation \cite{christlieb2016weno}, among many others. This methodology can be generalized to solving  some nonlinear problems, such as the Cahn-Hilliard equation \cite{causley2017method}.  However, it rarely applied to general nonlinear problems, mainly because efficient fast algorithms of inverting nonlinear BVPs are lacking and hence the advantage of the MOL$^T$ is compromised.   

More recently, the authors proposed a novel numerical scheme for solving the nonlinear degenerate parabolic equations \cite{christlieb2017kernel}. Following the MOL$^T$ philosophy,
 %of employing a kernel-based approach, 
 we express the spatial derivatives in terms of a kernel-based representation. 
By construction, the solution is evolved by an explicit SSP RK method, similar to the MOL approach. However, by  carefully choosing a parameter $\beta$ introduced in the formulation, the scheme is proven to be unconditionally stable, and hence allowing for large time step evolution and saving computational cost. This is considered the most remarkable property of the proposed method. The robust WENO methodology and a nonlinear filter are further employed to capture sharp gradients of the solution without producing spurious oscillations, and at the same time, the high order accuracy is attained in smooth regions. The scheme can be easily implemented in a nonuniform but orthogonal mesh and hence very effective in handling complex geometry.

In this paper, we extend the work in \cite{christlieb2017kernel} to solve H-J equations \eqref{eq:HJ}. In particular, under such a framework, the  spatial derivatives are still reconstructed by the kernel-based approach, which is genuinely high order accurate in both space and time and essentially non-oscillatory due to the use of the WENO methodology as well as the nonlinear filter. This is very desired since the viscosity solution of the H-J equations may develop discontinuous derivatives. More importantly, the scheme is unconditionally stable if the parameter $\beta$ is appropriately chosen.  By further incorporating the recently proposed  high order inverse Lax-Wendroff (ILW) boundary treatment technique, the scheme is able to deal with many types of boundary conditions effectively such as non-homogeneous Dirichlet and Neumann boundary conditions.   In summary, the proposed scheme for solving the H-J equations  is robust, high order accurate,  unconditionally stable, flexible and  efficient.

The rest of this paper is organized as follows. In Section 2, we present our numerical scheme for one-dimensional (1D) H-J equations. The two-dimensional (2D) case is considered in Section 3. In Section 4, a collection of numerical examples is presented to demonstrate the performance of the proposed method. In Section 5, we conclude the paper with some remarks and future work.
\section{One-dimensional case}

In the 1D case, \eqref{eq:HJ} becomes
\begin{align}
\label{eq:1D}
\phi_{t}+ H(\phi_{x})=0,\quad \phi(x,0)=\phi^{0}(x).
\end{align}
Assume the domain is a closed interval $[a,b]$ which is partitioned with $N+1$ points
\begin{align}
a=x_{0}<x_{1}<\cdots<x_{N-1} <x_{N} =b,
\end{align}
with $\Delta x_{i}=x_{i}-x_{i-1}$. Note that the mesh can be nonuniform. Let $\phi_{i}(t)$ denote the solution $\phi(x_{i},t)$ at mesh point $x_i$. 

As with many H-J solvers, we will construct the following semi-discrete scheme 
\begin{align}
\label{eq:1Dscheme}
\frac{d}{dt}\phi_{i}(t)+\hat{H}(\phi_{x,i}^{-},\phi_{x,i}^{+}) =0,\quad i=0,\ldots,N,
\end{align}
where $\hat{H}$ is a numerical Hamiltonian. In this work, we use
%, which is a Lipschitz continuous monotone flux consistent with $H$:
%$$\hat{H}(u,u)=H(u).$$
%Monotonicity means that $\hat{H}$ is  non-increasing in its first argument and non-decreasing in the second one. Symbolically, $\hat{H}(\uparrow,\downarrow)$. 
the local Lax-Friedrichs flux
\begin{align}
\label{eq:LLF}
\hat{H}(u^{-},u^{+})=H(\frac{u^{-}+u^{+}}{2}) -\alpha(u^{-},u^{+})\frac{u^{+}-u^{-}}{2},
\end{align}
where $\alpha(u^{-},u^{+})=\max_u |H'(u)|$ with the maximum taken over the range
%$I(u^{-},u^{+})=[\min(u^{-},u^{+}), \max(u^{-},u^{+})].$
bounded by $u^-$ and $u^+$.
$\phi_{x,i}^{-}$ and $\phi_{x,i}^{+}$ in \eqref{eq:1Dscheme} are the approximations to $\phi_{x}$ at $x_{i}$ obtained by left-biased and right-biased methods, respectively, to take into the account the direction of characteristics propagation of the H-J equation. This is the main part of this paper and will be detailed in the following subsections. 

\subsection{Approximation of the first order derivative $\partial_{x}$}
We start with a brief review on construction of the kernel-based representation of $\partial_{x}$ proposed in \cite{christlieb2017kernel}.
First, we introduce two operators $\mathcal{L}_{L}$ and $\mathcal{L}_{R}$ %to account for waves traveling in opposite directions:
\begin{align}
\label{eq:operL}
\mathcal{L}_{L}=\mathcal{I}+\frac{1}{\gamma}\partial_{x}, \quad
\mathcal{L}_{R}=\mathcal{I}-\frac{1}{\gamma}\partial_{x},  \quad x\in[a,b],
\end{align} 
where $\mathcal{I}$ is the identity operator and $\gamma>0$ is a constant. Then we can invert $\mathcal{L}_{L}$ and $\mathcal{L}_{R}$ analytically as follows:
\begin{subequations}
\begin{align}
\label{eq:LLinverse}
& \mathcal{L}_{L}^{-1}[v,\gamma](x)=I^{L}[v,\gamma](x) + A_{L}e^{-\gamma (x-a)},\\
\label{eq:LRinverse}
& \mathcal{L}_{R}^{-1}[v,\gamma](x)=I^{R}[v,\gamma](x) + B_{R} e^{-\gamma (b-x)},
\end{align}
\end{subequations}
where
\begin{subequations}
\begin{align}
\label{eq:IL}
& I^{L}[v,\gamma](x)=\gamma \int_a^x e^{-\gamma (x-y)}v(y)dy,\\
\label{eq:IR}
& I^{R}[v,\gamma](x)=\gamma \int_x^{b} e^{-\gamma (y-x)}v(y)dy,
\end{align}
\end{subequations}
with constant $A_L$ and $B_{R}$ being determined by the boundary condition imposed for the operators. For example, if assume $\mathcal{L}_{L}^{-1}$ and $\mathcal{L}_{R}^{-1}$ are periodic, i.e., 
$$\mathcal{L}_{L}^{-1}[v,\gamma](a)=\mathcal{L}_{L}^{-1}[v,\gamma](b), \quad \text{and} \quad
\mathcal{L}_{R}^{-1}[v,\gamma](a)=\mathcal{L}_{R}^{-1}[v,\gamma](b),$$
then we have
\begin{align}
\label{eq:bc_per}
A_{L}=\frac{I^{L}[v,\gamma](b)}{1-\mu}, \quad \text{and} \quad B_{R}=\frac{I^{R}[v,\gamma](a)}{1-\mu},
\end{align} 
with $\mu=e^{-\gamma(b-a)}$. 
%If we want the operators to satisfy boundary conditions
%$$\mathcal{L}_{L}^{-1}[v,\gamma](a)=C_{a}, \quad \text{and} \quad
%\mathcal{L}_{R}^{-1}[v,\gamma](b)=C_{b},$$
%with given numerbers $C_{a}$ and $C_{b}$, then the coefficients are given as
%\begin{align}
%\label{eq:bc_dir}
%A_{L}=C_{a}, \quad \text{and} \quad B_{R}=C_{b}.
%\end{align} 

In addition, it is straightforward to show that
\begin{subequations}
	\label{eq:sum}
	\begin{align}
	\label{eq:LL}
	& \frac{1}{\gamma}\partial_{x}=\mathcal{L}_{L}-\mathcal{I} 
	=\mathcal{L}_{L} (\mathcal{I}-\mathcal{L}^{-1}_{L})
	=\mathcal{D}_{L}/(\mathcal{I} -\mathcal{D}_{L} )
	=\sum_{p=1}^{\infty}\mathcal{D}_{L}^{p}.\\
	\label{eq:LR}
	& \frac{1}{\gamma}\partial_{x}=\mathcal{I} - \mathcal{L}_{R}
	=\mathcal{L}_{R} (\mathcal{L}^{-1}_{R}-\mathcal{I})
	=-\mathcal{D}_{R}/(\mathcal{I} -\mathcal{D}_{R} )
	=-\sum_{p=1}^{\infty}\mathcal{D}_{R}^{p},
	\end{align}
\end{subequations}
where the two operators $\mathcal{D}_{L}$ and $\mathcal{D}_{R}$ are defined by
\begin{align}	
\label{eq:operD}
\mathcal{D}_{L}=\mathcal{I}-\mathcal{L}^{-1}_{L}, \quad
\mathcal{D}_{R}=\mathcal{I}-\mathcal{L}^{-1}_{R}, \quad x\in[a,b].
\end{align} 

Hence, following the idea in \cite{christlieb2017kernel}, when $\phi$ is a periodic function, we can approximate the first derivative $\phi^{\pm}_{x}$ with (modified) partial sums in \eqref{eq:sum}, 
\begin{subequations}
	\label{eq:partialsum_per}
\begin{align}
	\phi_{x}^{-}(x)\approx \mP^{L}_{k}[\phi,\gamma](x) = \left\{\begin{array}{ll}
	\gamma\sum\limits_{p=1}^{k}\mD_{L}^{p}[\phi,\gamma](x), & k=1, 2,\\
	\gamma\sum\limits_{p=1}^{k}\mD_{L}^{p}[\phi,\gamma](x) - \gamma\mD_{0}*\mD_{L}^2[\phi,\gamma](x), & k=3,\\
	\end{array}
	\right.
\end{align}
\begin{align}
	\phi_{x}^{+}(x)\approx \mP^{R}_{k}[\phi,\gamma](x)  = \left\{\begin{array}{ll}
	-\gamma\sum\limits_{p=1}^{k}\mD_{R}^{p}[\phi,\gamma](x), & k=1, 2,\\
	-\gamma\sum\limits_{p=1}^{k}\mD_{R}^{p}[\phi,\gamma](x) + \gamma\mD_{0}*\mD_{R}^2[\phi,\gamma](x), & k=3.\\
	\end{array}
	\right.
\end{align}
\end{subequations}
Note that by construction, $\mD_{L}^{p}[\phi,\gamma](x)$ only depends on the values of $\phi$ on $[a,x]$, while $\mD_{R}^{p}[\phi,\gamma](x)$ only depends on the values of $\phi$ on $[x,b]$. This design in fact takes into account the direction of propagating characteristics of the H-J equation. Also note that there is an extra term for $k=3$. As remarked in  \cite{christlieb2017kernel}, such a term is needed for 
attaining unconditional stability of the scheme.
$\mD_{0}$ is defined as
\begin{align}
\label{eq:D0}
\mD_{0}[v,\alpha](x) = v(x) -\frac{\gamma}{2}\int_{a}^{b}e^{-\gamma|x-y|}v(y)dy -A_{0}e^{-\gamma(x-a)}-B_{0}e^{-\gamma(b-x)}.
\end{align}
The coefficients $A_{0}$ and $B_{0}$ are also obtained from the boundary condition. For instance, if we require $\mD_{0}[v,\alpha](x)$ to be a periodic function, i.e.,
$$\mD_{0}[v,\gamma](a)=\mD_{0}[v,\gamma](b), \quad\text{and}\quad \partial_{x}\mD_{0}[v,\gamma](a)=\partial_{x}\mD_{0}[v,\gamma](b),$$
then we have
\begin{align}
	\label{eq:gxx_bc_per}
A_{0} = \frac{I^{0}[v,\gamma](b)}{1-\mu}, \quad
B_{0} = \frac{I_{0}[v,\gamma](a)}{1-\mu}
\end{align}
with $I^{0}[v,\gamma](x)=\gamma/2\int_{a}^{b}e^{-\gamma|x-y|}v(y)dy$. 
%If we want 
%$$\mD_{0}[v,\gamma](a)=C_{a}, \quad\text{and}\quad \mD_{0}[v,\gamma](b)=C_{b},$$
%for some given number $C_{a}$ and $C_{b}$, then, the coefficients are given as
%\begin{subequations}
%	\label{eq:gxx_bc_dir}
%	\begin{align}
%	& A_{0}=\frac{1}{1-\mu^2}\left( \mu\left(I^{0}[v,\gamma](b)-v(b)+C_{b}\right) - \left(I^{0}[v,\gamma](a)-v(a)+C_{a}\right)\right), \\
%	& B_{0}=\frac{1}{1-\mu^2}\left( \mu\left(I^{0}[v,\gamma](a)-v(a)+C_{a}\right) - \left(I^{0}[v,\gamma](b)-v(b)+C_{b}\right)\right).
%	\end{align}
%\end{subequations}

It is natural to require the following boundary conditions for the operators
\begin{align}
\mD^{p}_{L}(a)=\mD_{L}^{p}(b), \quad \mD^{p}_{L}(a)=\mD_{L}^{p}(b), \quad \mD_{0}(a)=\mD_{0}(b),
\end{align}
 where $p\geq 1$, if periodic boundary conditions of the solution are imposed. 
Moreover, an error estimate for the partial sums approximation \eqref{eq:partialsum_per} carried out in \cite{christlieb2017kernel}  shows that
\begin{subequations}
	\begin{align}
	& \|\partial_{x}\phi(x)-\mP^{L}_{k}[\phi,\gamma](x) \|_{\infty}\leq C \left(\frac{1}{\gamma}\right)^{k} \|\partial_{x}^{k+1}\phi(x)\|_{\infty},\\
	& \|\partial_{x}\phi(x)-\mP^{R}_{k}[\phi,\gamma](x) \|_{\infty}\leq C \left(\frac{1}{\gamma}\right)^{k} \|\partial_{x}^{k+1}\phi(x)\|_{\infty}.
	\end{align}
\end{subequations}

In numerical simulations, we will take $\gamma=\beta/(\alpha\Delta t)$ in \eqref{eq:partialsum_per}, with $\alpha$ being the maximum wave propagation speed. Here, $\Delta t$ denotes the time step and $\beta$ is a constant independent of $\Delta t$. Hence, the partial sums approximate $\phi_{x}$ with accuracy $\mathcal{O}(\Delta t^k)$.

For time integration, we propose to use the classical explicit SSP RK methods \cite{gottlieb2001strong} to advance the solution from time $t^{n}$ to $t^{n+1}$. We denote $\phi^{n}$ as the semi-discrete solution at time $t^{n}$. In this paper, we use the following SSP RK methods, including the first order forward Euler scheme
\begin{align}
\label{eq:rk1}
\phi^{n+1}=\phi^{n}-\Delta t \hat{H}(\phi^{n,-}_{x},\phi^{n,+}_{x});
\end{align} 
the second order SSP RK scheme
\begin{align}
\label{eq:rk2}
& \phi^{(1)}=\phi^{n}-\Delta t \hat{H}(\phi^{n,-}_{x},\phi^{n,+}_{x}),\nonumber\\
& \phi^{n+1}=\frac{1}{2}\phi^{n}+\frac{1}{2}\left( \phi^{(1)} -\Delta t \hat{H}(\phi^{(1),-}_{x},\phi^{(1),+}_{x}) \right);
\end{align}
and the third order SSP RK scheme
\begin{align}
\label{eq:rk3}
& \phi^{(1)}=u^{n}-\Delta t \hat{H}(\phi^{n,-}_{x},\phi^{n,+}_{x}),\nonumber\\
& \phi^{(2)}=\frac{3}{4}\phi^{n}+\frac{1}{4} \left( \phi^{(1)}-\Delta t \hat{H}(\phi^{(1),-}_{x},\phi^{(1),+}_{x}) \right), \nonumber\\
& \phi^{n+1}=\frac{1}{3}\phi^{n}+\frac{2}{3} \left( \phi^{(2)}-\Delta t \hat{H}(\phi^{(2),-}_{x},\phi^{(2),+}_{x}) \right).
\end{align}
In addition, linear stability of the proposed kernel-based schemes has been established in \cite{christlieb2017kernel}. In particular, we proved the following theorem.
\begin{thm}\label{thm4}
For the linear equation $\phi_{t}+c\phi_{x}=0$, (i.e. the Hamiltonian is linear) with periodic boundary conditions, we consider the $k^{th}$ order SSP RK method as well as the  $k^{th}$ partial sum in \eqref{eq:partialsum_per}, with $\gamma=\beta/(|c|\Delta t)$. Then there exists a constant $\beta_{k,max}>0$ for $k=1,\,2,\,3$, such that the scheme is unconditionally stable provided  $0<\beta\leq\beta_{k,\max}$. The constants $\beta_{k,max}$ for $k=1,\,2,\,3$ are summarized in Table \ref{tab0}.
\end{thm}

\begin{table}[htb]
	\caption{\label{tab0}\em $\beta_{k,\max}$ in Theorem \ref{thm4} for  $k=1,\,2,\,3$.}
	\centering
	\vspace{0.3cm}
	\begin{tabular}{| l | p{1cm} | p{1cm} | p{1cm} |}
		\hline
		$k$ &  1  & 2  & 3  \\\hline
		$\beta_{k,max}$  &  2  &  1  &  1.243  \\\hline
	\end{tabular}
\end{table}

\subsection{Non-periodic boundary conditions}
\label{sec:non_periodic}

In this subsection, we will focus on the 1D H-J equation \eqref{eq:1D} with two commonly considered non-periodic boundary conditions, including Dirichlet and  Neumann boundary conditions.
Unlike periodic boundary conditions, it is not straightforward to specify the boundary conditions  for the operators $\mathcal{D}_{*}^{p}$, where $*$ could be $L$, $R$ or $0$. The reason is that boundary conditions imposed on the operators $\mathcal{D}_{*}$ must be consistent with the boundary condition specified on $\phi$ (i.e. Dirichlet boundary condition) or its derivative (i.e. Neumann boundary condition). To address the difficulty, we propose to modify the partial sum \eqref{eq:partialsum_per}.
We will first investigate the operators $\mD_{*}$ and prove a lemma in Subsection \ref{sec:operator}, which is crucial for the scheme development as well as the error analysis.
In Subsection \ref{sec:modify}, under the assumption that we already have all boundary values as needed, e.g., the derivatives of $\phi$ at the boundary, we show that how to specify proper boundary conditions for the operators $\mathcal{D}_{*}$, i.e., the coefficients  $A_{0}$, $B_{0}$, $A_{L}$ and $B_{R}$, such that the partial sum can approximate $\phi_x$ with high order accuracy and the Dirichlet or Neumann boundary condition is satisfied. In the Subsection \ref{sec:ILW}, we will provide a systematic approach to reconstructing those needed boundary values.

\subsubsection{Operators with non-periodic boundary treatment}
\label{sec:operator}

We first study the operator $\mD_{*}$ with a non-periodic boundary condition, where $*$ can be $0$, $L$ and $R$. Suppose $C_{a}$ and $C_{b}$ are given numbers.
\begin{itemize}
	\item If we require
	$$\mD_{L}[v,\gamma](a)=C_{a}, \quad\text{and}\quad \mD_{R}[v,\gamma](b)=C_{b},$$
	then, the coefficients are given as
	\begin{align}
	\label{eq:fx_bc_dir}
	A_{L}=v(a) - C_{a} , \quad \text{and} \quad B_{R}= v(b) - C_{b}.
	\end{align}
	\item If we require
	$$\mD_{0}[v,\gamma](a)=C_{a}, \quad\text{and}\quad \mD_{0}[v,\gamma](b)=C_{b},$$
	then, the coefficients are given as
	\begin{subequations}
		\label{eq:gxx_bc_dir}
		\begin{align}
		& A_{0}=\frac{1}{1-\mu^2}\left( \mu\left(I^{0}[v,\gamma](b)-v(b)+C_{b}\right) - \left(I^{0}[v,\gamma](a)-v(a)+C_{a}\right)\right), \\
		& B_{0}=\frac{1}{1-\mu^2}\left( \mu\left(I^{0}[v,\gamma](a)-v(a)+C_{a}\right) - \left(I^{0}[v,\gamma](b)-v(b)+C_{b}\right)\right).
		\end{align}
	\end{subequations}
\end{itemize}

We start with proving the following lemma, which connects the function $\mD_{*}[v,\gamma](x)$ with the derivatives of the function $v(x)$ as well as their boundary values.

\begin{lem}\label{lem1}
	Consider the boundary treatments \eqref{eq:fx_bc_dir} and \eqref{eq:gxx_bc_dir}.
	\begin{itemize}
		\item Suppose $v\in \mathcal{C}^{k+1}[a,b]$. 
%		If we set the operator $\mD_{L}$ and $\mD_{R}$ with boundary treatments
%		$\mD_{L}[v,\gamma](a) = C_{a}$, and $\mD_{R}[v,\gamma](b) = C_{b}$. 
		Then, we can obtain that
		\begin{subequations}
			\begin{align}
			\mathcal{D}_{L}[v,\gamma](x)
			=& -\sum_{p=1}^{k} \left(-\frac{1}{\gamma}\right)^{p}\left(\partial_{x}^{p}v(x)-\partial_{x}^{p}v(a)e^{-\gamma(x-a)}\right)  - \left(-\frac{1}{\gamma}\right)^{k+1} I^{L}[\partial_{x}^{k+1}v,\gamma](x) + C_{a}e^{-\gamma(x-a)},\\
			\mathcal{D}_{R}[v,\gamma](x)
			=& -\sum_{p=1}^{k} \left(\frac{1}{\gamma}\right)^{p}\left(\partial_{x}^{p}v(x)-\partial_{x}^{p}v(b)e^{-\gamma(b-x)}\right)  -\left(\frac{1}{\gamma}\right)^{k+1}I^{R}[\partial_{x}^{k+1}v,\gamma](x) + C_{b}e^{-\gamma(b-x)}.
			\end{align}
		\end{subequations}
		
		\item Suppose $v\in \mathcal{C}^{2k+2}[a,b]$. 
%		If we set $\mD_{0}$ with treatment $\mD_{0}[v,\gamma](a)=C_{a}$, and $\mD_{0}[v,\gamma](b)=C_{b}$.
		Then, we have
		\begin{align}
		\mathcal{D}_{0}[v,\gamma](x)
		=& -\sum_{p=1}^{k} \left(\frac{1}{\gamma}\right)^{2p} \left( \partial_{x}^{2p}v(x) +\frac{\mu \partial_{x}^{2p}v(b) - \partial_{x}^{2p}v(a)}{1-\mu^2} e^{-\gamma(x-a)}  +\frac{\mu \partial_{x}^{2p}v(a) - \partial_{x}^{2p}v(b)}{1-\mu^2} e^{-\gamma(b-x)} \right) \nonumber\\
		& -\frac{\mu C_{b} - C_{a}}{1-\mu^2} e^{-\gamma(x-a)} -\frac{\mu C_{a} - C_{b}}{1-\mu^2} e^{-\gamma(b-x)} 
		- \left(\frac{1}{\gamma}\right)^{2k+2} \left( I^{0}[\partial_{x}^{2k+2}v,\gamma](x) \right. \nonumber\\
		& \left. 
		+ \frac{\mu I^{0}[\partial_{x}^{2k+2}v,\gamma](b) - I^{0}[\partial_{x}^{2k+2}v,\gamma](a)}{1-\mu^2} e^{-\gamma(x-a)}  
		+ \frac{\mu I^{0}[\partial_{x}^{2k+2}v,\gamma](a) - I^{0}[\partial_{x}^{2k+2}v,\gamma](b)}{1-\mu^2} e^{-\gamma(b-x)} \right).
		\end{align}		
	\end{itemize}
\end{lem}

\begin{proof}
	For brevity, we only show the details of the proof for operator $\mD_{L}$. Following a similar argument, one can easily prove the other cases.
	
	Using integration by parts repeatedly, we have
	\begin{align*}
	I^{L}[v,\gamma](x)
	%=& v(y)e^{-\gamma(x-y)}|_{y=a}^{y=x}-\int_{a}^{x}e^{-\gamma(x-y)}v'(y)dy\\
	=& v(x)-v(a)e^{-\gamma(x-a)} - \frac{1}{\gamma}I^{L}[v'(y),\gamma](x) = \cdots \\
	=& \sum_{p=0}^{k}\left(-\frac{1}{\gamma}\right)^{p} \left( \partial_{x}^{p}v(x)-\partial_{x}^{p}v(a)e^{-\gamma(x-a)}\right) +\left(-\frac{1}{\gamma}\right)^{k+1} I^{L}[\partial_{x}^{k+1}v(y),\gamma](x).
	\end{align*}	
	Therefore, with the coefficient $A_{L}=v(a) - C_{a}$, $\mD_{L}[v,\gamma](x)$ can be rewritten as 
	\begin{align*}
	\mD_{L}[v,\gamma](x) 
	=&  \sum_{p=1}^{k}\left(-\frac{1}{\gamma}\right)^{p} \left( \partial_{x}^{p}v(x)-\partial_{x}^{p}v(a)e^{-\gamma(x-a)}\right) +\left(-\frac{1}{\gamma}\right)^{k+1} I^{L}[\partial_{x}^{k+1}v(y),\gamma](x) 
	+ C_{a} e^{-\gamma(x-a)}.
	\end{align*}
\end{proof}

In light of Lemma \ref{lem1}, it is easy to check that given the boundary values
$\phi_{x}(a)$ and $\phi_{x}(b)$ from the boundary condition,
 $\phi^{\pm}_{x}$ can be approximated by
\begin{align}
\phi^{-}_{x}\approx \gamma \mD_{L}[\phi,\gamma](x) , \quad
\phi^{+}_{x}\approx-\gamma \mD_{R}[\phi,\gamma](x) ,
\end{align}
with order $\mO(1/\gamma)$ (i.e., first order accuracy since we choose $\gamma=\mathcal{O}(1/\Delta t)$), by requiring
\begin{align}
\label{eq:bc_test}
\gamma \mD_{L}(a) = \phi_{x}(a), \quad \text{and} \quad 
-\gamma \mD_{R}(a) = \phi_{x}(b).
\end{align}

It seems that for the higher order scheme given in \eqref{eq:partialsum_per}, the boundary condition for $\phi$ can be satisfied based on \eqref{eq:bc_test} together with $\gamma\mathcal{D}^{p}_{L}[\phi,\gamma](a)=\gamma\mathcal{D}_{R}^{p}[\phi,\gamma](b)= 0$, for $p\geq2$ and $\gamma\mD_{0}*\mathcal{D}^{2}_{L}[\phi,\gamma](a) = \gamma\mD_{0}*\mathcal{D}_{R}^{2}[\phi,\gamma](b)=0$ if $k=3$. However, due to the extra boundary terms appear in Lemma \ref{lem1}, one can check that the scheme \eqref{eq:partialsum_per} is not high order accurate. In fact, it is only first order in the case of non-periodic boundary conditions. To circumvent this problem, we will use a modified scheme instead of \eqref{eq:partialsum_per}.

\subsubsection{The modified partial sum}
\label{sec:modify}

Suppose we have obtained by some means the derivative values at boundary, i.e., $\partial_{x}^{m}\phi(a)$ and $\partial_{x}^{m}\phi(b)$, $m\geq1$.
 %In Subsection \ref{sec:ILW}, we will introduce an approach that systematically generate these boundary derivative values. 
 We propose the following modified partial sums for $k\leq3$ to handle non-periodic boundary conditions 
\begin{subequations}
	\label{eq:partialsum_dir}
	\begin{align}
	\phi_{x}^{-}(x)\approx\widetilde{\mP}^{L}_{k}[\phi,\gamma](x)=\left\{\begin{array}{ll}
	\gamma\sum\limits_{p=1}^{k}\mathcal{D}_{L}[\phi_{1,p},\gamma](x), & k=1,\, 2,\\
	\gamma\sum\limits_{p=1}^{k}\mathcal{D}_{L}[\phi_{1,p},\gamma](x) -\gamma \mD_{0}[\phi_{1,3},\gamma](x), & k=3,\\
	\end{array}
	\right.
	\end{align}
	\begin{align}
	\phi_{x}^{+}(x)\approx\widetilde{\mP}^{R}_{k}[\phi,\gamma](x)=\left\{\begin{array}{ll} -\gamma\sum\limits_{p=1}^{k}\mathcal{D}_{R}[\phi_{2,p},\gamma](x),& k=1,\, 2,\\
	-\gamma\sum\limits_{p=1}^{k}\mathcal{D}_{R}[\phi_{2,p},\gamma](x) +\gamma \mD_{0}[\phi_{2,3},\gamma](x), & k=3.\\
	\end{array}
	\right.
	\end{align}
\end{subequations}
And $\phi_{1,p}$ and $\phi_{2,p}$ are given as
\begin{subequations}
	\label{eq:expression}
	\begin{align}
	& \left\{\begin{array}{ll}
	\phi_{1,1}=\phi,\\
	\displaystyle \phi_{1,2}=\mathcal{D}_{L}[\phi_{1,1},\gamma] - \sum_{m=2}^{k}\left(-\frac{1}{\gamma}\right)^{m} \partial_{x}^{m}\phi(a) e^{-\gamma(x-a)},\\
	\displaystyle \phi_{1,3}=\mathcal{D}_{L}[\phi_{1,2},\gamma] + \sum_{m=2}^{k}(m-1)\left(-\frac{1}{\gamma}\right)^{m} \partial_{x}^{m}\phi(a) e^{-\gamma(x-a)},\\
	\end{array}
	\right.\\
	& \left\{\begin{array}{ll}
	\phi_{2,1}=\phi, \\
	\displaystyle \phi_{2,2}=\mathcal{D}_{R}[\phi_{2,1},\gamma] - \sum_{m=2}^{k}\left(\frac{1}{\gamma}\right)^{m} \partial_{x}^{m}\phi(b) e^{-\gamma(b-x)}, \\
	\displaystyle \phi_{2,3}=\mathcal{D}_{R}[\phi_{2,2},\gamma] + \sum_{m=2}^{k}(m-1)\left(\frac{1}{\gamma}\right)^{m} \partial_{x}^{m}\phi(b) e^{-\gamma(b-x)},\\
	\end{array}
	\right.
	\end{align}
\end{subequations}
with the boundary conditions for the operators
	\begin{align*}
	& \gamma\mathcal{D}_{L}[\phi_{1,1},\gamma](a)=\phi_{x}(a), \quad
	\gamma\mathcal{D}_{R}[\phi_{2,1},\gamma](b)=-\phi_{x}(b),\\
	& \gamma\mathcal{D}_{L}[\phi_{1,p},\gamma](a)=\gamma\mathcal{D}_{R}[\phi_{2,p},\gamma](b)= 0,
	\quad \text{for} \ p\geq2,\\
	& \gamma\mD_{0}[\phi_{*,3},\gamma](a) = \gamma\mD_{0}[\phi_{*,3},\gamma](b)=0, \quad \text{$*$ could be 1 or 2.}
	\end{align*} 
One can easily check that the modified partial sum \eqref{eq:partialsum_dir} agrees with the derivative values at the boundary, i.e.,
$$
\widetilde{\mP}^{L}_{k}[\phi,\gamma](a) = \phi_x(a),\quad \widetilde{\mP}^{R}_{k}[\phi,\gamma](b) = \phi_x(b),
$$
meaning that the modified partial sum approximation \eqref{eq:partialsum_dir} is consistent with the boundary condition imposed on $\phi$.
 Moreover, we have the following  theorem:
\begin{thm}
	Suppose $\phi\in\mathcal{C}^{k+1}[a,b]$, $k=1, \, 2,\, 3$. Then, the modified partial sums \eqref{eq:partialsum_dir} satisfy
	\begin{subequations}
		\begin{align}
		& \|\partial_{x}\phi(x)-\widetilde{\mP}^{L}_{k}[\phi,\gamma](x)\|_{\infty}\leq C \left(\frac{1}{\gamma}\right)^{k} \|\partial_{x}^{k+1}\phi(x)\|_{\infty},\\
		& \|\partial_{x}\phi(x)-\widetilde{\mP}^{R}_{k}[\phi,\gamma](x)\|_{\infty}\leq C \left(\frac{1}{\gamma}\right)^{k} \|\partial_{x}^{k+1}\phi(x)\|_{\infty}.
		\end{align}
	\end{subequations}
\end{thm}

\begin{proof}
	With the help of Lemma \ref{lem1}, we have 
	\begin{align*}
	\mD_{L}[\phi_{1,1},\gamma](x)=\frac{1}{\gamma} \phi_{x} -\sum_{p=2}^{k} \left(-\frac{1}{\gamma}\right)^{p}\left(\partial_{x}^{p}\phi(x)-\partial_{x}^{p}\phi(a)e^{-\gamma(x-a)}\right)  - \left(-\frac{1}{\gamma}\right)^{k+1} I^{L}[\partial_{x}^{k+1}\phi,\gamma](x).
	\end{align*}
	Thus,
	\begin{align*}
	\phi_{1,2}(x) =  -\sum_{p=1}^{k} \left(-\frac{1}{\gamma}\right)^{p} \partial_{x}^{p}\phi(x)   - \left(-\frac{1}{\gamma}\right)^{k+1} I^{L}[\partial_{x}^{k+1}\phi,\gamma](x).
	\end{align*}
	Using the fact the $\partial_{x} I^{L}[v,\gamma](x)=\gamma v(x)-\gamma I^{L}[v,\gamma](x)$, we have the general form for $0\leq q\leq k$
	\begin{align*}
	\partial_{x}^{q}I^{L}[v,\gamma](x) 
	=& -\sum_{p=0}^{q-1}(-\gamma)^{q-p}\partial_{x}^p v(x)  + (-\gamma)^{q} I^{L}[v,\gamma](x).
%	=& -\sum_{p=0}^{q-1}(-\gamma)^{q-p}\partial_{x}^p v(x)  + (-\gamma)^{q} \left(\sum_{p=0}^{k}\left(-\frac{1}{\gamma}\right)^{p} \left( \partial_{x}^{p}v(x)-\partial_{x}^{p}v(a)e^{-\gamma(x-a)}\right) \right.\\
%	& \left. +\left(-\frac{1}{\gamma}\right)^{k+1} I^{L}[\partial_{x}^{k+1}v(y),\gamma](x)\right) \\
%	=& -\sum_{p=q}^{k}\left(-\frac{1}{\gamma}\right)^{p-q}\partial_{x}^p v(x) - \sum_{p=0}^{k}\left(-\frac{1}{\gamma}\right)^{p-q} \partial_{x}^{p}v(a)e^{-\gamma(x-a)}\\ 
%	& +\left(-\frac{1}{\gamma}\right)^{k+1-q} I^{L}[\partial_{x}^{k+1}v(y),\gamma](x)
		\end{align*}
	Thus, we have 	
	\begin{align*}
	\left(-\frac{1}{\gamma}\right)^{q}\partial_{x}^{q}\phi_{1,2}(x) 
	=& \left(-\frac{1}{\gamma}\right)^{q} \partial_{x}^{q} \left(  -\sum_{p=1}^{k} \left(-\frac{1}{\gamma}\right)^{p} \partial_{x}^{p}\phi(x) -\left(-\frac{1}{\gamma}\right)^{k+1} I^{L}[\partial_{x}^{k+1}\phi,\gamma](x)\right) \\
	=& -\sum_{p=q+1}^{q+k} \left(-\frac{1}{\gamma}\right)^{p} \partial_{x}^{p}(x)
	-\left(-\frac{1}{\gamma}\right)^{k+1+q} \left(  -\sum_{p=0}^{q-1}(-\gamma)^{q-p}\partial_{x}^p \partial_{x}^{k+1}\phi(x)  + (-\gamma)^{q} I^{L}[\partial_{x}^{k+1}\phi,\gamma](x) \right)\\
	=& -\sum_{p=q+1}^{q+k} \left(-\frac{1}{\gamma}\right)^{p} \partial_{x}^{p}(x)
	 +\sum_{p=k+1}^{k+q}\left(-\frac{1}{\gamma}\right)^{p}  \partial_{x}^{p}\phi(x)  -\left(-\frac{1}{\gamma}\right)^{k+1}  I^{L}[\partial_{x}^{k+1}\phi,\gamma](x) \\
	 =& -\sum_{p=q+1}^{k} \left(-\frac{1}{\gamma}\right)^{p} \partial_{x}^{p}(x) -\left(-\frac{1}{\gamma}\right)^{k+1}  I^{L}[\partial_{x}^{k+1}\phi,\gamma](x).
	\end{align*}
	Furthermore, we apply $\mD_{L}$ on $\phi_{1,2}$ with $\mD_{L}(a)=0$, and obtain
	\begin{align*}
	\mD_{L}[\phi_{1,2},\gamma] (x)
	=& -\sum_{q=1}^{k-1} \left(-\frac{1}{\gamma}\right)^{q}\left(\partial_{x}^{q}\phi_{1,2}(x)-\partial_{x}^{q}\phi_{1,2}(a)e^{-\gamma(x-a)}\right)  -\left(-\frac{1}{\gamma}\right)^{k} I^{L}[\partial_{x}^{k}\phi_{1,2},\gamma](x) \\
	=& -\sum_{q=1}^{k-1} \left( -\sum_{p=q+1}^{k} \left(-\frac{1}{\gamma}\right)^{p} \partial_{x}^{p}\phi(x) -\left(-\frac{1}{\gamma}\right)^{k+1}  I^{L}[\partial_{x}^{k+1}\phi,\gamma](x) + \sum_{p=q+1}^{k} \left(-\frac{1}{\gamma}\right)^{p} \partial_{x}^{p}\phi(a) e^{-\gamma(x-a)} \right) \\
	& +  \left(-\frac{1}{\gamma}\right)^{k+1} \left(I^{L}\right)^2 [\partial_{x}^{k+1}\phi,\gamma](x)\\
	=& 	\sum_{p=2}^{k} (p-1)\left(-\frac{1}{\gamma}\right)^{p} \left(\partial_{x}^{p}\phi(x)  - \partial_{x}^{p}\phi(a) e^{-\gamma(x-a)} \right) + (k-2) \left(-\frac{1}{\gamma}\right)^{k+1}  I^{L}[\partial_{x}^{k+1}\phi,\gamma](x) \\
	& +  \left(-\frac{1}{\gamma}\right)^{k+1} \left(I^{L}\right)^2 [\partial_{x}^{k+1}\phi,\gamma](x).
	\end{align*}
	Therefore,
	\begin{align*}
	 \mD_{L}[\phi_{1,1},\gamma] (x) + \mD_{L}[\phi_{1,2},\gamma] (x)
	=& \frac{1}{\gamma} \phi_{x}(x) + \sum_{p=3}^{k} (p-2)\left(-\frac{1}{\gamma}\right)^{p} \left(\partial_{x}^{p}\phi(x)  - \partial_{x}^{p}\phi(a) e^{-\gamma(x-a)} \right) \\
	& + (k-3) \left(-\frac{1}{\gamma}\right)^{k+1}  I^{L}[\partial_{x}^{k+1}\phi,\gamma](x) 
	+  \left(-\frac{1}{\gamma}\right)^{k+1} \left(I^{L}\right)^2 [\partial_{x}^{k+1}\phi,\gamma](x),
	\end{align*}
	which proves the case of $k=2$ for the approximation of $\phi^{-}_{x}$.
	
	Following a similar argument, one can easily prove the case of $k=3$. With little modification, the proof holds for $\phi^{+}_{x}$ as well.
\end{proof}

\subsubsection{Boundary values reconstruction}\label{sec:ILW}

We have shown that the derivatives $\partial_{x}^{m}\phi(a)$ and $\partial_{x}^{m}\phi(b)$, $m\geq1$, are needed in \eqref{eq:expression} when dealing with the non-periodic boundary conditions. For an outflow boundary condition in the sense of characteristics propagation, it is natural to employ high order extrapolation to obtain the boundary values. For an inflow boundary condition, it becomes more complicated. In \cite{huang2008numerical,xiong2010fast}, the ILW boundary treatment methodology was designed to reconstruct accurate solution values at the ghost points near the boundary for a class of fast sweeping finite difference WENO schemes that solves  the static H-J equations. Over the past several years, ILW has experienced systematic development in various settings, see \cite{tan2010inverse,tan2012efficient}.
Here, we will apply this methodology to obtain the boundary values of solution derivatives at an inflow boundary. The main procedure of ILW  is to convert the spatial derivatives into the time derivatives at an inflow boundary by repeatedly utilizing the underlying PDE. 

To illustrate the main idea, we consider the left boundary $x = a$ for example:
\begin{itemize}
	\item If $x=a$ is an \textbf{outflow} boundary of the domain, where no physical boundary condition should be given, then we  obtain all the derivatives $\partial_{x}^{m}\phi(a)$, $m\geq 1$, through  extrapolation with a suitable order of accuracy.

	\item If $x=a$ is an \textbf{inflow} boundary of the domain, with a \textbf{Dirichlet} boundary condition $\phi(a,t)=f(t)$. Then, the following ILW procedure is applied:
	\begin{enumerate}
		\item To obtain $\phi_{x}(a)$, we first note that
		$$\phi_t(a,t)=f'(t).$$
	    We evaluate \eqref{eq:1D} at $x=a$ and obtain 
		\begin{align}
		\label{eq:ILW1}
		f'(t)+H(\phi_{x}(a))=0.
		\end{align} 
		Then $\phi_{x}(a)$ can be solved from \eqref{eq:ILW1}. Note that there might be more than one root. In this case we should choose the root $\phi_x(a)$ so that
		\begin{align}
		\label{eq:ILW_condition}
		H'(\phi_{x}(a))>0,
		\end{align}
		which guarantees that the boundary $x=a$ is indeed an inflow boundary. Moreover, if the condition \eqref{eq:ILW_condition} still cannot pin down a root, then we would choose the root which is closest to the value obtained through extrapolation.
		
		\item To obtain $\phi_{xx}(a)$, we first take the derivative with respect to $t$ in the original H-J equation \eqref{eq:1D}, yielding
		\begin{align}
		\label{eq:ILW2}
		\phi_{tt}=-H(\phi_{x})_{t}
		=-H'(\phi_{x})\phi_{xt}
		=-H'(\phi_{x}) \left(-H(\phi_{x})\right)_{x}=\left(H'(\phi_{x})\right)^{2} \phi_{xx}.
		\end{align}
		Note that at $x=a$, $\phi_{tt}=f''(t)$ and $\phi_{x}(a)$ is available from the last step, we then solve this equation  for $\phi_{xx}(a)$; that is
		\begin{equation}
		\phi_{xx}(a) = \frac{f''(t)}{ H'(\phi_{x}(a))^{2}}.
		\end{equation}
		
		\item To obtain $\phi_{xxx}(a)$, we further differentiate \eqref{eq:ILW2} with respect to $t$ and make use of the original H-J equation \eqref{eq:1D} again, we then have
		\begin{align}
		\label{eq:ILW3}
		\phi_{ttt}=-3H''(\phi_{x}) \left(H'(\phi_{x})\right)^{2} \phi_{xx}^{2} -\left(H'(\phi_{x})\right)^{3} \phi_{xxx}.
		\end{align}
		 At $x=a$, note that $\phi_{ttt}=f'''(t)$, then $\phi_{xxx}(a)$ is the only unknown quantity and can be easily solved
		 \begin{equation}
		 \phi_{xxx}(a) =- \frac{f'''(t)+3H''(\phi_{x}(a)) \left(H'(\phi_{x}(a))\right)^{2} \phi_{xx}(a)^{2}}{H'(\phi_{x}(a))^{3}}.
		 \end{equation}
		
		\item Repeating this procedure, we can compute any order derivative at $x=a$ as needed. 
	\end{enumerate}

	\item If $x=a$ is an \textbf{inflow} boundary of the domain with a \textbf{Neumann} boundary condition  $\phi_{x}(a,t)=g(t)$, the same ILW procedure can be applied, except that we start from solving the second derivative $\phi_{xx}(a)$.
\end{itemize}

Similarly, the ILW procedure for the right boundary $x=b$ can be developed. Note that in this case, the inflow condition \eqref{eq:ILW_condition} should be changed into 
\begin{align}
H'(\phi_{x}(b))<0.
\end{align}

\subsection{Space Discretization }
In this subsection, we present the details about the spatial discretization of $\mathcal{D}_{L}$ and $\mathcal{D}_{R}$.
We denote $L^{*}[v,\gamma](x_{i})$ as $L^{*}_{i}$ on each grid point $x_{i}$, where $*$ can be $L$ and $R$. Note that 
\begin{subequations}
	\label{eq:recursive}
	\begin{align}
	& I^L_i = e^{-\gamma\Delta x_{i}} I^L_{i-1} + J^L_i,\quad i=1,\ldots,N, \quad I^L_0 = 0, \\
	& I^R_i = e^{-\gamma\Delta x_{i+1}} I^R_{i+1} + J^R_i,\quad i=0,\ldots,N-1, \quad I^R_N = 0,
	\end{align}
\end{subequations}
where,
\begin{align}
\label{eq:JLR}
J^L_{i} =  \gamma \int_{x_{i-1}}^{x_{i}} v(y)e^{-\gamma (x_{i}-y)}dy,\ \ \ \
J^R_{i} =  \gamma \int_{x_{i}}^{x_{i+1}} v(y)e^{-\gamma (y-x_{i})}dy.
\end{align}
Therefore, once  $J^{L}_{i}$ and $J^{R}_{i}$ are computed for all $i$, we then can obtain $I^{L}_{i}$ and $I^{R}_{i}$ via the recursive relation \eqref{eq:recursive}.

In \cite{christlieb2017kernel}, we proposed a high order and robust methodology to compute $J^L_{i}$ and $J^R_{i}$. In what follows we briefly review the underlying procedure for completeness of the paper. 

Note that, to approximate $J^L_{i}$ with $k^{th}$ order accuracy, we may choose the interpolation stencil 
$S(i)=\left\{x_{i-r},\ldots,x_{i-r+k} \right\}$,
which contains $x_{i-1}$ and $x_{i}$. There is a unique polynomial $p(x)$ of degree at most $k$ that interpolates $v(x)$ at the nodes in $S(i)$. Then $J^L_{i}$ is approximated by
\begin{align}
\label{eq:linear_JL}
J^L_{i}\approx\gamma \int_{x_{i-1}}^{x_{i}} p(y)e^{-\gamma (x_{i}-y)}dy.
\end{align}
Note that, the integral on the right hand side can be evaluated exactly.
%{\color{red} Usually, we choose the stencil symmetric about $\left\{x_{i-1},x_{i}\right\}$ or one point biased stencil to ensure stability.
%In particular, when the cell $[x_{i-1},x_{i}]$ near the boundary, we can choose the stencil embedded in the domain. Thus this scheme can avoid finding the values on ghost points when treating complicated geometry.}
Similarly, we can approximate $J^R_{i}$ by
\begin{align}
\label{eq:linear_JR}
J^R_{i}\approx\gamma \int_{x_{i}}^{x_{i+1}} p(y)e^{-\gamma (y-x_{i})}dy,
\end{align}
with polynomial $p(x)$ interpolating $v(x)$ on stencil $S(i)=\left\{x_{i+r-k},\ldots,x_{i+r} \right\}$ which includes $x_{i}$ and $x_{i+1}$.

However, in some cases, the linear quadrature formula \eqref{eq:linear_JL} and \eqref{eq:linear_JR} with a fixed stencil may generate the entropy-violating solutions, which will be demonstrated in Section \ref{sec:result}. The main reason is that the viscosity solution of the H-J equation may involve discontinuous derivative. Hence,  an approximation with a fixed stencil such as \eqref{eq:linear_JL} and \eqref{eq:linear_JR} will develop spurious oscillations and eventually lead to failure of the scheme. We found that the WENO quadrature and the nonlinear filter, discussed in \cite{christlieb2017kernel}, are effective to control oscillations and capture the correct solution for solving nonlinear degenerate parabolic equations. In this work, we adapt this methodology to solve H-J equations as both types of equations may develop discontinuous derivatives for the solutions. The numerical evidence provided in Section \ref{sec:result} justifies the effectiveness of the WENO method and nonlinear filter in controlling oscillation as well as in capturing the viscosity solution to the H-J equation.

In summary, when $\phi_{x}$ is periodic for example, we will use the following modified formulation instead of \eqref{eq:partialsum_per} to approximate $\phi_{x}^{\pm}$:
\begin{subequations}
\label{eq:change_per}
\begin{align}
& \phi_{x,i}^{-}=\gamma\mathcal{D}_{L}[\phi,\gamma](x_{i}) + \gamma\sum_{p=2}^{k}\sigma_{i,L}^{p-2}\mathcal{D}_{L}^{p}[\phi,\gamma](x_{i}),\\
& \phi_{x,i}^{+}=-\gamma\mathcal{D}_{R}[\phi,\gamma](x_{i}) - \gamma\sum_{p=2}^{k}\sigma_{i,R}^{p-2}\mathcal{D}_{R}^{p}[\phi,\gamma](x_{i});
\end{align}
\end{subequations}
while for non-periodic boundary conditions, \eqref{eq:partialsum_dir} are changed into
\begin{subequations}
	\label{eq:change_dir}
	\begin{align}
	& \phi_{x,i}^{-}=\gamma\mathcal{D}_{L}[\phi_{1,1},\gamma](x_{i}) + \gamma\sum_{p=2}^{k}\sigma_{i,L}^{p-2}\mathcal{D}_{L}[\phi_{1,p},\gamma](x_{i}),\\
	& \phi_{x,i}^{+}=-\gamma\mathcal{D}_{R}[\phi_{2,1},\gamma](x_{i}) - \gamma\sum_{p=2}^{k}\sigma_{i,R}^{p-2}\mathcal{D}_{R}[\phi_{2,1},\gamma](x_{i}).
	\end{align}
\end{subequations} 
Note that the WENO quadrature is only applied for approximating the operators with $p = 1$ in \eqref{eq:change_per}, and cheaper high order linear formula are used for those with $p>1$. The filter $\sigma_{i,L}$ and $\sigma_{i,R}$ are generated based on the smoothness indicators from the WENO methodology. 
%For brevity, we skip the details about the construction of the WENO quadrature as well as the nonlinear filter. All the formula have been well documented in \cite{christlieb2017kernel}.

Here, the fifth order WENO-based quadrature for approximating $J^{L}_{i}$ is provided as an example. The corresponding stencil is presented in Figure \ref{Fig0}. We choose the big stencil as $S(i)=\{x_{i-3},\ldots, x_{i+2}\}$ and the three small stencils as $S_{r}(i) =\{ x_{i-3+r},\ldots, x_{i+r}\}$, $r = 0, 1, 2$. For simplicity, we only provide the formulas for the case of a uniform mesh, i.e., $\Delta x_{i}=\Delta x$ for all $i$. Note that the WENO methodology is still applicable to the case of a nonuniform mesh, see \cite{shu2009high}. 

\begin{figure}
	\centering
	\includegraphics[width=0.5\textwidth]{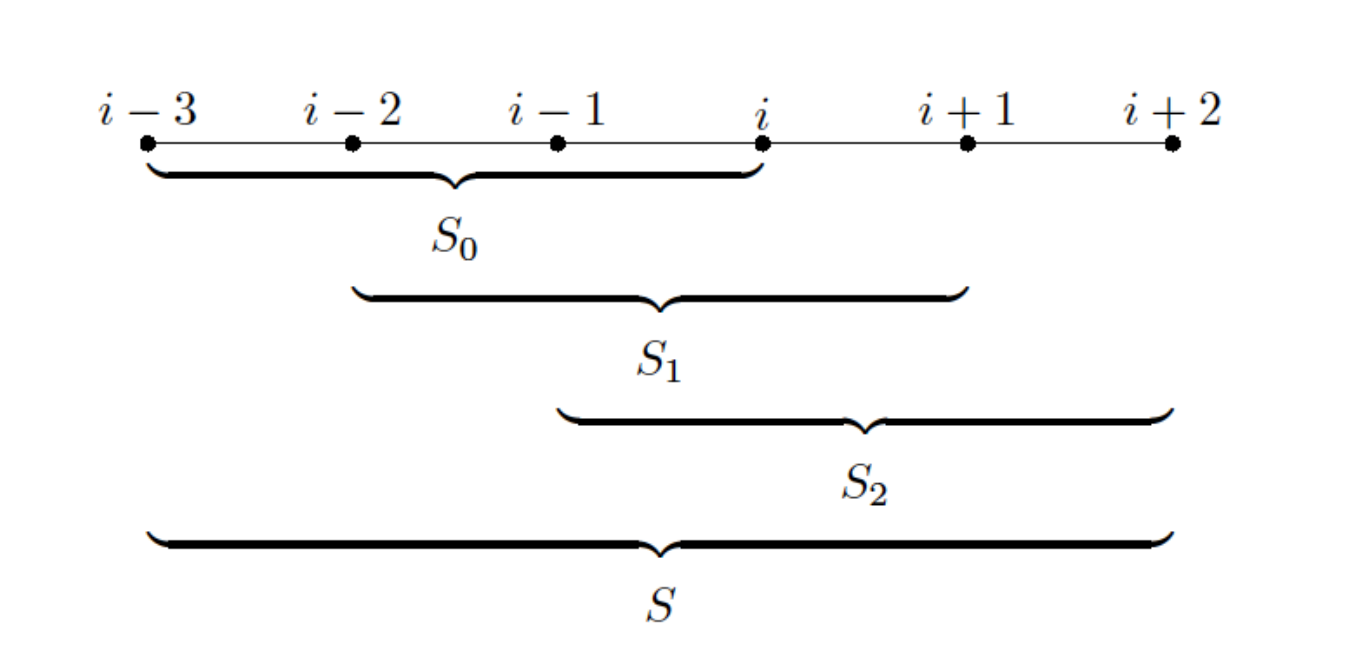}
	\caption{\em The structure of the stencils in WENO integration.   }
	\label{Fig0}
\end{figure}

\begin{enumerate}
	\item
	On each small stencil $S_{r}(i)$, we obtain the approximation as
	\begin{equation}
	\label{eq:weno1}
	J^{L}_{i,r} =\gamma \int_{x_{i-1}}^{x_{i}}e^{-\gamma(x_{i}-y)}p_{r}(y)dx,
	%=\sum_{j=0}^{k}c^{(r)}_{i-r-1+j}v_{i-r-1+j},
	\end{equation}
	where $p_{r}(x)$ is the polynomial interpolating $v(x)$ on nodes $S_{r}(i)$.
	
	\item
	Similarly, on the big stencil $S(i)$, we have
	\begin{equation}
	\label{eq:weno2}
	J^{L}_{i} =\gamma \int_{x_{i-1}}^{x_{i}}e^{-\gamma (x_{i}-y)}p(y)dx=\sum_{r=0}^{2}d_{r}J^{L}_{i,r},
	%=\sum_{j=0}^{2k-1}c_{i-k+j}v_{i-k+j},
	\end{equation}
	with the linear weights $d_{r}$ satisfying $\sum_{r=0}^{2}d_{r}=1$.
	
	\item
	We change the linear weights $d_{r}$ into the following nonlinear weights $\omega_{r}$, which are defined as
	\begin{equation}
	\omega_{r}=\tilde{\omega}_{r}/\sum\limits_{s=0}^{2}\tilde{\omega}_{s}, \ \ r=0,\ 1,\ 2,
	\end{equation}
	with
	\begin{equation*}
	\tilde{\omega}_{r}=d_{r}\left( 1+\frac{\tau_{5}}{\epsilon+\beta_{r}}\right) .
	\end{equation*}
	Here, $\epsilon>0$ is a small number to avoid the denominator becoming zero. We  take $\epsilon=10^{-6}$ in our numerical tests. The smoothness indicator $\beta_{r}$ is defined by
	\begin{equation}
	\beta_{r}=\sum_{l=2}^{3} \int_{x_{i-1}}^{x_{i}}\Delta x_{i}^{2l-3} \left(\frac{\partial^{l}p_{r}(x)}{\partial x^{l}}\right)^2 dx,
	\end{equation}
	which is used to  measure the relative smoothness of the function $v(x)$ in the stencil $S_{r}(i)$. In particular, we have the expressions as 
	\begin{subequations}
		\begin{align*}
		& \beta_{0} = \frac{13}{12} ( -v_{i-3} + 3v_{i-2} - 3v_{i-1} + v_{i} )^2
		+ \frac{1}{4} ( v_{i-3} - 5v_{i-2} + 7v_{i-1} - 3v_{i} )^2,\\
		& \beta_{1} = \frac{13}{12} (-v_{i-2} + 3v_{i-1} - 3v_{i} + v_{i+1} )^2
		+ \frac{1}{4} ( v_{i-2} - v_{i-1} - v_{i} + v_{i+1} )^2,\\
		& \beta_{2} = \frac{13}{12} (-v_{i-1} + 3v_{i} - 3v_{i+1} + v_{i+2} )^2 
		+ \frac{1}{4} (-3v_{i-1} + 7v_{i} - 5v_{i+1} + v_{i+2} )^2.
		\end{align*}
	\end{subequations}
	$\tau_{5}$ is simply defined as the absolute difference between $\beta_{0}$ and $\beta_{2}$ 
	$$\tau_{5}=|\beta_{0}-\beta_{2}|.$$
	Moreover, we define a parameter $\xi_{i}$ as
	\begin{align}
	\xi_{i}=\frac{\beta_{min}}{\beta_{max}},
	\end{align}
	which will be use to construct the nonlinear filter. Here,
	\begin{subequations}
		\begin{align*}
		\beta_{max} = 1 + \left(\frac{\tau_{5} }{  \epsilon + \min(\beta_{0},\beta_{2}) } \right)^2, \quad
		\beta_{min} = 1 + \left(\frac{\tau_{5} }{  \epsilon + \max(\beta_{0},\beta_{2}) } \right)^2.
		\end{align*}
	\end{subequations}
	Note that we have employed the nonlinear weights based on the idea of the WENO-Z scheme proposed in \cite{borges2008improved}, which enjoys less dissipation and higher resolution compared with the original WENO scheme.
	\item
	Lastly, we define the approximation
	\begin{equation}
	J^{L}_{i}=\sum_{r=0}^{2}\omega_{r}J^{L}_{i,r}.
	\end{equation}
	The filter $\sigma_{i,L}$ is defined as 
	\begin{align}
	\sigma_{i,L}=\min (\xi_{i-1}, \xi_{i}).
	\end{align}
	
\end{enumerate}

The process to obtain $J^{R}_{i}$ and $\sigma_{i,R}$ is mirror symmetric to that of $J^{L}_{i}$ and $\sigma_{i,L}$ with
respect to point $x_{i}$.

\section{Two-dimensional Case}
\label{section:2D}

Consider the 2D H-J equation
\begin{equation}
\label{eq:2d}
\phi_{t}+H(\phi_{x},\phi_{y})=0,\quad \phi(x,y,0)=\phi^{0}(x,y).
\end{equation}
Let $(x_{i}, y_{j})$ be the $(i,j)$ node of a 2D orthogonal grid, which can be nonuniform. Denote	
by $\Delta x_{i}=x_{i}-x_{i-1}$ and $\Delta y_{j}=y_{j}-y_{j-1}$ and denote $\phi_{i,j}(t)$ as the solution $\phi(x_{i},y_{j},t)$. Then, as with the 1D case, the semi-discrete scheme is
\begin{align}
\label{eq:2Dscheme}
\frac{d}{dt}\phi_{i,j} + \hat{H}( \phi^{-}_{x,i,j}, \phi^{+}_{x,i,j}; \phi^{-}_{y,i,j}, \phi^{+}_{y,i,j} ) = 0,
\end{align}
%Here, the numerical Hamiltonian $\hat{H}$ is a Lipschitz continuous functions of all four arguments and consistent with $H$
%$$\hat{H}(u,u;v,v)=H(u,v).$$
%$\hat{H}$ is also monotone, meaning it is non-increasing in its first  and third arguments, and non-decreasing in the second and fourth ones, $\hat{H}(\uparrow,\downarrow,\uparrow,\downarrow)$. And
where  $\hat{H}(u,u;v,v)$ is a Lipschitz continuous monotone flux that is consistent with the Hamiltonian $H$. In this work, we employ the following  local Lax-Friedrichs flux given as
\begin{align}
\label{eq:LLF_2D}
\hat{H}(u^{-},u^{+};v^{-},v^{+}) = H(\frac{u^{-}+u^{+}}{2},\frac{v^{-}+v^{+}}{2}) -\alpha_{x}(u^{-},u^{+})\frac{u^{+}-u^{-}}{2} -\alpha_{y}(v^{-},v^{+})\frac{v^{+}-v^{-}}{2},
\end{align}
where,
\begin{align*}
& \alpha_{x}(u^{-},u^{+})=\max_{u\in I(u^{-},u^{+}),C\leq v\leq D} |H_{1}(u,v)|,\\ & \alpha_{y}(v^{-},v^{+})=\max_{A\leq v\leq B, v\in I(v^{-}, v^{+})} |H_{2}(u,v)|.
\end{align*}
Here, $H_{i}(u,v)$ denotes the partial derivative of $H(u,v)$ with respect to the $i^{th}$ argument. $[A,B]$ is the value range of $u^{\pm}$ and $[C,D]$ is the value range of $v^{\pm}$.

In \eqref{eq:2Dscheme}, $\phi^{\pm}_{x,i,j}$ and $\phi^{\pm}_{y,i,j}$ are approximations to $\phi_x(x_{i},y_{j})$ and $\phi_y(x_{i},y_{j})$, respectively, which can be computed directly by the proposed 1D formulation. In particular, when constructing $\phi_{x}^{\pm}$, we fix $y_j$ and apply the 1D operators in $x$-direction. For instance, if $\phi_{x}$ is periodic in $x$-direction, then
\begin{align*}
& \phi_{x,i,j}^{-} \approx \gamma_{x}\sum_{p=1}^{k}\mathcal{D}_{L}^{p} [\phi(\cdot,y_{j}),\gamma_{x}](x_{i}),
& \phi_{x,i,j}^{+}\approx -\gamma_{x}\sum_{p=1}^{k}\mathcal{D}_{R}^{p} [\phi(\cdot,y_{j}),\gamma_{x}](x_{i}).
\end{align*}
Here, we choose $\gamma_{x}=\beta/(\alpha_{x}\Delta t)$ with $\alpha_{x}=\max_{A\leq u\leq B,C\leq v\leq D} |H_{1}(u,v)|$. Similarly, to approximate $\phi_y^\pm$, we fix $x_i$ and obtain
\begin{align*}
& \phi_{y,i,j}^{-} \approx \gamma_{y}\sum_{p=1}^{k}\mathcal{D}_{L}^{p} [\phi(x_{i},\cdot),\gamma_{y}](y_{j}),
& \phi_{y,i,j}^{+}\approx -\gamma_{y}\sum_{p=1}^{k}\mathcal{D}_{R}^{p} [\phi(x_{i},\cdot),\gamma_{y}](y_{j}),
\end{align*}
with $\alpha_{y}=\max_{A\leq u\leq B,C\leq v\leq D} |H_{2}(u,v)|$ and $\gamma_{y}=\beta/(\alpha_{y}\Delta t)$.
Note that in the 2D case,  we need to choose $\beta_{max}$ as half of that for the 1D case to ensure the unconditional stability of the scheme.

For a non-periodic boundary condition, the derivative values at an inflow boundary are still calculated through ILW. To illustrate the idea, we consider a rectangular domain $[a_{x},b_{x}]\times[a_{y},b_{y}]$ for simplicity.  Without loss of generality, assume the left boundary $x=a_{x}$ is an inflow Dirichlet boundary that satisfies 
$$\phi(a_{x},y,t)=f(y,t), \quad  y\in[a_y,b_y]$$
and $f(y,t)$ is a given function. For a fixed time $t$, denote $\phi^{(k_1,k_2)}_{0,\, j} = \frac{\partial^{k_1+k_2}}{\partial x^{k^1}\partial y^{k_2}}\phi(a_x,y_j,t)$, where $(a_x,y_j)$ is a boundary grid point. For a $k^{th}$ order approximation, we need to compute the derivative values $\phi^{(k_1,0)}_{0,\, j}$ for $1\le k_1\le k$ as in the 1D case. Note that $\phi_{t}$ and $\phi_{y}$ at boundary $x=a_x$ can be easily obtained, which are the functions of $y$. Hence, $\phi^{1,0}_{0,\, j}$ can be computed through relation
\begin{align}
f_{t}(y_j)+H(\phi^{(1,0)}_{0,\, j},f_{y}(y_j))=0.
\end{align}
In the case of more than one roots, we should use  
$$H_{1}(\phi^{(1,0)}_{0,\, j},f_{y}(y_j))\geq0,$$
to single out the correct solution of $\phi^{(1,0)}_{0,\, j}$ that ensures the left boundary $x=a_{x}$ is an inflow boundary.
Similar to the 1D case, if this condition still cannot pin down a root, then we would choose the root that is closest to the value obtained from extrapolation.
If we take the derivative with respect to $y$ for the original H-J equation \eqref{eq:2d}, then we obtain
\begin{align}
\label{eq:ilw1}
\phi_{ty}=-H_{1}(\phi_{x},\phi_{y})\phi_{xy} -H_{2}(\phi_{x},\phi_{y})\phi_{yy}.
\end{align}
Notice that at the point $(a_x,y_j)$, the equation becomes 
\begin{align}
\label{eq:ilw1_1}
f_{ty}(y_j)=-H_{1}(\phi^{(1,0)}_{0,\, j},f_{y}(y_j))\phi_{0,\, j}^{(1,1)} -H_{2}(\phi^{(1,0)}_{0,\, j},f_{y}(y_j))f_{yy}(y_j).
\end{align}
The only unknown quantity in \eqref{eq:ilw1} is $\phi_{0,\, j}^{(1,1)}$, which can be easily solved. Furthermore, we take the derivative with respect to $t$ for \eqref{eq:2d} and get
\begin{align}
\label{eq:ilw2}
\phi_{tt}=(H_{1})^2\phi_{xx} +2H_{1}H_{2}\phi_{xy} +(H_{2})^2\phi_{yy}.
\end{align}
When evaluating equation $\eqref{eq:ilw2}$ at point $(a_x,y_j)$, the only unknown quantity is $\phi^{(2,0)}_{0,\, j}$, which can be easily obtained by solving \eqref{eq:ilw1}.  Repeatedly applying this ILW procedure, we are able to obtain the derivative values $\phi^{(k_1,0)}_{0,\, j}$ for  $1\le k_1\le k$. 

We now consider an inflow Neumann boundary condition imposed at boundary $x=a_x$, or equivalently $\phi_x(a_x,y,t) = g(y,t)$ is specified, where $g(y,t)$ is a given function. In this case, at a fixed time $t$, we directly get $\phi_{0,\, j}^{(1,0)} = g(y_j)$. To obtain $\phi_{0,\, j}^{(2,0)}$, we first differentiate the H-J equation with respect to $x$, yielding
$$\phi_{tx} = -H_1(\phi_x,\phi_y)\phi_{xx} - H_2(\phi_x,\phi_y)\phi_{xy}.$$
When evaluating at $(a_x,y_j)$, the equation becomes 
\begin{equation}
f_t(y_j) = - H_1(\phi_{0,\, j}^{(1,0)},\phi_j^{(0,1)})\phi_{0,\, j}^{(2,0)} -H_2(\phi_{0,\, j}^{(1,0)},\phi_{0,\, j}^{(0,1)})g_{y}(y_j),\label{eq:ilwn1}
\end{equation}
where we have used the identity $\phi_{xy}(a_x,y) = g_y(y)$.
Note that, unlike the Dirichlet case, there are two unknowns in \eqref{eq:ilwn1} including $\phi^{(0,1)}_{0,\, j}$ and $\phi^{(2,0)}_{0,\, j}$. As suggested in \cite{tan2010inverse}, we can use the numerical differentiation to approximate the tangential derivatives at the boundary points, i.e., $\phi^{(0,1)}_{0,\, j}$ in this case. Once $\phi_{0,\, j}^{(0,1)}$ is obtained, we can solve for the only unknown quantity $\phi^{(2,0)}_{0,\, j}$ from \eqref{eq:ilwn1}.  We then continue applying ILW to generate $\phi_{0,\, j}^{(k_1,0)}$ for all $ k_1\le k$ and again the needed tangential derivatives are obtained by numerical differentiation.

As in the 1D case, for an outflow boundary, we compute derivative values through extrapolation with suitable orders of accuracy. 

Lastly, we remark that, as a remarkable property, ILW can be conveniently applied to handle complex geometry. In fact, we only need to rotate the coordinate system locally at the underlying boundary point so that the new axes point to the normal and tangential directions of the boundary. By doing so, we are able to apply the above ILW procedure and generate the required derivative values at the inflow boundary.

\section{Numerical results}\label{sec:result}
In this section, we present the results of the proposed schemes to demonstrate their efficiency and efficacy. For the 1D cases, we choose the time step as
\begin{align}
\Delta t=\text{CFL}\frac{\Delta x}{\alpha},
\end{align}
where $\alpha$ is the maximum wave propagation speed.
For the 2D cases, the time step is chosen as
\begin{align}
\Delta t=\frac{\text{CFL}}{\max(\alpha_{x}/\Delta x,\alpha_{y}/\Delta y)},
\end{align}
where $\alpha_x$ and $\alpha_y$ are the maximum wave propagation speeds in $x-$ and $y-$ directions, respectively. We use the fifth order linear or WENO quadrature formula to compute $J_i^L$ and  $J_i^R$.
%with $h_{x}=\min_{i}\Delta x_{i}$ and with $h_y=\min_{j}\Delta y_{j}$.
Unless additional declaration is made, linear quadrature formula will be adopted. Note that the lower order temporal error will always dominate the total numerical error in our examples. Furthermore, except we make additional declaration, we only report the numerical results computed by the scheme with $k=3$ for brevity and two CFLs including 0.5 and 2 are used to demonstrate the performance.

\begin{exa} \label{ex:add1}
We first test the accuracy of the schemes for solving the linear problem
\begin{align}
\left\{\begin{array}{ll}
\phi_{t}+\phi_{x}=0, & -\pi\leq x\leq \pi\\
\phi^0(x)=\sin(x).\\
\end{array}
\right.
\end{align}
with a Dirichlet boundary condition $\phi(-\pi,t)=\sin(-\pi-t)$. The exact solution is $\phi(x,t)=\sin(x-t)$.
In Table \ref{tab:add1}, we test the problem on a uniform mesh and report the errors and the associated order of accuracy  at time $T=20$. 
Note that when $k=1$, the second order accuracy is observed, which is unexpected since the error estimate only implies first order accuracy. Such superconvergence is observed because when applied to the linear problem, the first order scheme with $\beta=2$ is equivalent to the second order Crank-Nicolson scheme in the MOL$^T$ framework \cite{christlieb2016weno}. For general nonlinear problems, the superconvergence cannot be observed anymore, demonstrating the sharpness of the error estimate.
\end{exa}

\begin{table}[htb]
	\caption{\label{tab:add1}\em Example \ref{ex:add1}: $L_{\infty}$ Errors and orders of accuracy for linear equation on the uniform meshes with a Dirichlet boundary condition. $T=20$. }
	\centering
	\bigskip
	\begin{small}
		\begin{tabular}{|c|c|cc|cc|cc|}
			\hline
			\multirow{2}{*}{CFL} &  \multirow{2}{*}{$N$} & \multicolumn{2}{c|}{$k=1$. $\beta=2$.} & \multicolumn{2}{c|}{$k=2$. $\beta=1$.} & \multicolumn{2}{c|}{$k=3$. $\beta=1.2$.}\\
			\cline{3-8}
			& &  error &   order  &  error &  order  &  error  & order  \\\hline
			\multirow{6}{*}{0.5}  
			&   20  &  8.878E-03  &     --     &  1.132E-01  &     --      &  1.989E-03  &     --     \\
			&   40  &  2.285E-03  &  1.958  &  3.054E-02  &  1.890  &  1.679E-04  &  3.566  \\
			&   80  &  5.760E-04  &  1.988  &  7.843E-03  &  1.961  &  1.676E-05  &  3.325  \\
			& 160  &  1.443E-04  &  1.997  &  1.973E-03  &  1.991  &  1.926E-06  &  3.121  \\
			& 320  &  3.615E-05  &  1.998  &  4.946E-04  &  1.996  &  2.356E-07  &  3.031  \\
			& 640  &  9.039E-06  &  2.000  &  1.237E-04  &  1.999  &  2.930E-08  &  3.007   \\\hline    
			\multirow{6}{*}{1} 
			&  20   &  3.872E-02  &     --     &  3.556E-01  &     --      &  2.335E-02  &     --      \\
			&  40   &  9.813E-03  &  1.980  &  1.133E-01  &  1.651  &  1.936E-03  &  3.592  \\
			&  80   &  2.465E-03  &  1.993  &  3.064E-02  &  1.886  &  1.649E-04  &  3.554  \\
			& 160  &  6.195E-04  &  1.992  &  7.846E-03  &  1.965  &  1.666E-05  &  3.307  \\
			& 320  &  1.551E-04  &  1.998  &  1.973E-03  &  1.991  &  1.923E-06  &  3.115  \\
			& 640  &  3.882E-05  &  1.998  &  4.947E-04  &  1.996  &  2.355E-07  &  3.029  \\\hline      
			\multirow{6}{*}{2}  
			&   20  &  1.754E-01   &    --     &  8.534E-01  &     --      &  2.069E-01  &      --     \\
			&   40  &  4.258E-02  &  2.042  &  3.561E-01  &  1.261  &  2.339E-02  &  3.145  \\
			&   80  &  1.067E-02  &  1.997  &  1.135E-01  &  1.649  &  1.945E-03  &  3.588  \\
			& 160  &  2.667E-03  &  2.000  &  3.065E-02  &  1.889  &  1.651E-04  &  3.559  \\
			& 320  &  6.685E-04  &  1.996  &  7.849E-03  &  1.965  &  1.667E-05  &  3.308  \\
			& 640  &  1.672E-04  &  1.999  &  1.974E-03  &  1.992  &  1.924E-06  &  3.115  \\\hline 	
		\end{tabular}
	\end{small}
\end{table}

\begin{exa}
	\label{ex1}
Next, we solve the 1D Burgers' equation:
\begin{align}
\left\{\begin{array}{ll}
\phi_{t}+\frac{1}{2}(\phi_{x}+1)^2=0, & -1\leq x\leq 1\\
\phi^0(x)=-\cos(\pi x)\\
\end{array}
\right.
\end{align}
with a $2$-periodic boundary condition.

When $T= 0.5/\pi^2$, the solution is still smooth. The $L_{\infty}$ errors and orders of accuracy are presented with uniform  and nonuniform meshes in  Table \ref{tab1} and Table \ref{tab1b}, respectively. The nonuniform meshes are obtained
from uniform meshes with $20\%$ random perturbation. The designed order of accuracy is observed for both types of meshes. Note that the schemes allow for large CFL numbers due to the notable unconditional stability property.

When $T=3.5/\pi^2$, the solution of the Burgers' equation has developed  discontinuous derivative. We benchmark our method by comparing the numerical results with the exact solution in Figure \ref{Fig1}. Here, a uniform mesh with $N=40$ is used. It is observed that our schemes are able to approximate the non-smooth solution very well without  producing noticeable oscillations. Large CFLs can be used to save computational cost as with the linear case. Furthermore, we note that, compared with low order schemes, the high order one is able to capture the solution structure more accurately.
\end{exa}

\begin{table}[htb]
	\caption{\label{tab1}\em Example \ref{ex1}: $L_{\infty}$ Errors and orders of accuracy for 1D Burgers' equation on the uniform meshes. $T=0.5/\pi^2$. }
	\centering
	\bigskip
	\begin{small}
		\begin{tabular}{|c|c|cc|cc|cc|}
			\hline
			\multirow{2}{*}{CFL} &  \multirow{2}{*}{$N$} & \multicolumn{2}{c|}{$k=1$. $\beta=2$.} & \multicolumn{2}{c|}{$k=2$. $\beta=1$.} & \multicolumn{2}{c|}{$k=3$. $\beta=1.2$.}\\
			\cline{3-8}
			& &  error &   order  &  error &  order  &  error  & order  \\\hline
			\multirow{6}{*}{0.5}  
			&   20  &  1.593E-02  &      --     &  1.757E-02  &      --      &  6.315E-04  &     --      \\
			&   40  &  7.912E-03  &   1.010  &  4.810E-03  &   1.869  &  4.710E-05  &   3.745  \\
			&   80  &  4.042E-03  &   0.969  &  1.279E-03  &   1.911  &  4.678E-06  &   3.332  \\
			& 160  &  2.067E-03  &   0.967  &  3.313E-04  &   1.949  &  5.285E-07  &   3.146  \\
			& 320  &  1.047E-03  &   0.982  &  8.492E-05  &   1.964  &  6.058E-08  &   3.125  \\
			& 640  &  5.239E-04  &   0.999  &  2.132E-05  &   1.994  &  7.097E-09  &   3.094  \\\hline            
			\multirow{6}{*}{1}  
			&   20  &  3.066E-02  &      --     &  5.891E-02  &     --   &  4.797E-03  &      --     \\
			&   40  &  1.579E-02  &   0.957  &  1.767E-02  &   1.737  &  5.779E-04  &   3.053  \\
			&   80  &  7.908E-03  &   0.998  &  4.839E-03  &   1.869  &  5.260E-05  &   3.458  \\
			& 160  &  4.082E-03  &   0.954  &  1.279E-03  &   1.919  &  5.030E-06  &   3.386  \\
			& 320  &  2.067E-03  &   0.981  &  3.314E-04  &   1.949  &  5.373E-07  &   3.227  \\
			& 640  &  1.048E-03  &   0.980  &  8.492E-05  &   1.964  &  6.082E-08  &   3.143  \\\hline 
			\multirow{6}{*}{2}  
			&   20  &  8.435E-02  &      --     &  1.753E-01  &      --      &  3.901E-02  &      --     \\
			&   40  &  3.160E-02  &   1.417  &  6.078E-02  &   1.528  &  5.347E-03  &   2.867  \\
			&   80  &  1.588E-02  &   0.992  &  1.777E-02  &   1.774  &  5.806E-04  &   3.203  \\
			& 160  &  8.004E-03  &   0.988  &  4.838E-03  &   1.877  &  5.275E-05  &   3.460  \\
			& 320  &  4.082E-03  &   0.972  &  1.280E-03  &   1.919  &  5.035E-06  &   3.389  \\
			& 640  &  2.069E-03  &   0.980  &  3.314E-04  &   1.949  &  5.374E-07  &   3.228  \\\hline 	
		\end{tabular}
	\end{small}
\end{table}

\begin{table}[htb]
	\caption{\label{tab1b}\em Example \ref{ex1}: $L_{\infty}$ Errors and orders of accuracy for 1D Burgers' equation on the non-uniform meshes. $T=0.5/\pi^2$. }
	\centering
	\bigskip
	\begin{small}
		\begin{tabular}{|c|c|cc|cc|cc|}
			\hline
			\multirow{2}{*}{CFL} &  \multirow{2}{*}{$N$} & \multicolumn{2}{c|}{$k=1$. $\beta=2$.} & \multicolumn{2}{c|}{$k=2$. $\beta=1$.} & \multicolumn{2}{c|}{$k=3$. $\beta=1.2$.}\\
			\cline{3-8}
			& &  error &   order  &  error &  order  &  error  & order  \\\hline
			\multirow{6}{*}{0.5}  
			&   20  &  1.526E-02  &      --      &  1.791E-02  &     --      &  5.339E-04  &      --     \\
			&   40  &  8.278E-03  &   0.882  &  4.866E-03  &   1.880  &  6.147E-05  &   3.118  \\
			&   80  &  4.124E-03  &   1.005  &  1.288E-03  &   1.917  &  4.665E-06  &   3.720  \\
			& 160  &  2.069E-03  &   0.995  &  3.324E-04  &   1.954  &  5.286E-07  &   3.142  \\
			& 320  &  1.047E-03  &   0.982  &  8.505E-05  &   1.967  &  6.059E-08  &   3.125  \\
			& 640  &  5.239E-04  &   0.999  &  2.134E-05  &   1.995  &  7.097E-09  &   3.094  \\\hline 
			\multirow{6}{*}{1}  
			&   20  &  3.085E-02  &      --      &  6.300E-02  &      --     &  4.751E-03  &      --     \\
			&   40  &  1.590E-02  &   0.956  &  1.826E-02  &   1.787  &  5.823E-04  &   3.029  \\
			&   80  &  7.990E-03  &   0.993  &  4.910E-03  &   1.894  &  5.028E-06  &   3.388  \\
			& 160  &  4.083E-03  &   0.968  &  1.288E-03  &   1.930  &  7.339E-05  &   2.596  \\
			& 320  &  2.068E-03  &   0.982  &  3.325E-04  &   1.954  &  5.373E-07  &   3.226  \\
			& 640  &  1.048E-03  &   0.981  &  8.506E-05  &   1.967  &  6.083E-08  &   3.143  \\\hline  		 
			\multirow{6}{*}{2}  
			&   20  &  8.400E-02  &      --     &  2.025E-01  &      --      &  3.790E-02  &      --     \\
			&   40  &  3.159E-02  &   1.411  &  6.421E-02  &   1.657  &  5.340E-03  &   2.827  \\
			&   80  &  1.599E-02  &   0.982  &  1.829E-02  &   1.812  &  5.796E-04  &   3.204  \\
			& 160  &  8.005E-03  &   0.998  &  4.909E-03  &   1.897  &  5.279E-05  &   3.457  \\
			& 320  &  4.083E-03  &   0.971  &  1.288E-03  &   1.930  &  5.034E-06  &   3.390  \\
			& 640  &  2.069E-03  &   0.981  &  3.325E-04  &   1.954  &  5.374E-07  &   3.228  \\\hline 	
		\end{tabular}
	\end{small}
\end{table}

\begin{figure}
	\centering
	\subfigure[$k=1$. $\beta=2$. Uniform mesh.]{
		\includegraphics[width=0.3\textwidth]{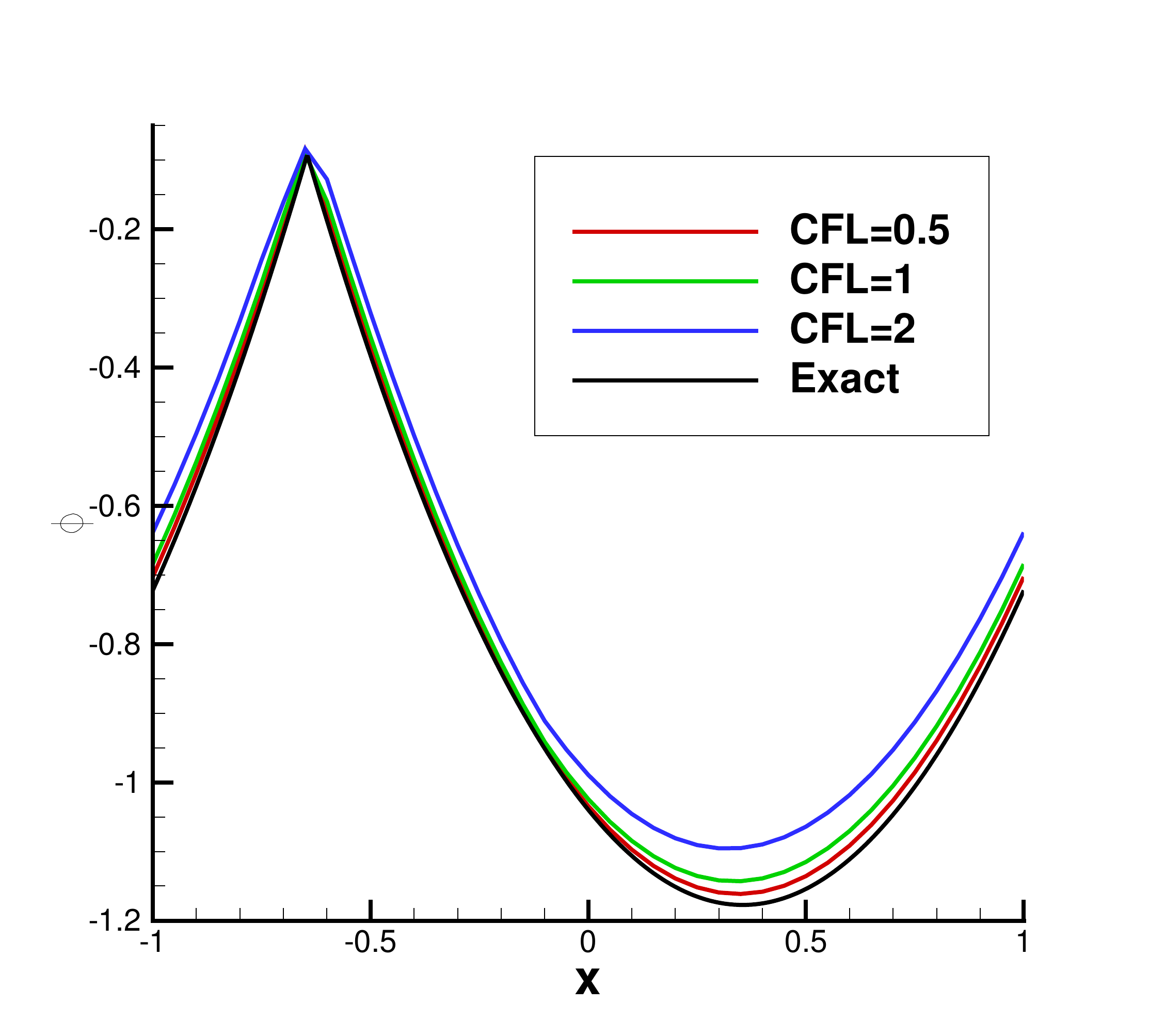}\label{Fig1.1}}
	\subfigure[$k=2$. $\beta=1$. Uniform mesh.]{
		\includegraphics[width=0.3\textwidth]{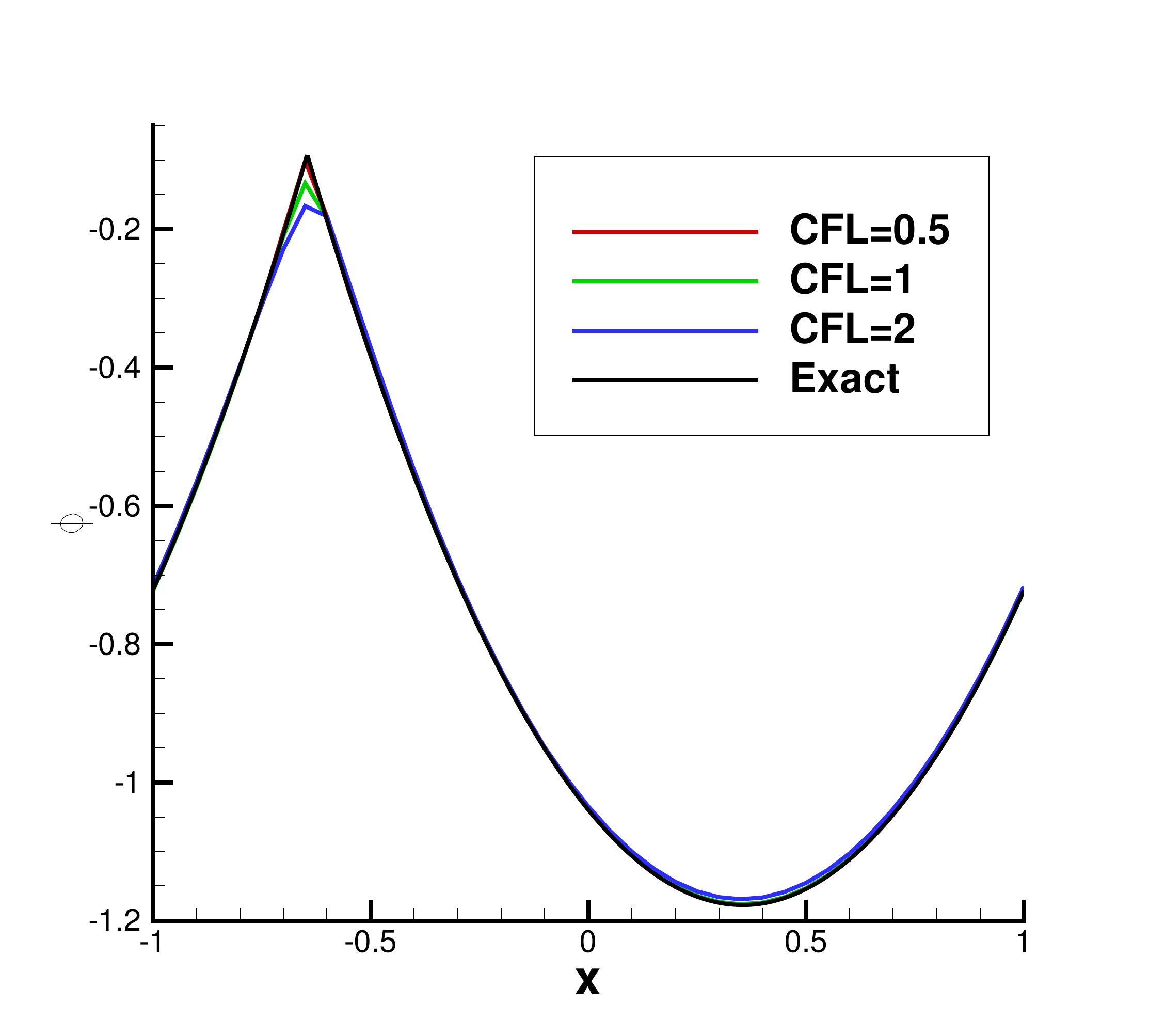}\label{Fig1.2}}
	\subfigure[$k=3$. $\beta=1.2$. Uniform mesh. ]{
		\includegraphics[width=0.3\textwidth]{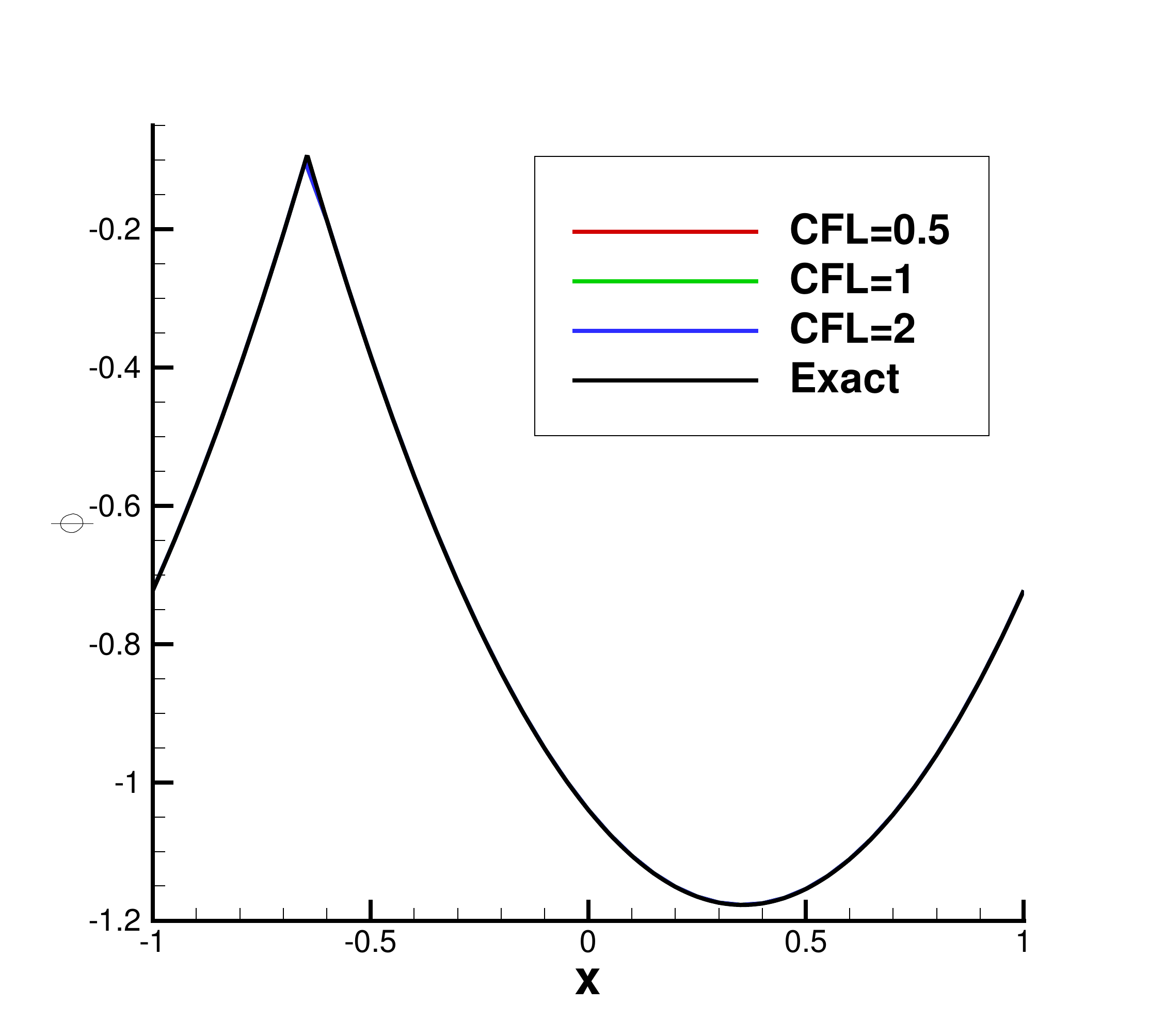}\label{Fig1.3}}
	\subfigure[$k=1$. $\beta=2$. Nonuniform mesh. ]{
		\includegraphics[width=0.3\textwidth]{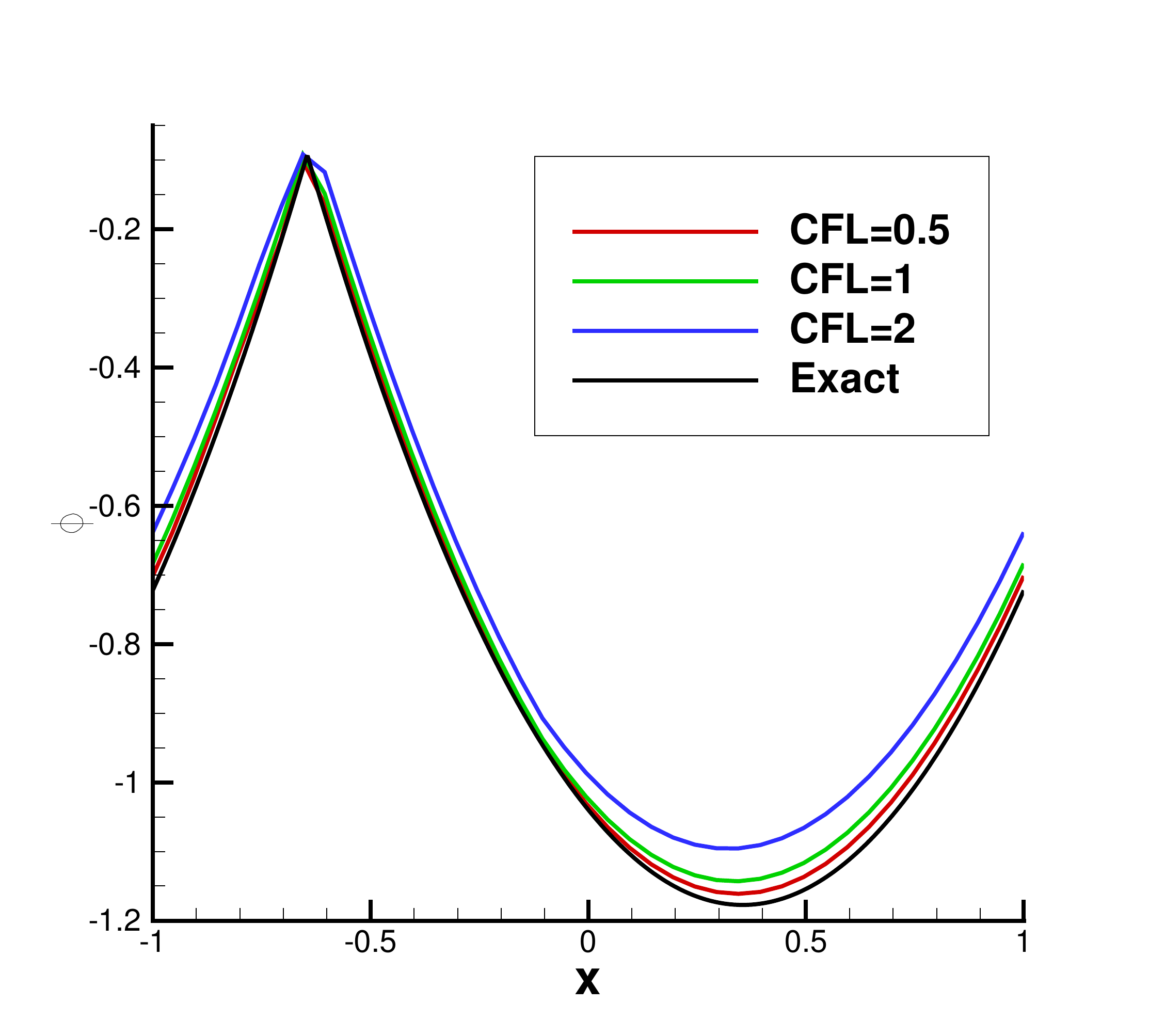}\label{Fig1.4}}
	\subfigure[$k=2$. $\beta=1$. Nonuniform mesh. ]{
		\includegraphics[width=0.3\textwidth]{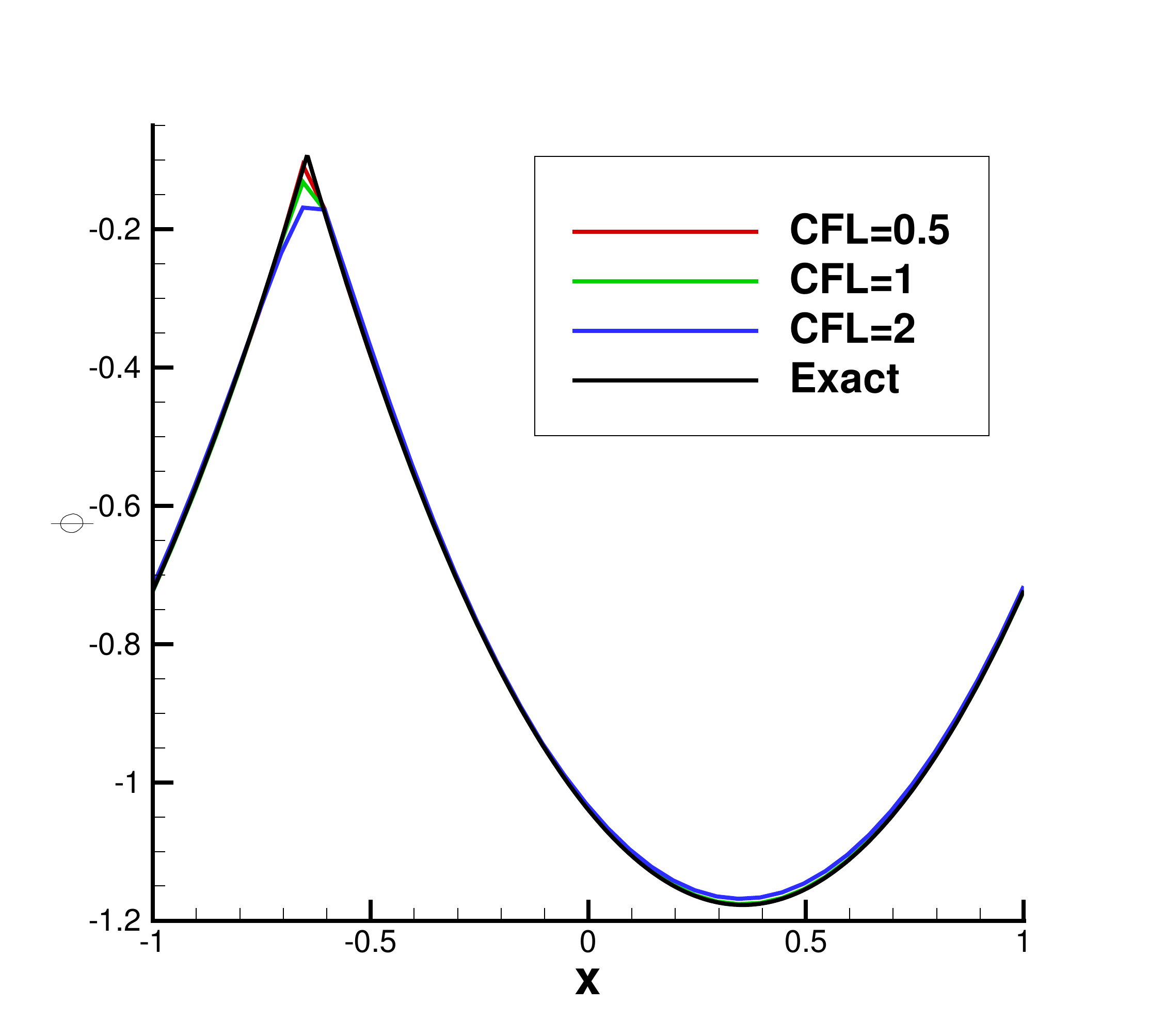}\label{Fig1.5}}
	\subfigure[$k=3$. $\beta=1.2$. Nonuniform mesh. ]{
		\includegraphics[width=0.3\textwidth]{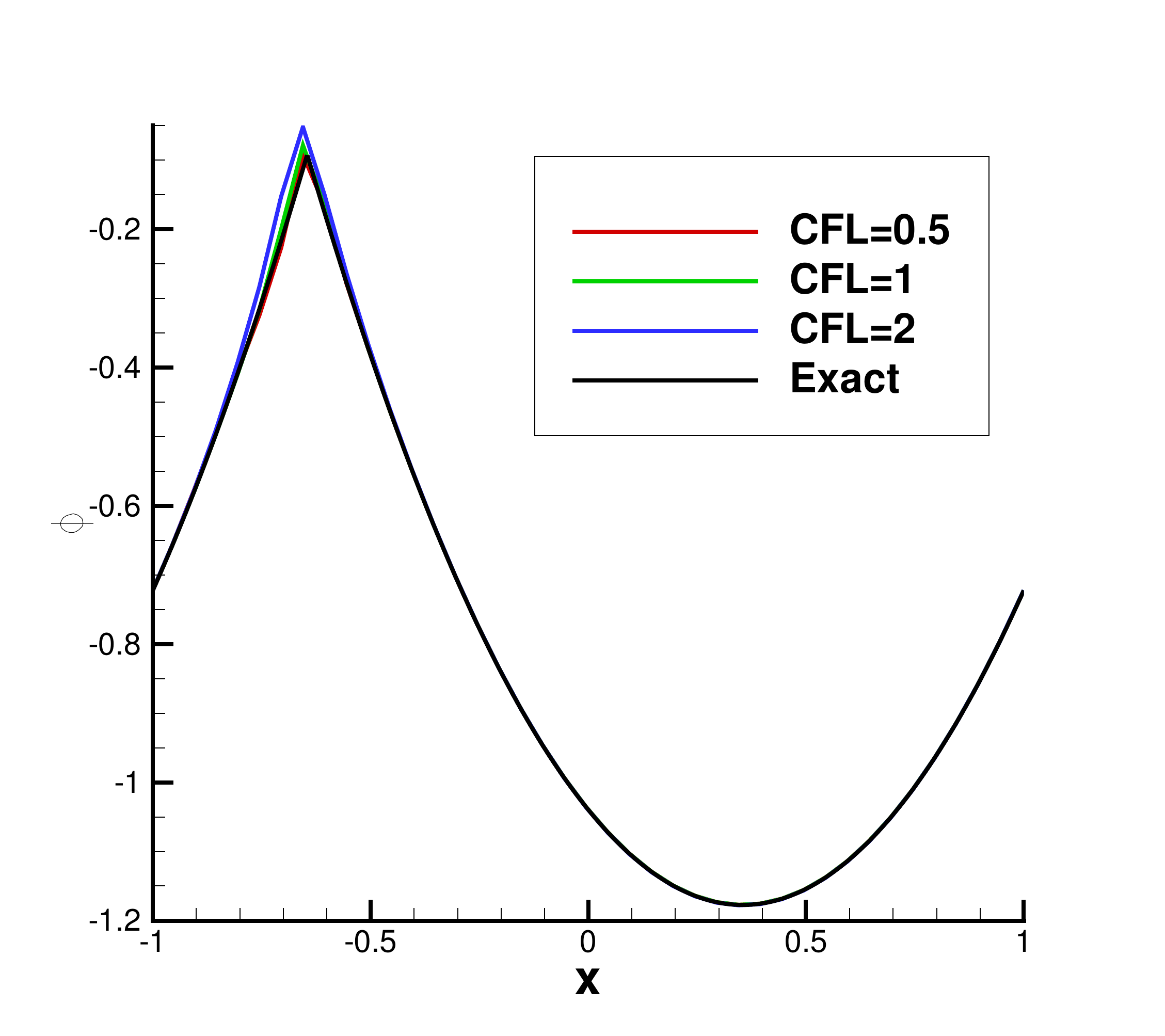}\label{Fig1.6}}
	\caption{\em  Example \ref{ex1}: 1D Burgers' equation. $T=3.5/\pi^2$. $N=40$ grid points. }
	\label{Fig1}
\end{figure}

\begin{exa} \label{ex2}
Consider the 1D H-J equation with a non-convex Hamiltonian:
\begin{align}
\left\{\begin{array}{ll}
\phi_{t}-\cos(\phi_{x}+1)=0, & -1\leq x\leq 1\\
\phi^{0}(x)=-\cos(\pi x)\\
\end{array}
\right.
\end{align}
with a $2$-periodic boundary condition.

When $T= 0.5/\pi^2$, the solution is still smooth. We report $L_{\infty}$ errors and orders of accuracy in Table \ref{tab2} to numerically demonstrate that our schemes can achieve the designed order accuracy even if large CFLs are used.
When $T=1.5/\pi^2$, the solution has developed discontinuous derivative. In Figure \ref{Fig2}, we compare the numerical results with the exact solutions on a uniform mesh with $N=40$. Similar to the previous example, our schemes are able to approximate the  solution accurately without producing noticeable oscillations.
\end{exa}

\begin{table}[htb]
	\caption{\label{tab2}\em Example \ref{ex2}: Errors and orders of accuracy for 1D equation with a non-convex Hamiltonian. $T=0.5/\pi^2$.  }
	\centering
	\bigskip
	\begin{small}
		\begin{tabular}{|c|c|cc|cc|cc|}
			\hline
			\multirow{2}{*}{CFL} &  \multirow{2}{*}{$N$} & \multicolumn{2}{c|}{$k=1$. $\beta=2$.} & \multicolumn{2}{c|}{$k=2$. $\beta=1$.} & \multicolumn{2}{c|}{$k=3$. $\beta=1.2$.}\\
			\cline{3-8}
			& &  error &   order  &  error &  order  &  error  & order  \\\hline
			\multirow{6}{*}{0.5}		
			&   20  &  4.250E-03  &      --      &  2.616E-03  &      --     &  4.115E-04  &      --     \\
			&   40  &  1.951E-03  &   1.123  &  8.490E-04  &   1.624  &  3.841E-05  &   3.421  \\
			&   80  &  1.029E-03  &   0.923  &  2.407E-04  &   1.818  &  3.944E-06  &   3.284  \\
			& 160  &  5.082E-04  &   1.017  &  6.510E-05  &   1.887  &  5.830E-07  &   2.758  \\
			& 320  &  2.546E-04  &   0.997  &  1.716E-05  &   1.923  &  7.282E-08  &   3.001  \\
			& 640  &  1.263E-04  &   1.011  &  4.367E-06  &   1.975  &  8.766E-09  &   3.054  \\\hline         
			\multirow{6}{*}{1}  
			&   20  &  1.078E-02  &      --      &  7.689E-03  &      --     &  4.453E-04  &      --     \\
			&   40  &  4.601E-03  &   1.229  &  2.630E-03  &   1.548  &  9.971E-05  &   2.159  \\
			&   80  &  2.079E-03  &   1.146  &  8.533E-04  &   1.624  &  2.936E-05  &   1.764 \\
			& 160  &  1.019E-03  &   1.029  &  2.407E-04  &   1.826  &  4.678E-06  &   2.650  \\
			& 320  &  5.082E-04  &   1.004  &  6.510E-05  &   1.886  &  6.026E-07  &   2.956  \\
			& 640  &  2.546E-04  &   0.997  &  1.716E-05  &   1.923  &  7.348E-08  &   3.036  \\\hline 
			\multirow{6}{*}{2}  
			&   20  &  3.438E-02  &      --     &  3.629E-02  &      --      &  7.600E-04  &       --     \\
			&   40  &  1.166E-02  &   1.560  &  7.990E-03  &   2.183  &  6.513E-05  &   3.545  \\
			&   80  &  4.665E-03  &   1.322  &  2.650E-03  &   1.592  &  1.183E-04  &  -0.861  \\
			& 160  &  2.073E-03  &   1.170  &  8.536E-04  &   1.634  &  2.995E-05  &   1.981  \\
			& 320  &  1.019E-03  &   1.025  &  2.407E-04  &   1.827  &  4.690E-06  &   2.675  \\
			& 640  &  5.082E-04  &   1.003  &  6.511E-05  &   1.886  &  6.032E-07  &   2.959  \\\hline 
		\end{tabular}
	\end{small}
\end{table}

\begin{figure}
	\centering
	\subfigure[$k=1$. $\beta=2$. ]{
		\includegraphics[width=0.3\textwidth]{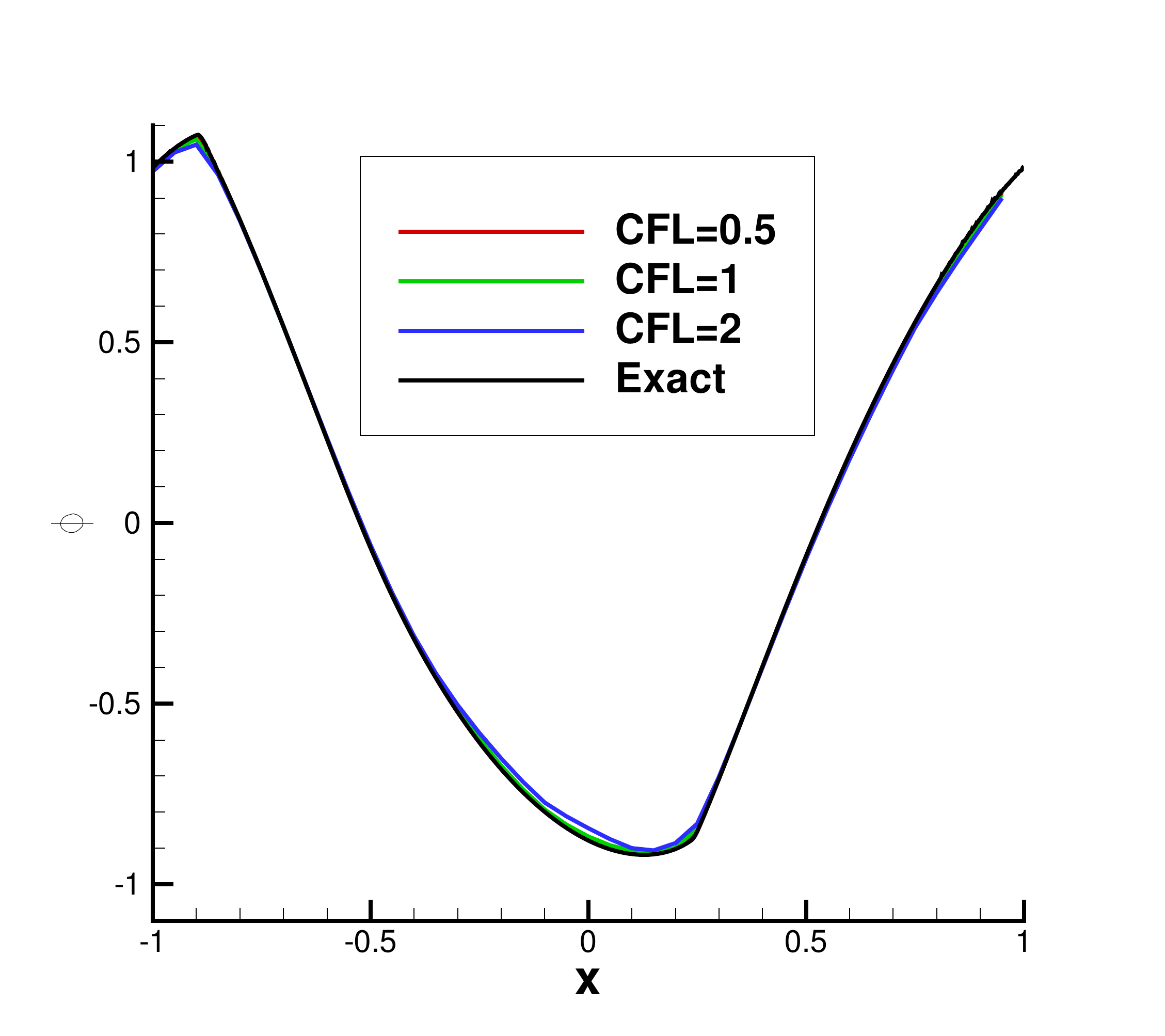}\label{Fig2.1}}
	\subfigure[$k=2$. $\beta=1$. ]{
		\includegraphics[width=0.3\textwidth]{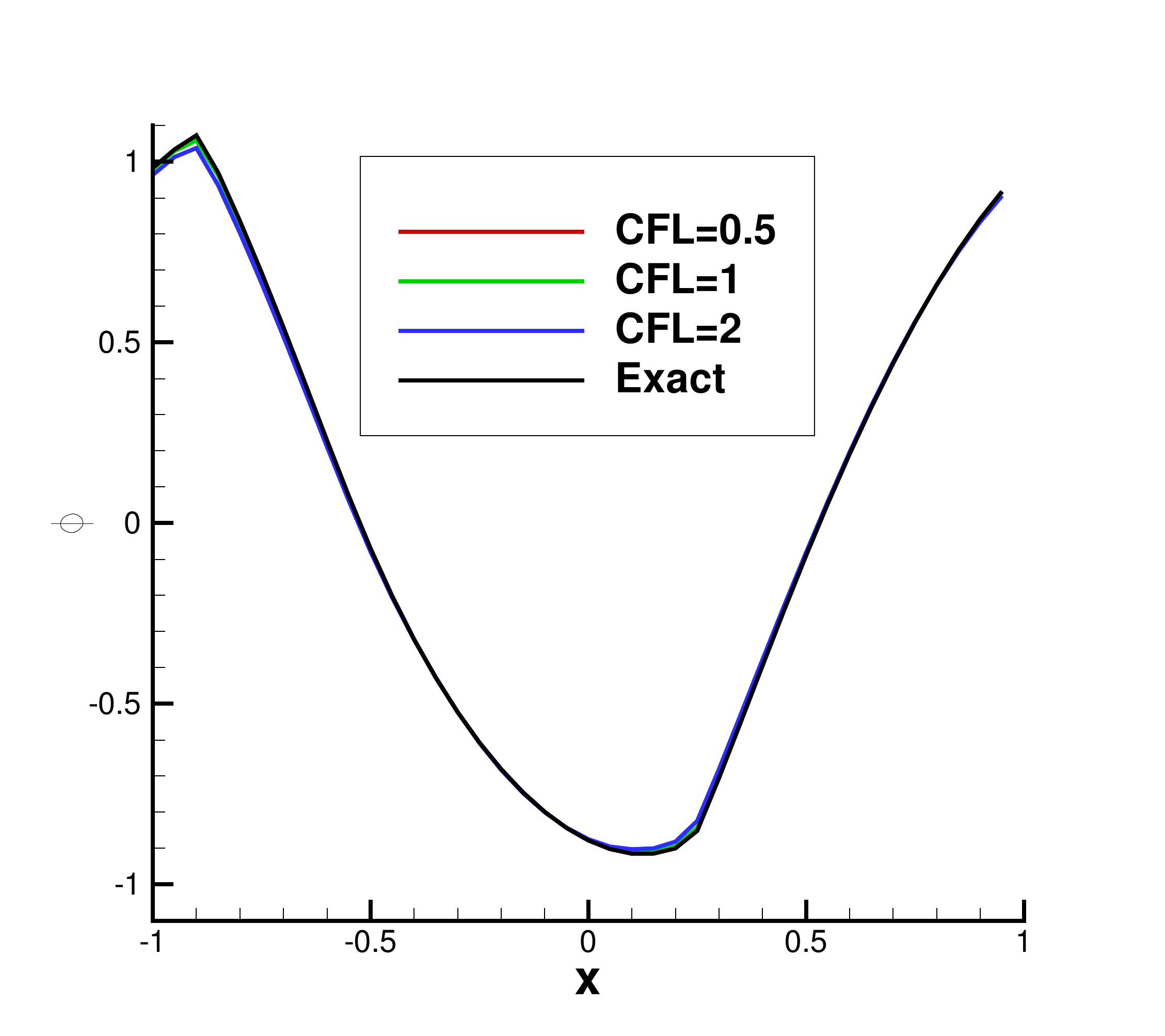}\label{Fig2.2}}
	\subfigure[$k=3$. $\beta=1.2$. ]{
		\includegraphics[width=0.3\textwidth]{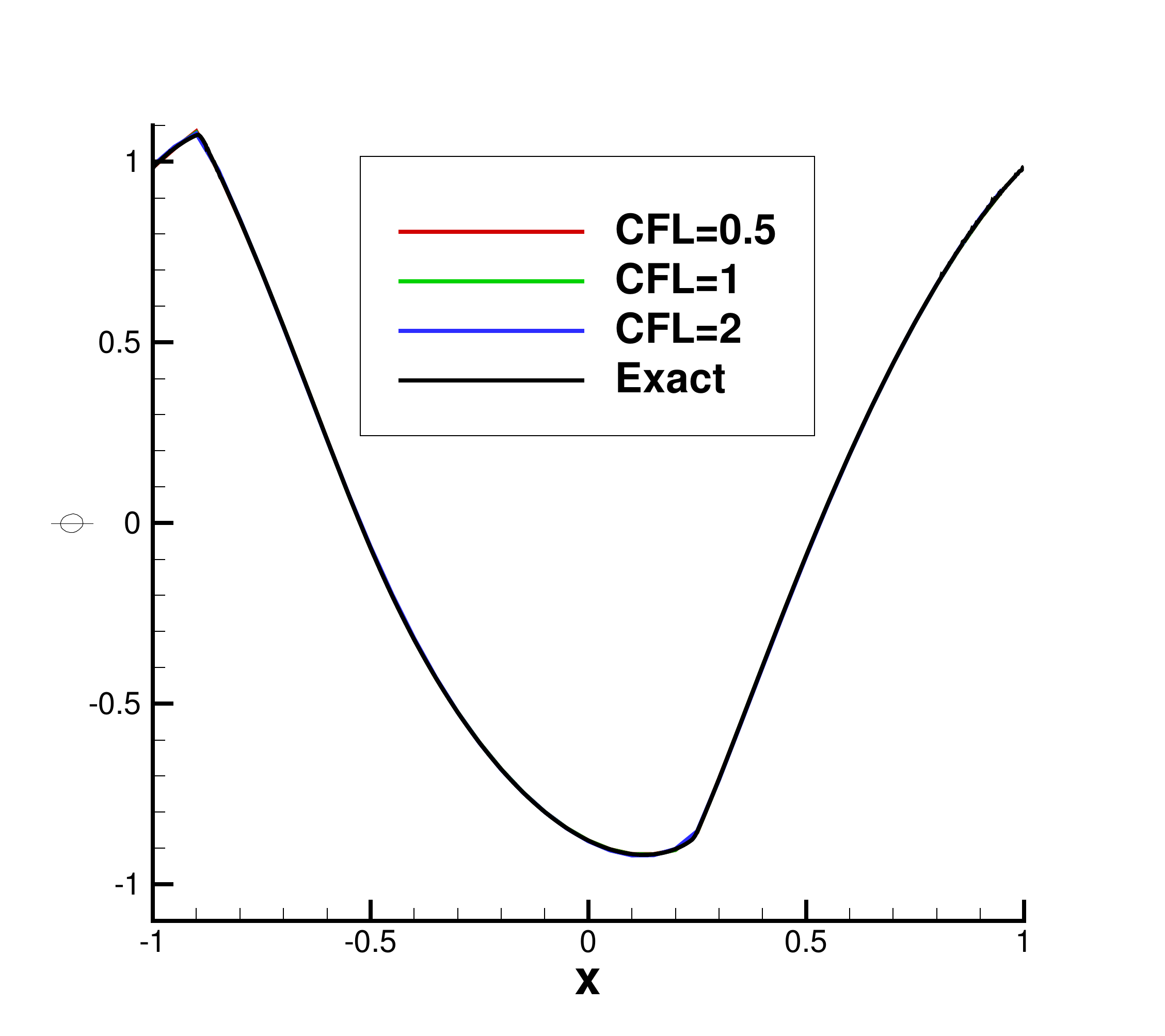}\label{Fig2.3}}
	\caption{\em  Example \ref{ex2}: 1D HJ equation with a non-convex Hamiltonian. $T=1.5/\pi^2$. $N=40$ grid points. }
	\label{Fig2}
\end{figure}

\begin{exa}\label{ex3}
In this example, we solve the 1D Riemann problem with a non-convex Hamiltonian:
\begin{align}
\left\{ \begin{array}{ll}
\phi_{t}+\frac{1}{4} \left(\phi_{x}^2-1\right) \left( \phi_{x}^2-4 \right)=0, & -1\leq x\leq 1,\\
\phi^{0}(x)=-2|x|.\\
\end{array}
\right.
\end{align}
Inflow Dirichlet boundary conditions are imposed with $\phi(-1,t)=-2$ and $\phi(1,t)=-2$. 
Numerical results computed by the proposed schemes with different orders are shown in Figure \ref{Fig3}. In the simulation, we compute the solution up to $T = 1$ and use a mesh with $N=80$ grids.  We remark that, to capture the correct viscosity solution, we need to incorporate the WENO quadrature as well as nonlinear filter, which is justified by the  numerical evidence (Figure \ref{Fig3.4}). This example and Example 4.7 are the only examples in which  the nonlinear WENO quadrature and filter are required.
In addition, we plot the numerical solutions with CFL$=10$, to demonstrate the unconditional stability property of our schemes. 
We also would like to remark that one can save computational cost by choosing large CFLs, e.g. CFL=10, but the approximation quality deteriorates due to the corresponding large temporal error. If small CFLs are used, e.g, CFL=1, our numerical schemes are able to generate high quality numerical results, but the cost increases. The optimal choice of CFL is problem-dependent. 
%In practice, we suggest to choose CFL between 1 and 2 and such a choice qualitatively balances the cost and the accuracy in our tests.  
% we show the results with $N=800$ grid points and  $CFL=10$ as well in Figure \ref{3.5}
\end{exa}

\begin{figure}
	\centering
	\subfigure[$k=1$. $\beta=2$.]{
		\includegraphics[width=0.4\textwidth]{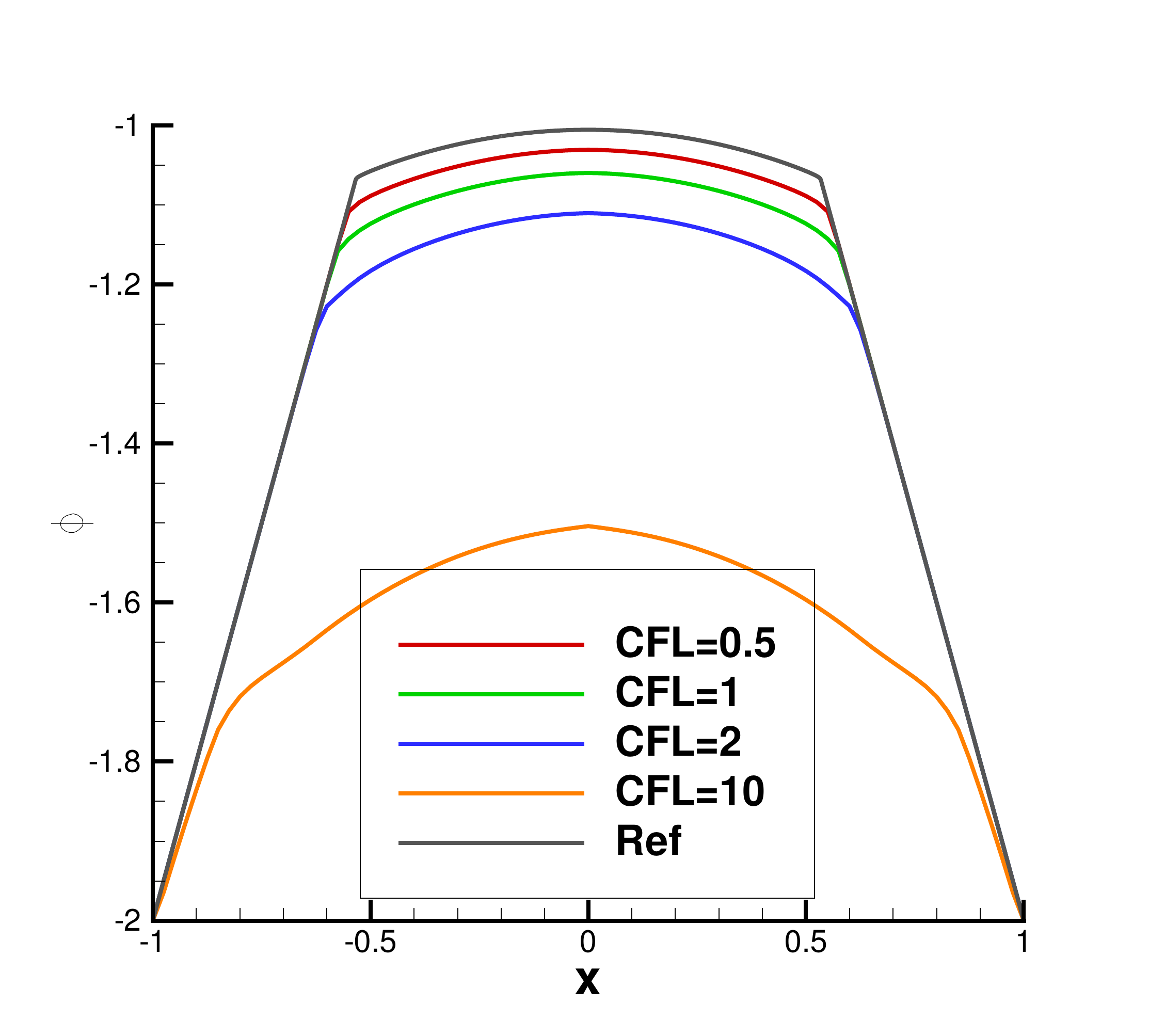}\label{Fig3.1}}
	\subfigure[$k=2$. $\beta=1$. ]{
		\includegraphics[width=0.4\textwidth]{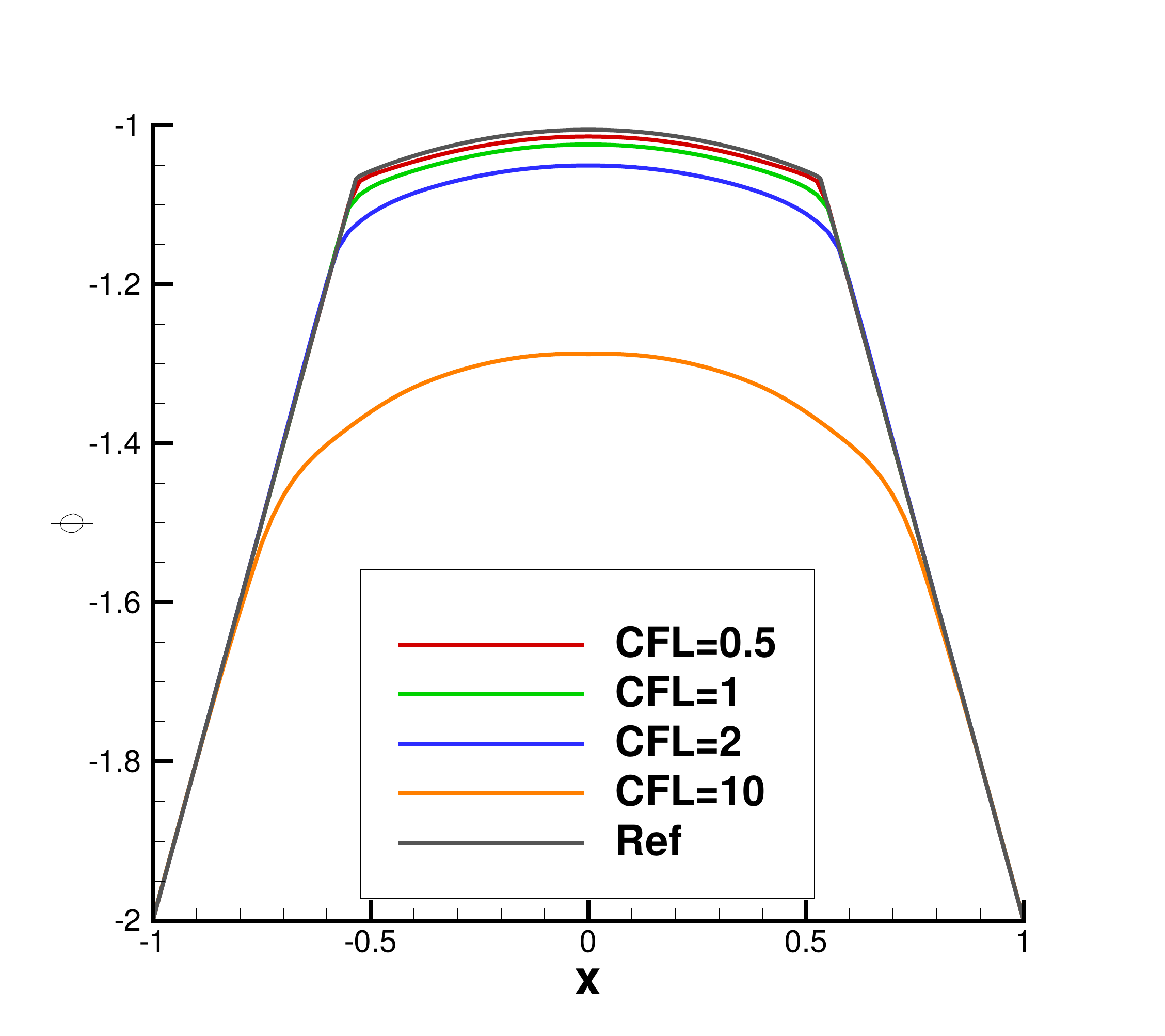}\label{Fig3.2}}\\
	\subfigure[$k=3$. $\beta=1.2$.]{
		\includegraphics[width=0.4\textwidth]{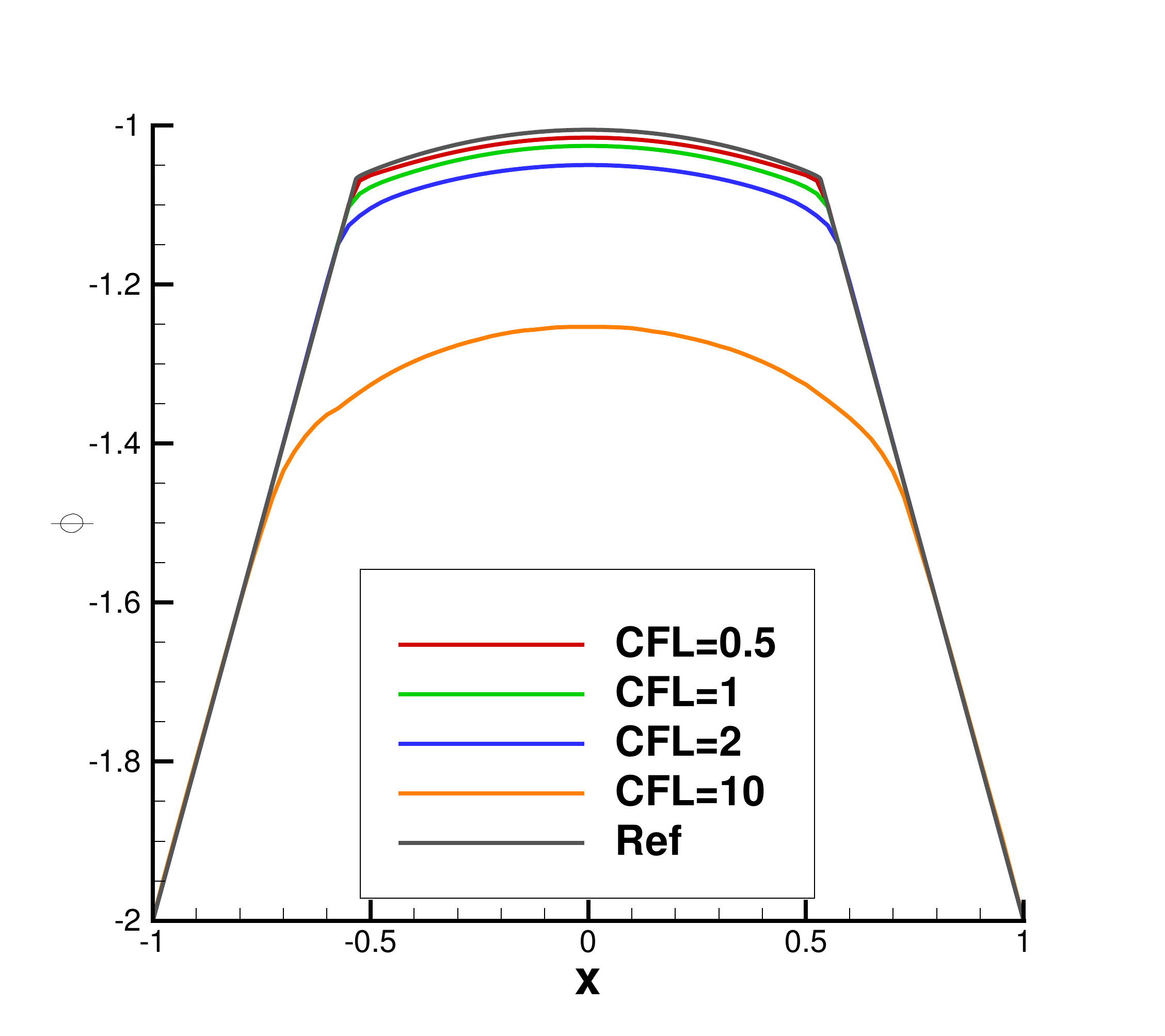}\label{Fig3.3}}
	\subfigure[Comparison. ]{
		\includegraphics[width=0.4\textwidth]{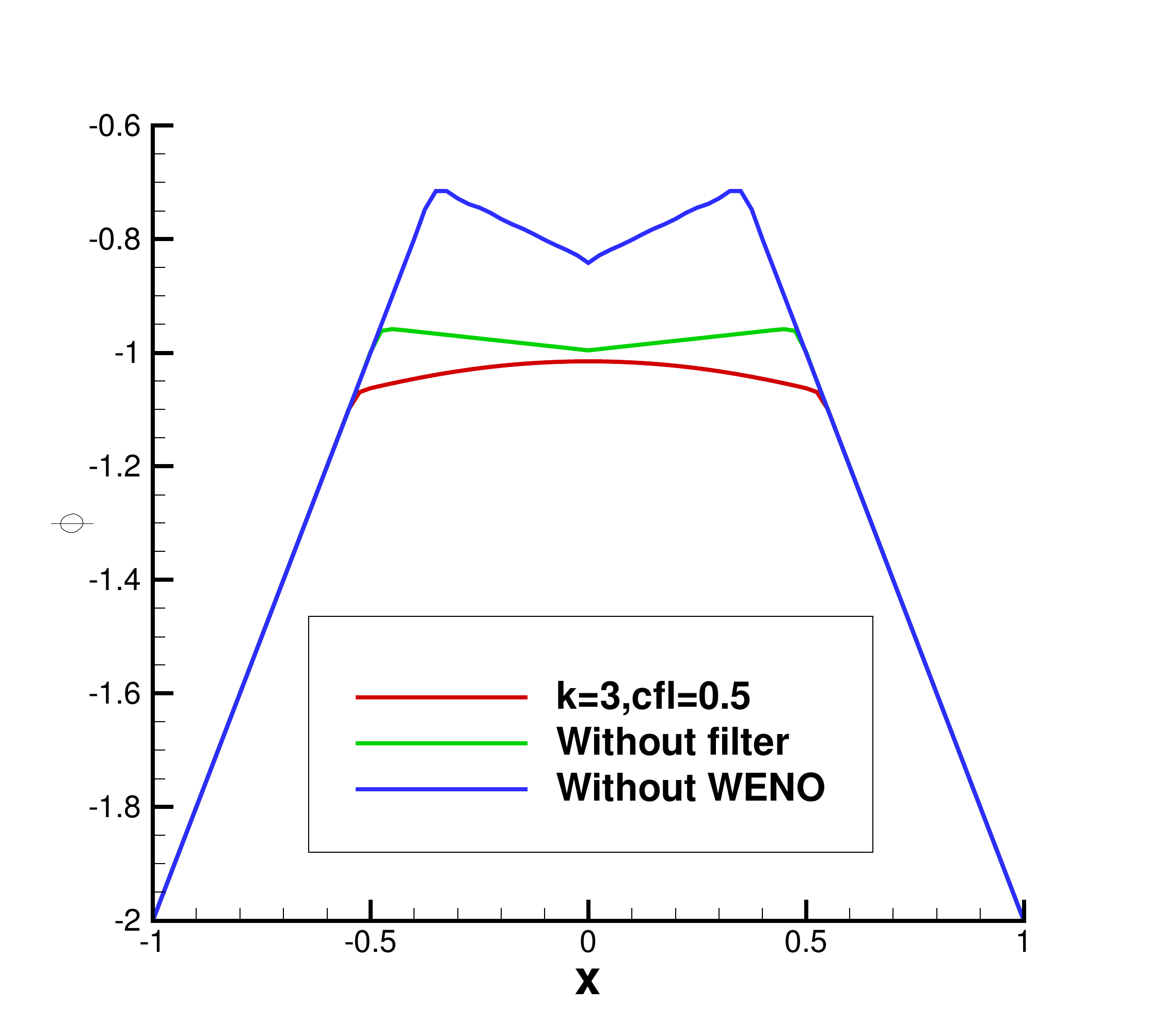}\label{Fig3.4}}
%		\subfigure[$N=800$. ]{
%		\includegraphics[width=0.3\textwidth]{figure/1D_Riemann_N=800-eps-converted-to.pdf}\label{Fig3.5}}
	\caption{\em  Example \ref{ex3}: 1D Riemann problem with a non-convex Hamiltonian. $T=1$.  $N=80$ grid points. }
	\label{Fig3}
\end{figure}

\begin{exa}\label{ex4}
 Now, we consider the 2D Burgers' equation
\begin{align}
\left\{\begin{array}{ll}
\phi_{t}+\frac{1}{2}(\phi_{x}+\phi_{y}+1)^2=0, & -2\leq x,y \leq 2,\\
\phi^{0}(x,y)=-\cos(\pi(x+y)/2),\\
\end{array}
\right.
\end{align}
with a $4$-periodic boundary condition in each direction. When $T=0.5/\pi^2$, the solution is smooth. We report the $L_\infty$ error and orders of accuracy in Table \ref{tab3} on uniform meshes. It is observed that our schemes can achieve the designed order accuracy. In Figure \ref{Fig4}, we plot the numerical solutions at time $T=1.5/\pi^2$, when discontinuous derivative has developed. Again, our schemes are able to capture the non-smooth structures very well and generate high quality numerical results.
\end{exa}

\begin{table}[htb]
	\caption{\label{tab3}\em Example \ref{ex4}: $L_\infty$ errors and orders of accuracy for 2D Burgers' equation. $T=0.5/\pi^2$. }
	\centering
	\bigskip
	\begin{small}
		\begin{tabular}{|c|c|cc|cc|cc|}
			\hline
			\multirow{2}{*}{CFL} &  \multirow{2}{*}{$N_x\times N_y$} & \multicolumn{2}{c|}{$k=1$. $\beta=1$.} & \multicolumn{2}{c|}{$k=2$. $\beta=0.5$.} & \multicolumn{2}{c|}{$k=3$. $\beta=0.6$.}\\
			\cline{3-8}
			& &  error &   order  &  error &  order  &  error  & order  \\\hline	
			\multirow{5}{*}{0.5}  
			&   $20\times20$  &  1.573E-02  &      --      &  1.772E-02  &      --     &  5.397E-04  &     --     \\
			&   $40\times40$  &  7.939E-03  &   0.987  &  4.793E-03  &   1.887  &  4.399E-05  &  3.617  \\
			&   $80\times80$  &  4.087E-03  &   0.958  &  1.287E-03  &   1.897  &  4.631E-06  &  3.248  \\
			& $160\times160$  &  2.072E-03  &   0.980  &  3.309E-04  &   1.959  &  5.226E-07  &  3.148  \\
			& $320\times320$  &  1.046E-03  &   0.987  &  8.468E-05  &   1.966  &  6.014E-08  &   3.119 \\\hline                   	
			\multirow{5}{*}{1} 
			&	 $20\times20$  &  3.139E-02  &      --     &  6.300E-02  &       --     &  4.022E-03  &      --     \\
			&   $40\times40$  &  1.558E-02  &   1.011  &  1.803E-02  &   1.805  &  4.995E-04  &  3.010  \\
			&   $80\times80$  &  7.936E-03  &   0.973  &  4.834E-03  &   1.899  &  4.953E-05  &  3.334  \\
			& $160\times160$  &  4.122E-03  &  0.945  &  1.287E-03  &  1.909  &  5.000E-06  &  3.308  \\
			& $320\times320$  &  2.072E-03  &  0.992  &  3.310E-04  &  1.959  &  5.314E-07  &  3.234  \\\hline 
			\multirow{4}{*}{2}  
			&   $20\times20$  &  9.137E-02  &    --    &  1.916E-01  &    --    &  3.366E-02  &    --    \\
			&   $40\times40$  &  3.201E-02  &   1.513  &  6.448E-02  &   1.571  &  4.455E-03  &  2.917  \\
			&   $80\times80$  &  1.562E-02  &   1.035  &  1.811E-02  &   1.832  &  5.069E-04  &  3.136  \\
			& $160\times160$  &  8.021E-03  &  0.962  &  4.834E-03  &  1.905  &  4.987E-05  &  3.345  \\
			& $320\times320$  &  4.123E-03  &  0.960  &  1.287E-03  &  1.909  &  5.009E-06  &  3.316  \\\hline   
		\end{tabular}
	\end{small}
\end{table}

\begin{figure}
	\centering
	\subfigure[Surface. CFL=0.5. ]{
		\includegraphics[width=0.4\textwidth]{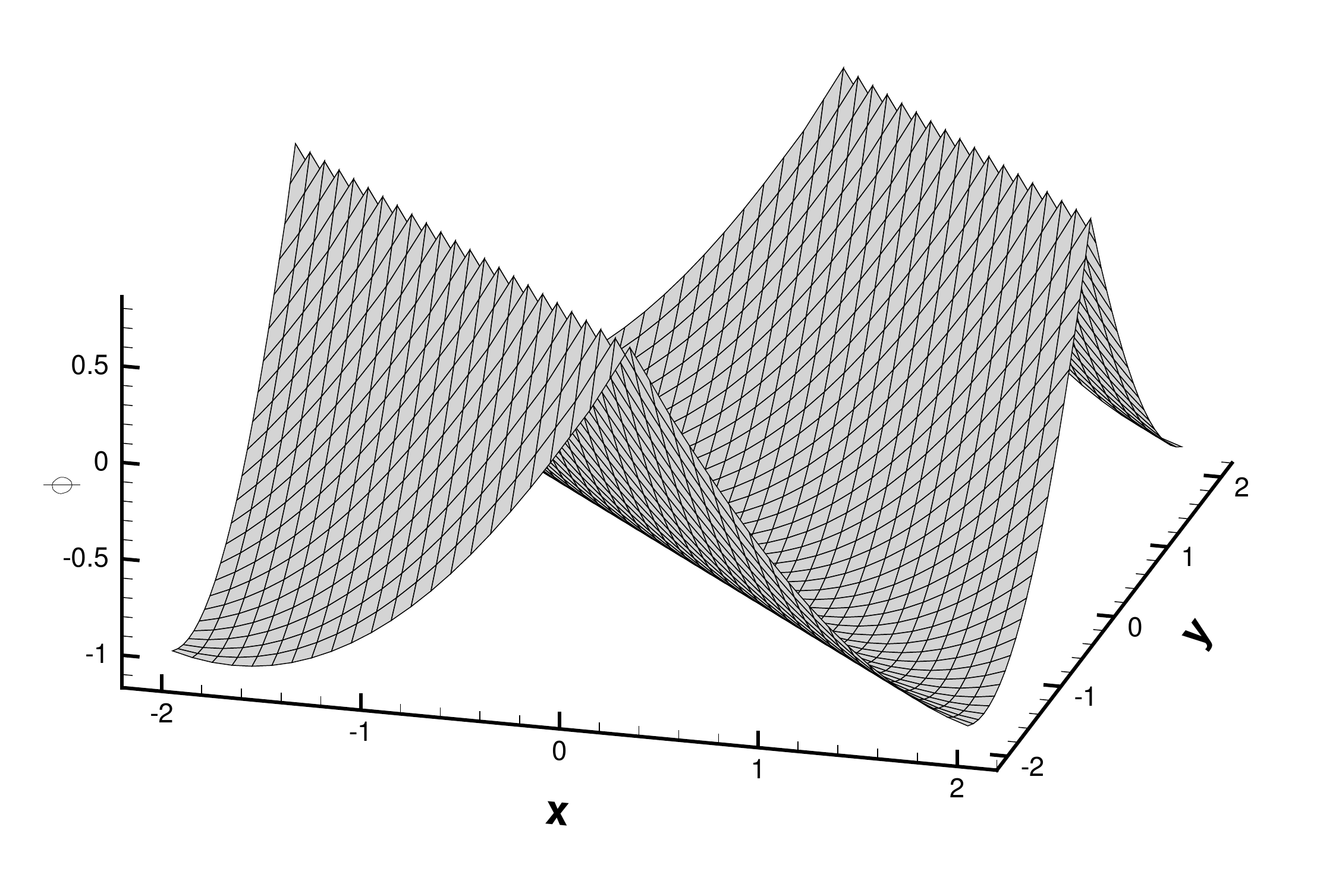}\label{Fig4.1}}
	\subfigure[Surface. CFL=2. ]{
		\includegraphics[width=0.4\textwidth]{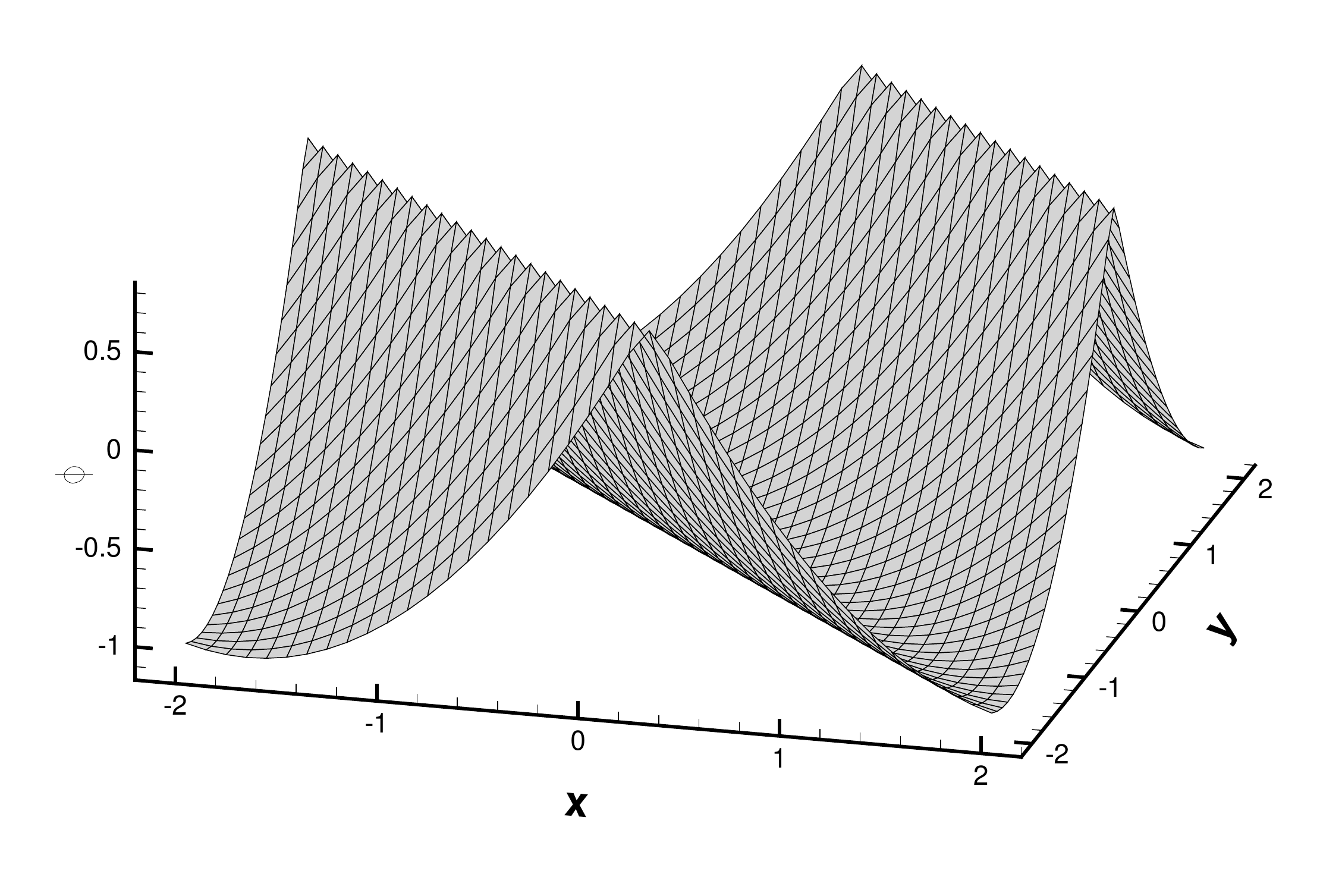}\label{Fig4.4}}
	\caption{\em  Example \ref{ex4}: 2D Burgers' equation. $T=1.5/\pi^2$. $40\times40$ grid points. $k=3$ with $\beta=0.6$. }
	\label{Fig4}
\end{figure}

\begin{exa}\label{ex5}
In this example, we solve the following 2D HJ equation with a non-convex Hamiltonian
\begin{align}
\left\{\begin{array}{ll}
\phi_{t}-\cos(\phi_{x}+\phi_{y}+1)=0, & -2\leq x,y\leq 2,\\
\phi^{0}(x,y)=-\cos(\pi (x+y)/2),\\
\end{array}
\right.
\end{align}
with a $4$-periodic boundary condition in each direction.
In table \ref{tab4},  we report the $L_\infty$ error and the orders of accuracy to demonstrate that the proposed schemes can achieve the designed order accuracy if the solution is smooth. In Figure \ref{Fig5}, the numerical solutions at $T=1.5/\pi^2$ are plotted. As with the former example, the schemes are observed to be able to capture the non-smooth solution structures without producing noticeable oscillations.
\end{exa}

\begin{table}[htb]
	\caption{\label{tab4} \em Example \ref{ex5}: Errors and orders of accuracy for 2D equation with a non-convex Hamiltonian. $T=0.5/\pi^2$.  }
	\centering
	\bigskip
	\begin{small}
		\begin{tabular}{|c|c|cc|cc|cc|}
			\hline
			\multirow{2}{*}{CFL} &  \multirow{2}{*}{$N_x\times N_y$} & \multicolumn{2}{c|}{$k=1$. $\beta=1$.} & \multicolumn{2}{c|}{$k=2$. $\beta=0.5$.} & \multicolumn{2}{c|}{$k=3$. $\beta=0.6$.}\\
			\cline{3-8}
			& &  error &   order  &  error &  order  &  error  & order  \\\hline
			\multirow{4}{*}{0.5}  
			&	 $40\times40$  &  1.917E-03  &       --     &  1.009E-03  &      --     &  3.379E-05  &      --      \\
			&	 $80\times80$  &  1.025E-03  &   0.902  &  2.594E-04  &   1.960  &  4.725E-06  &  2.838  \\
			&  $160\times160$  &  5.022E-04  &   1.030  &  6.667E-05  &   1.960  &  6.176E-07  &  2.935  \\
			&  $320\times320$  &  2.501E-04  &   1.006  &  1.705E-05  &   1.967  &  7.341E-08  &  3.073  \\\hline              
			\multirow{4}{*}{1}  
			&	 $40\times40$  &  6.198E-03  &      --      &  4.130E-03  &      --     &  2.842E-04  &       --     \\
			&	 $80\times80$  &  2.095E-03  &   1.565  &  1.023E-03  &   2.013  &  4.257E-05  &  2.739  \\
			&  $160\times160$  &  1.017E-03  &   1.043  &  2.594E-04  &   1.979  &  5.425E-06  &  2.972  \\
			&  $320\times320$  &  5.022E-04  &   1.018  &  6.668E-05  &   1.960  &  6.373E-07  &  3.090  \\\hline  
			\multirow{4}{*}{2} 
			&	 $40\times40$  &  6.410E-03  &       --     &  4.293E-03  &      --     &  2.978E-04  &      --      \\
			&	 $80\times80$  &  6.207E-03  &   0.046  &  4.172E-03  &   0.041  &  3.078E-04  & -0.048  \\
			&  $160\times160$  &  2.091E-03  &   1.570  &  1.023E-03  &   2.028  &  4.313E-05  &  2.835  \\
			&  $320\times220$  &  1.017E-03  &   1.040  &  2.595E-04  &   1.979  &  5.439E-06  &  2.987  \\\hline    
		\end{tabular}
	\end{small}
\end{table}

\begin{figure}
	\centering
	\subfigure[Surface. CFL=0.5. ]{
		\includegraphics[width=0.4\textwidth]{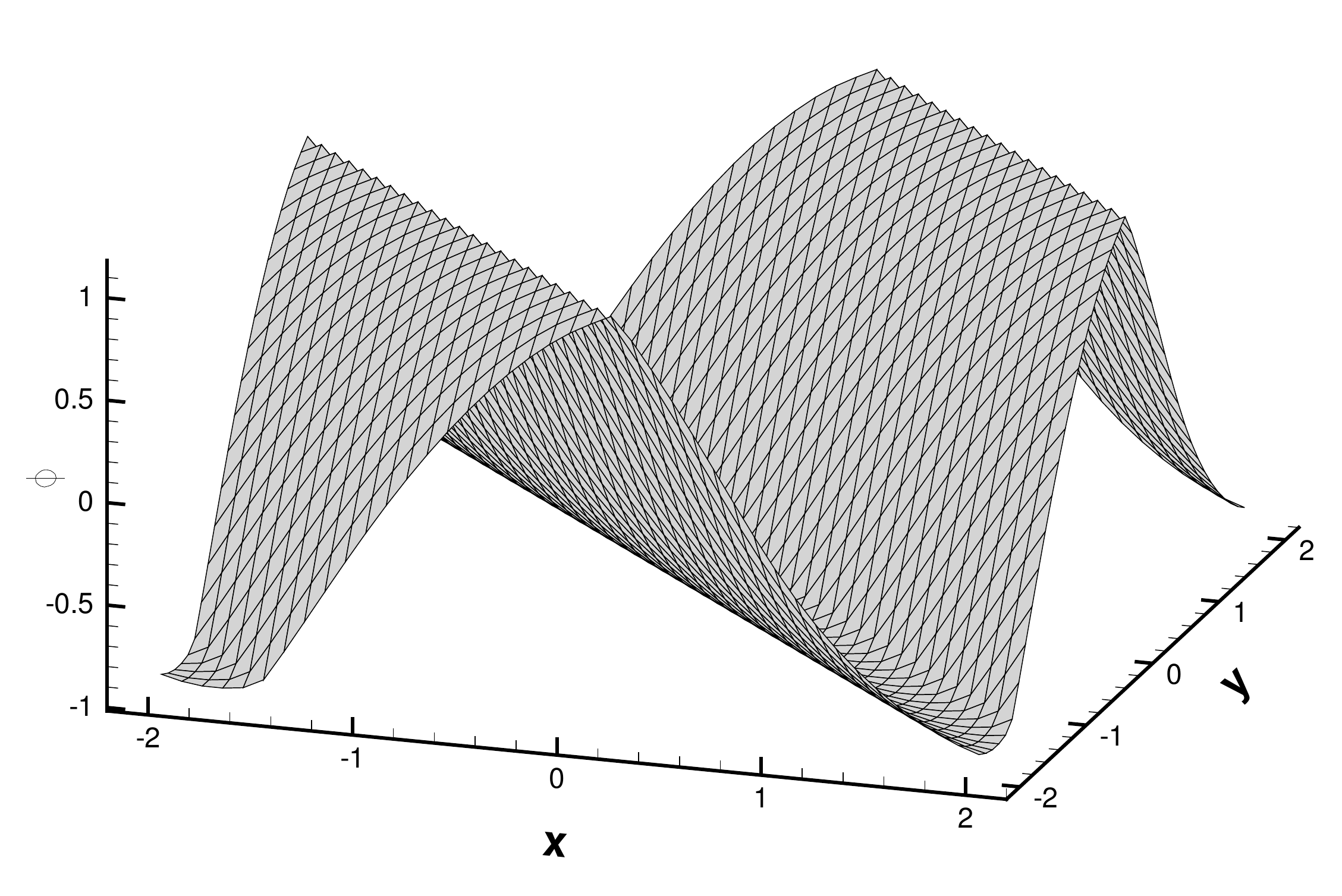}\label{Fig5.2}}
	\subfigure[Surface. CFL=2. ]{
		\includegraphics[width=0.4\textwidth]{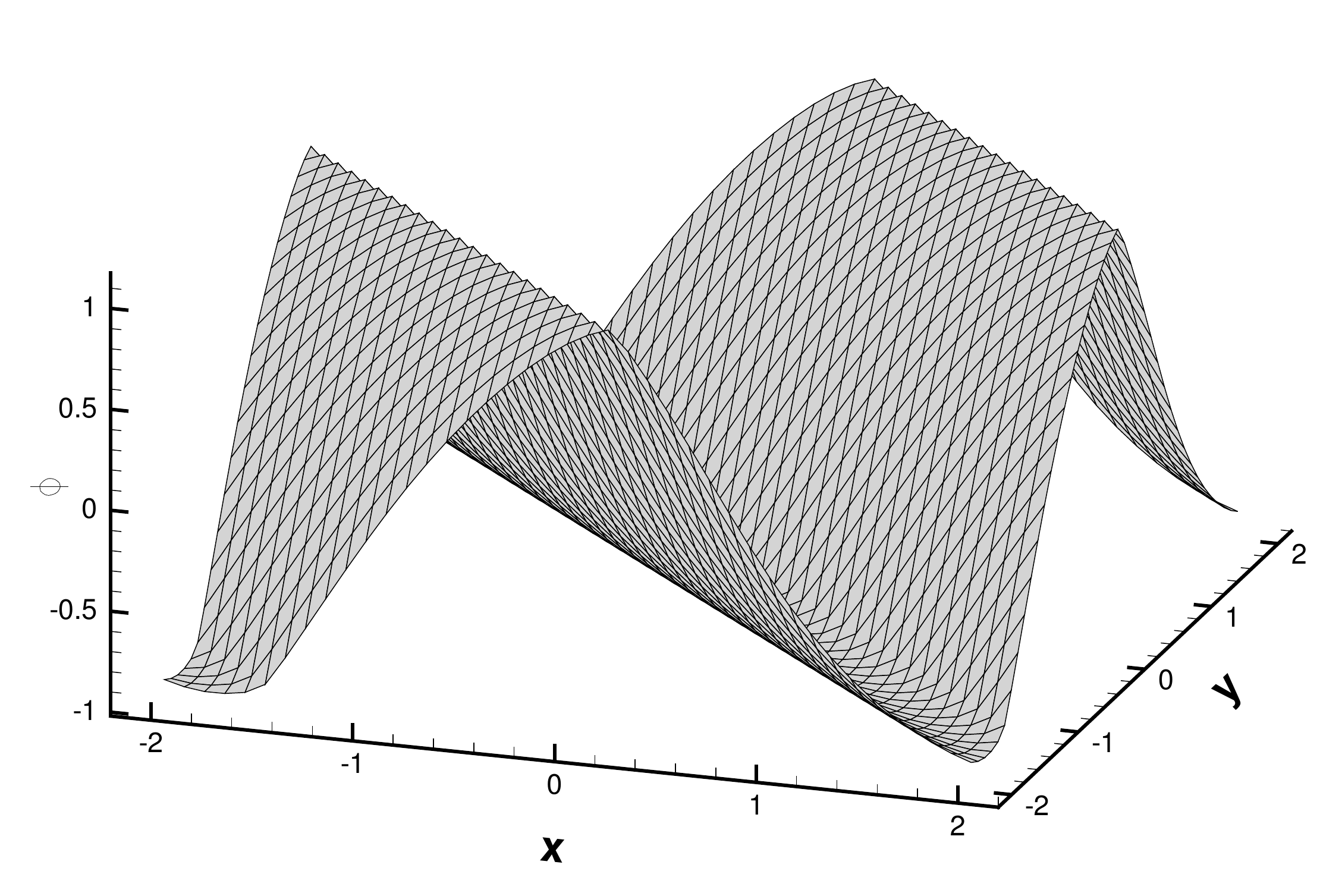}\label{Fig5.4}}
	\caption{\em  Example \ref{ex5}: 2D equation with a non-convex Hamiltonian. $T=1.5/\pi^2$. $40\times40$ grid points. $k=3$ with $\beta=0.6$. }
	\label{Fig5}
\end{figure}

\begin{exa}\label{ex6}
 We consider the 2D Riemann problem with a non-convex Hamiltonian
\begin{align}
\left\{ \begin{array}{ll}
\phi_{t} + \sin( \phi_{x}+\phi_{y} )=0, & -1\leq x,y,\leq 1\\
\phi(x,y,0)=\pi (|y|-|x|)\\
\end{array}
\right.
\end{align}
and an outflow boundary condition is imposed. 
We compute the solutions up to $T=1$. A uniform mesh with $N_{x}=N_{y}=80$ is used. We plot the numerical solutions in Figure \ref{Fig6}. It is observed that the proposed schemes are able to capture the viscosity solution accurately. In particular, the numerical result agrees well with those computed by other methods, see e.g., \cite{hu1999discontinuous,zhang2003high}. We remark that in this example, we should incorporate the robust WENO quadrature as well as the nonlinear filter to ensure the convergence towards the viscosity solution. Otherwise, the rarefaction wave is missing, and the numerical solution exhibits spurious oscillations, see Figure \ref{Fig6.5}-\ref{Fig6.6}.
\end{exa}

\begin{figure}
	\centering
	\subfigure[Contour. CFL=0.5. ]{
		\includegraphics[width=0.4\textwidth]{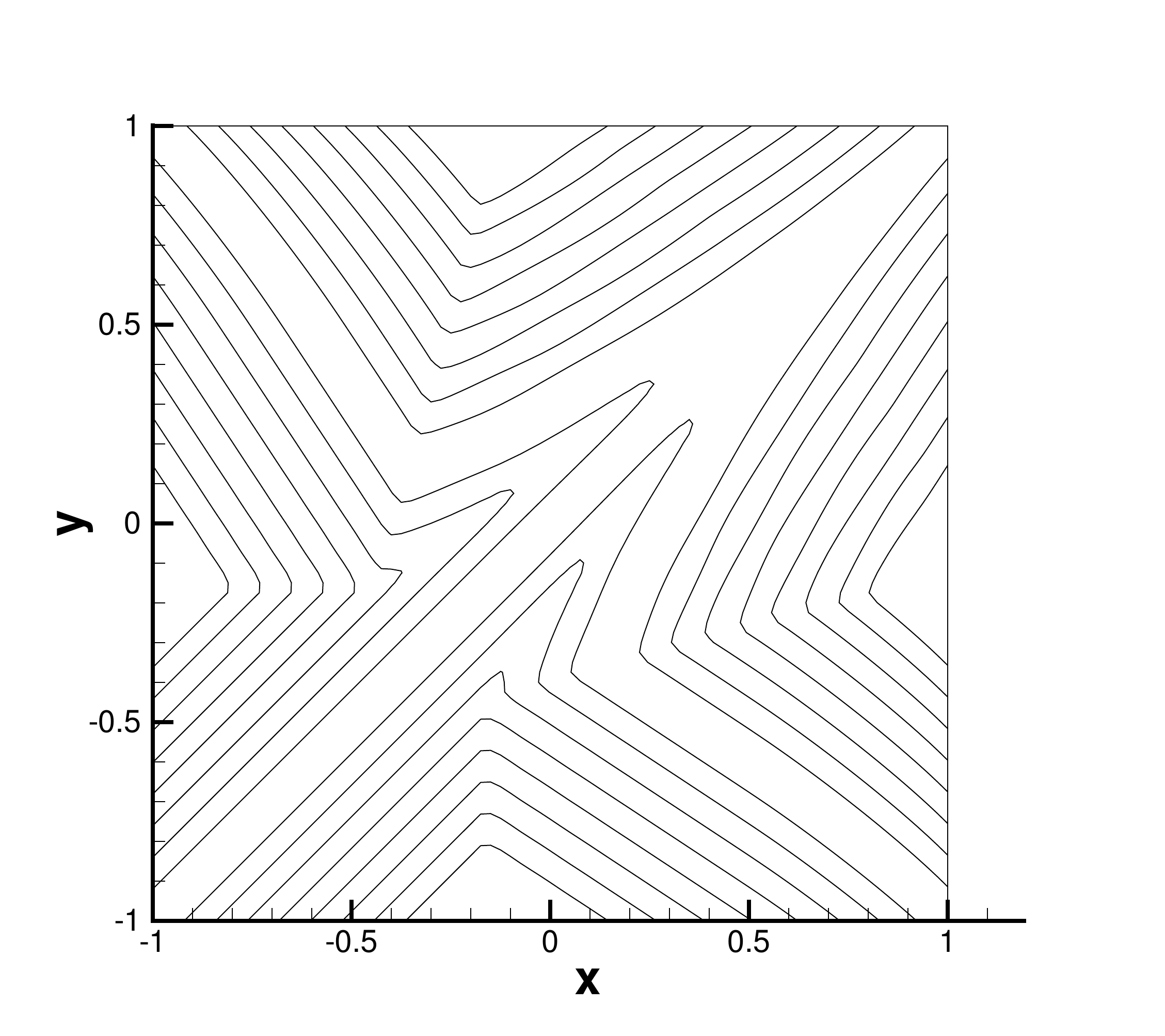}\label{Fig6.1}}
	\subfigure[Surface. CFL=0.5. ]{
		\includegraphics[width=0.4\textwidth]{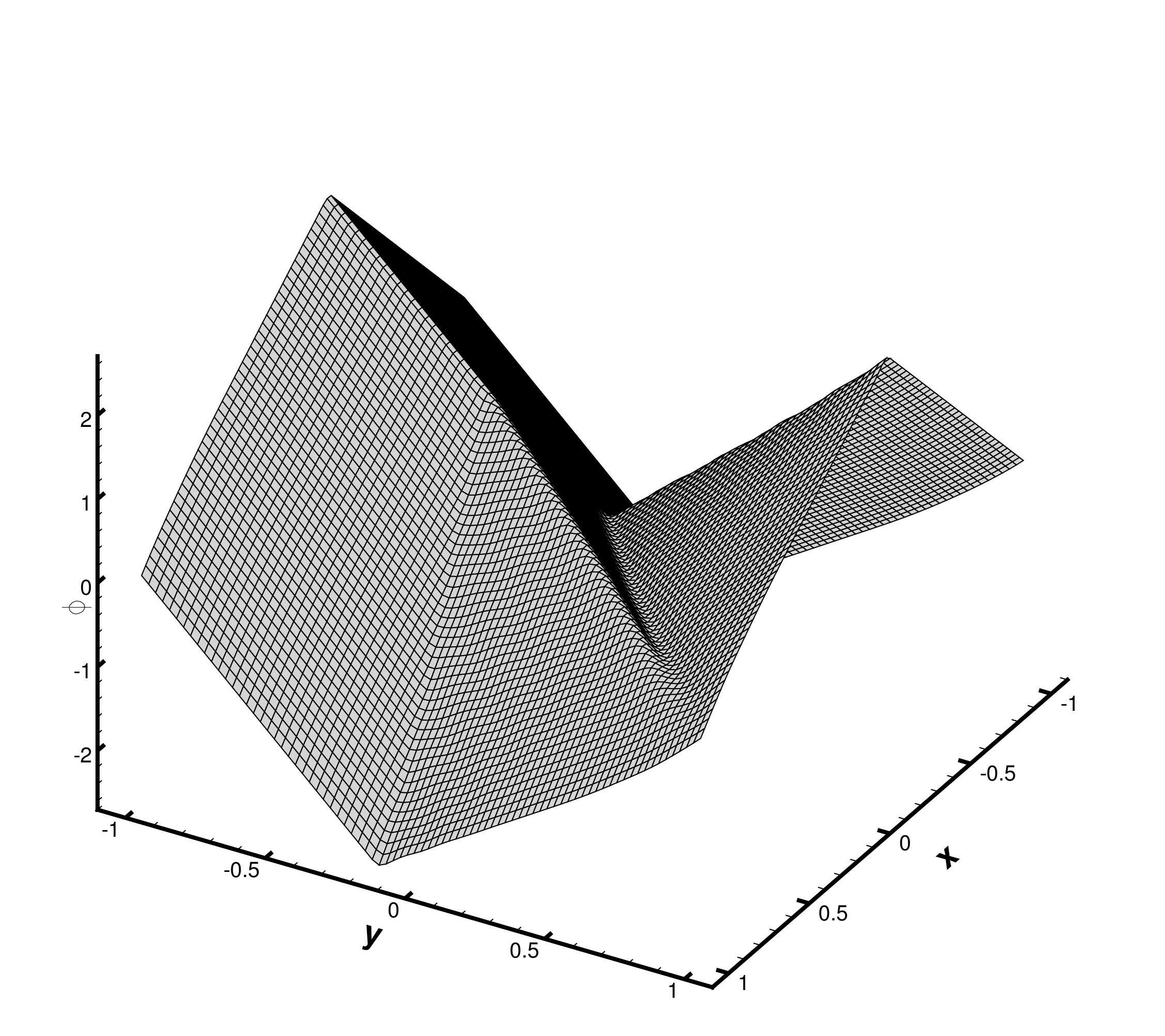}\label{Fig6.2}}
	\subfigure[Contour. CFL=2. ]{
		\includegraphics[width=0.4\textwidth]{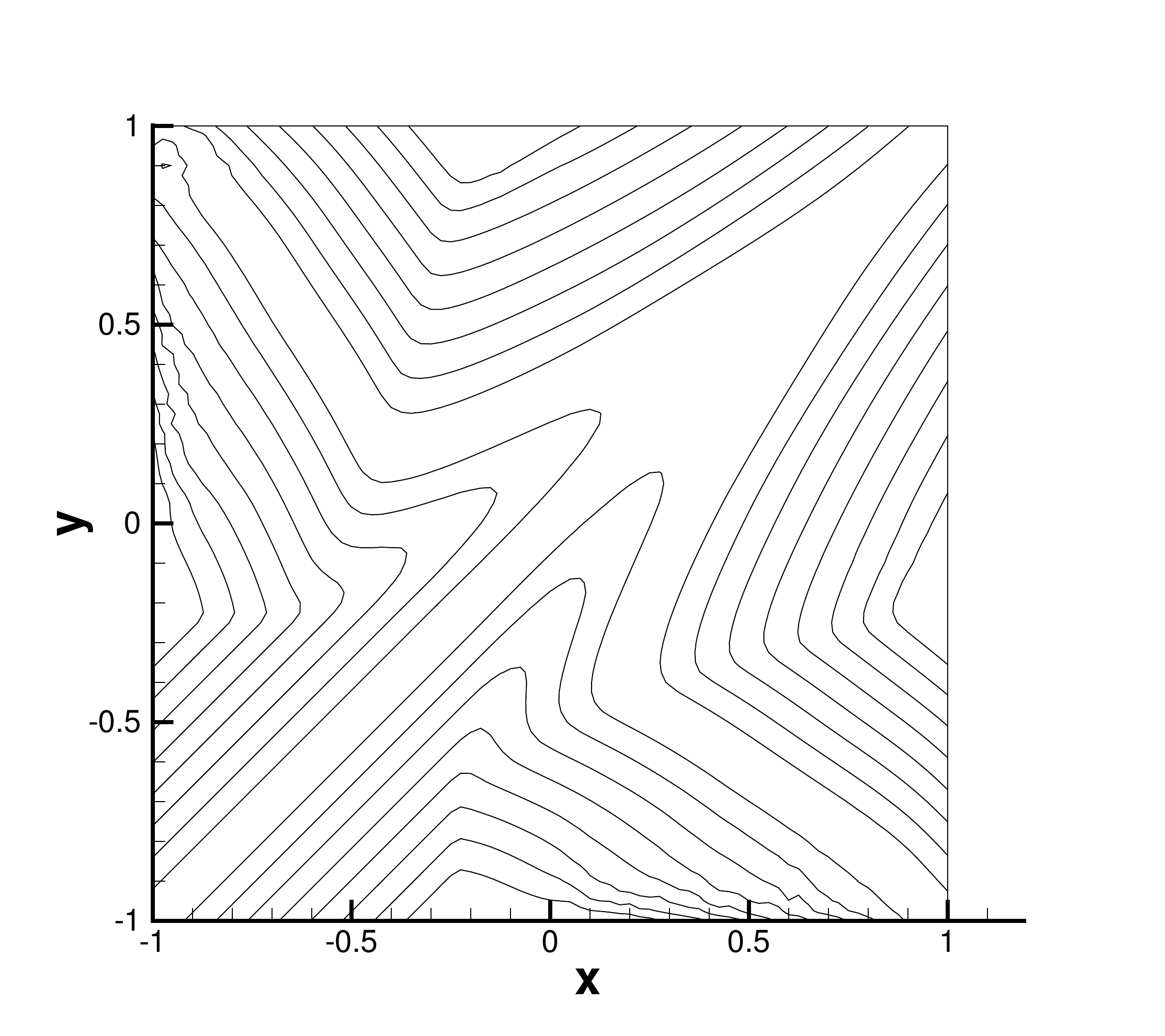}\label{Fig6.3}}
	\subfigure[Surface. CFL=2. ]{
		\includegraphics[width=0.4\textwidth]{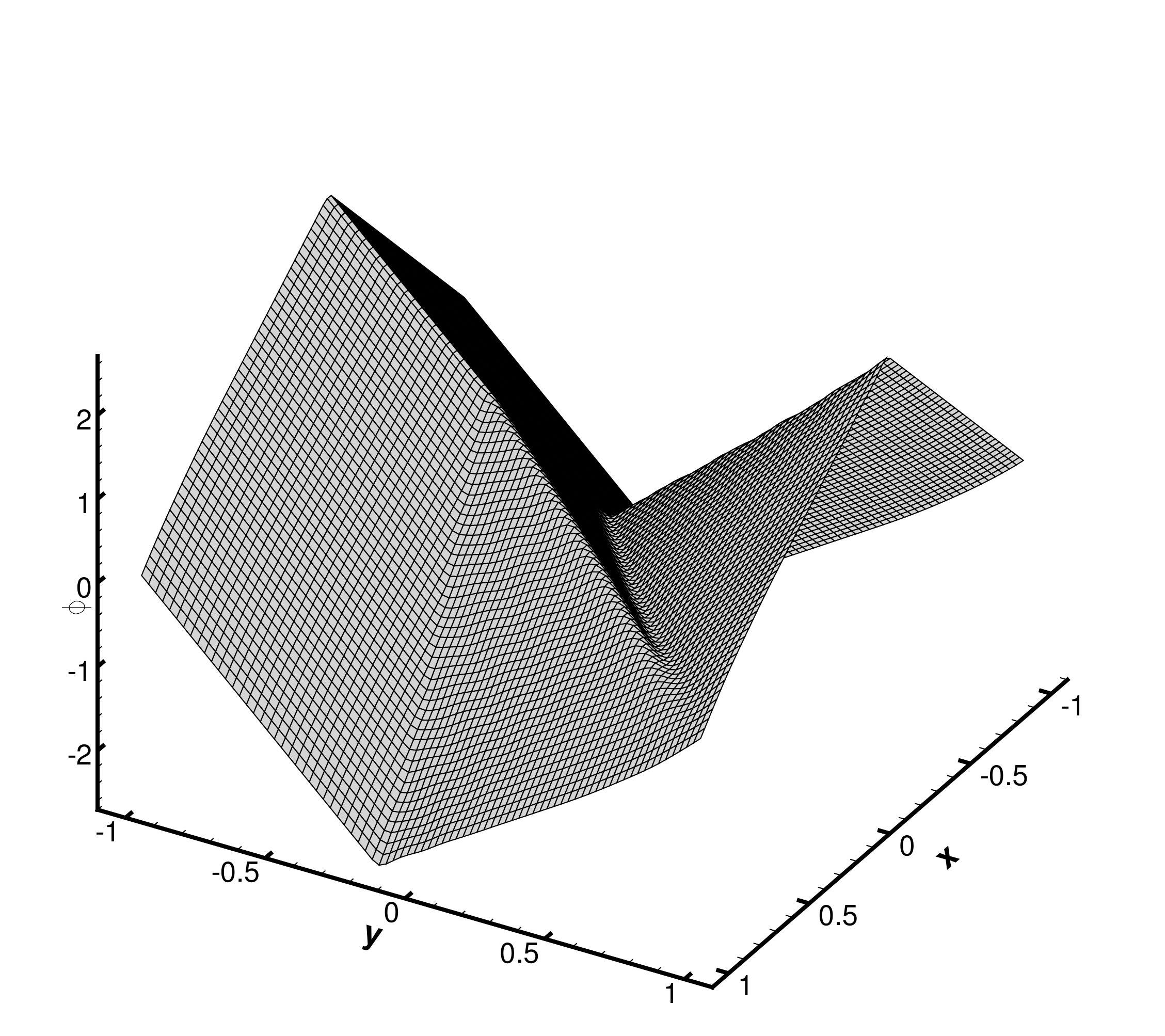}\label{Fig6.4}}
	\subfigure[Contour. Without WENO quadrature.]{
		\includegraphics[width=0.4\textwidth]{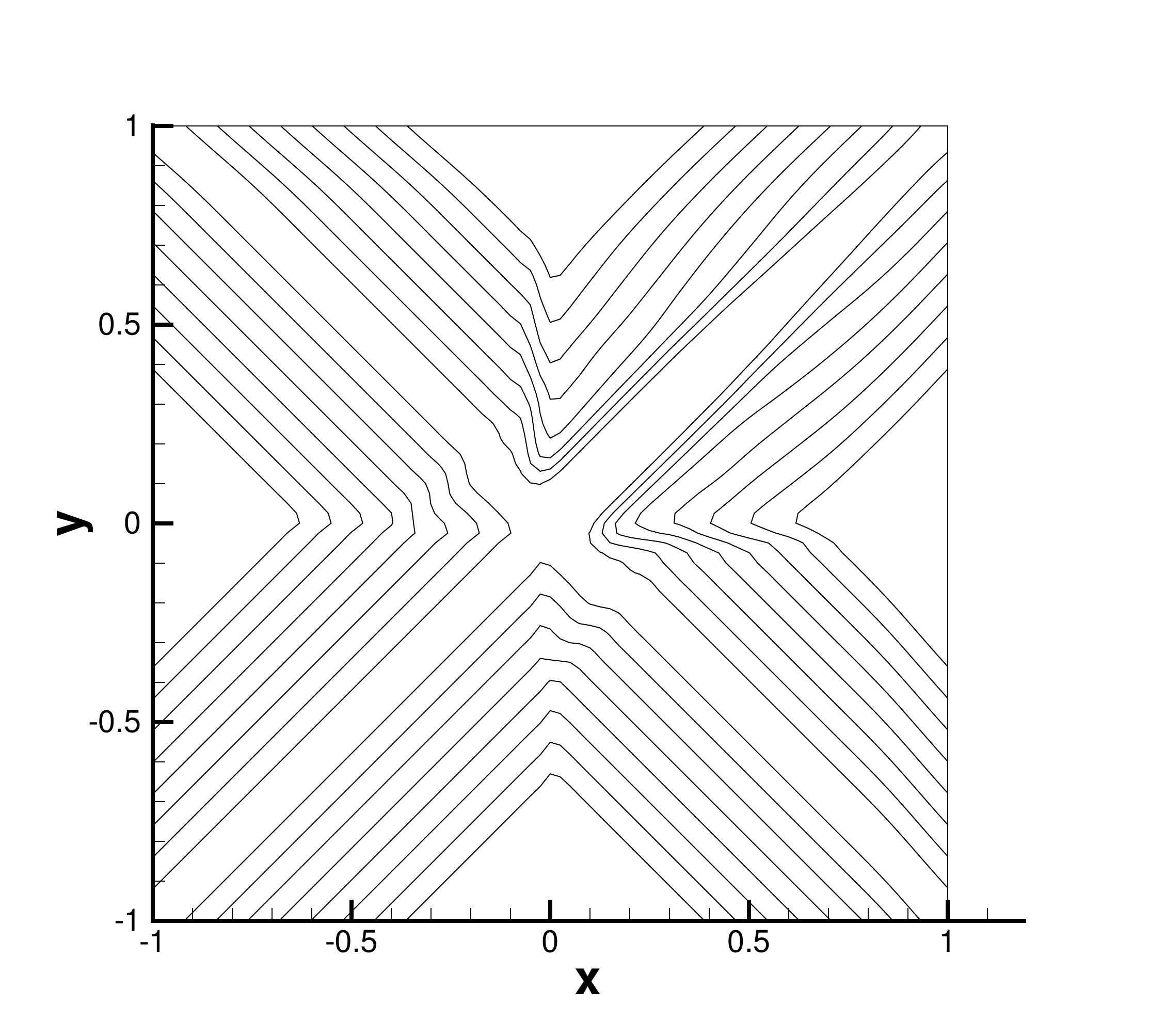}\label{Fig6.5}}
	\subfigure[Surface. Without WENO quadrature.]{
		\includegraphics[width=0.4\textwidth]{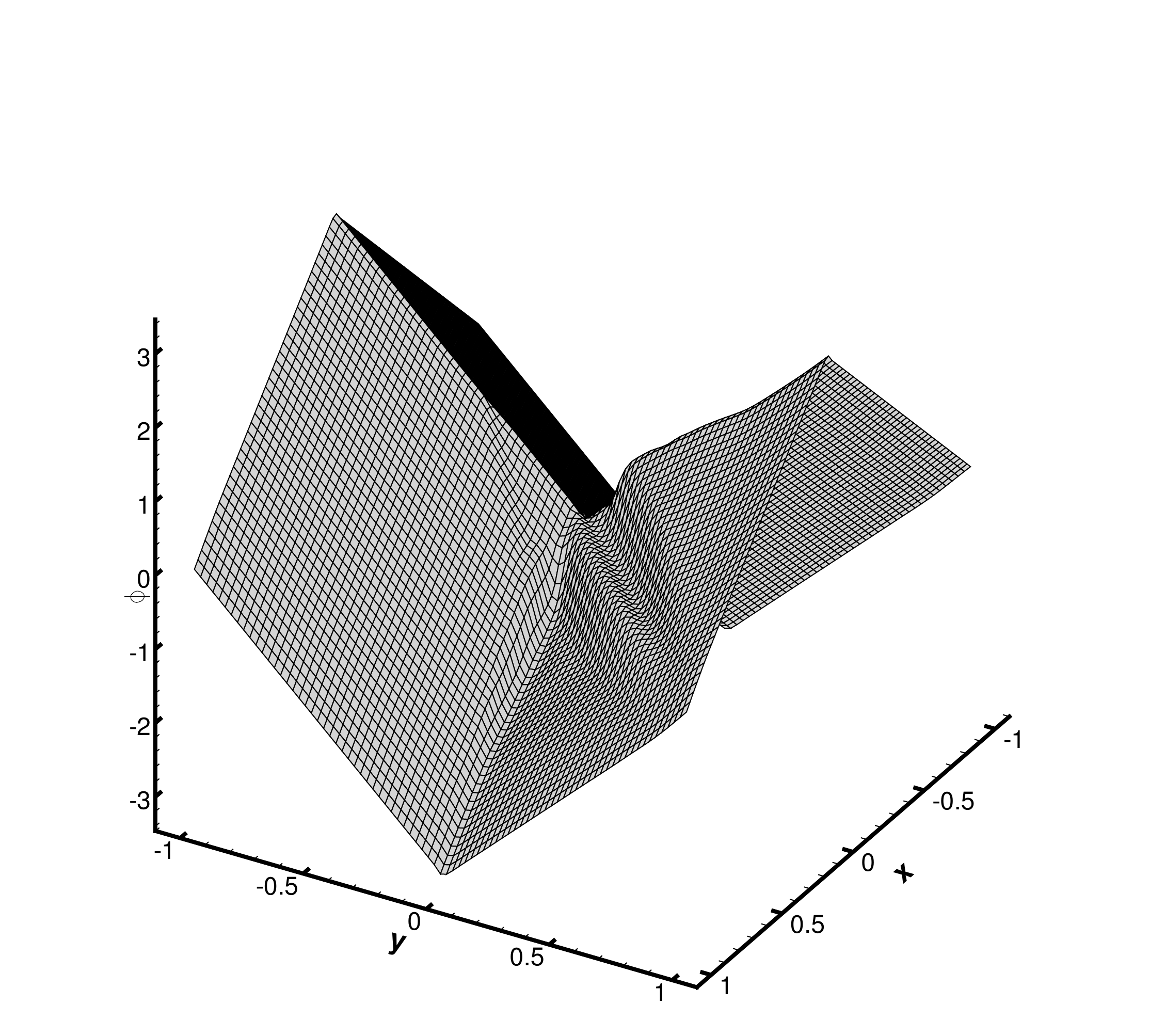}\label{Fig6.6}}
	\caption{\em  Example \ref{ex6}: 2D Riemann problem with a non-convex Hamiltonian. $T=1$. $80\times80$ grid points. $k=3$ with $\beta=0.6$. }
	\label{Fig6}
\end{figure}

\begin{exa}\label{ex7}
We consider solving the following problem from optimal control:
\begin{align}
\left\{ \begin{array}{ll}
\phi_{t} + \sin(y) \phi_{x}+(\sin(x)+sign(\phi_{y}))\phi_{y}-\frac{1}{2}\sin(y)^2-(1-\cos(x)) =0, & -\pi\leq x,y,\leq \pi\\
\phi^{0}(x,y)=0,\\
\end{array}
\right.
\end{align}
with periodic boundary conditions. We plot the numerical solutions at $T=1$ in Figure \ref{Fig7}. A uniform mesh with $N_{x}=N_{y}=60$ is used.  The optimal control $sign(\phi_y)$ which contains the most interesting information in the optimal control problem is shown as well. Again, our numerical schemes are able to  resolve the non-smooth solution structures very nicely.
\end{exa}

\begin{figure}
	\centering
	\subfigure[Solution. CFL=0.5. ]{
		\includegraphics[width=0.4\textwidth]{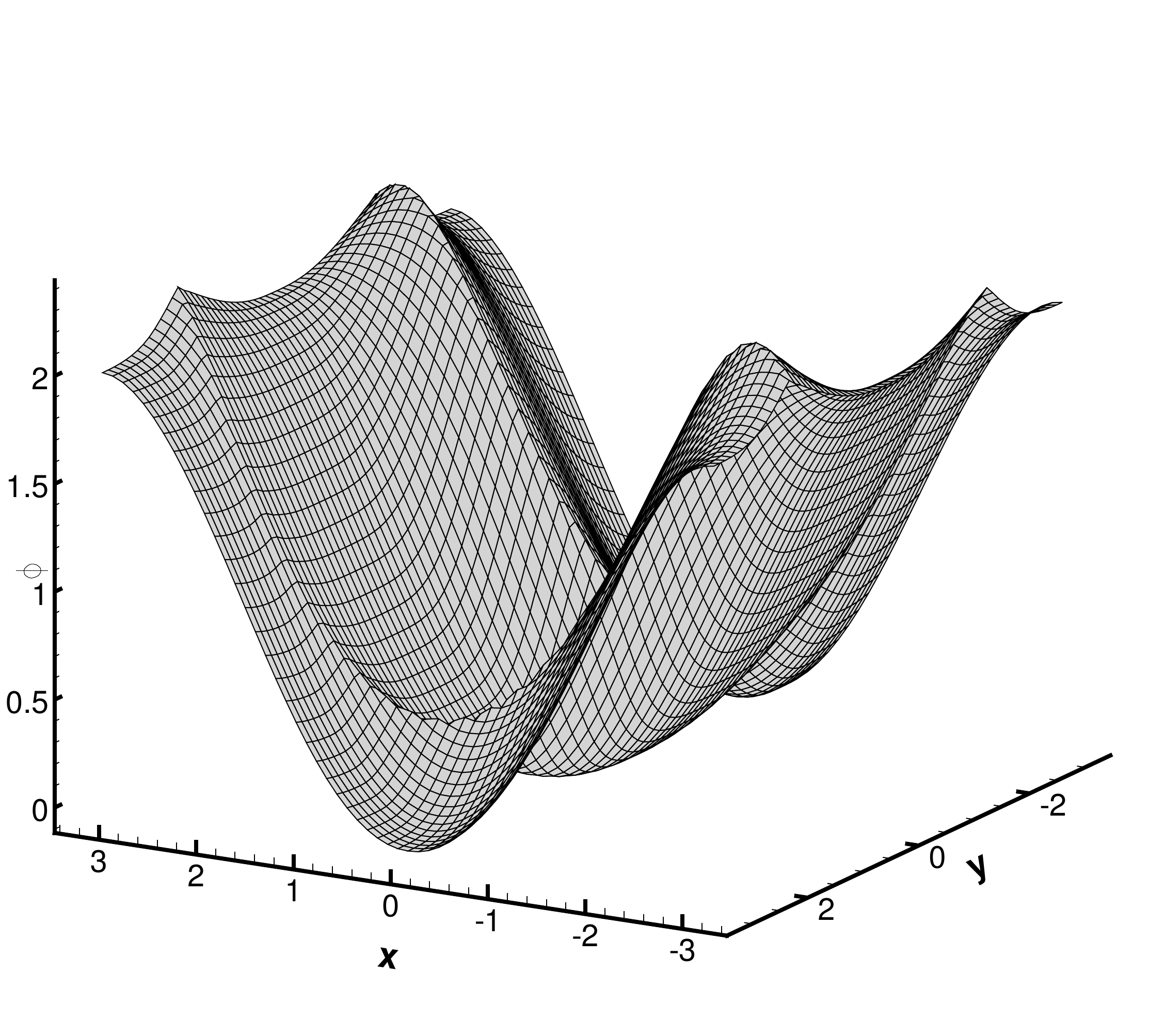}\label{Fig7.1}}
	\subfigure[Optimal control $sign(\phi_y)$. CFL=0.5. ]{
		\includegraphics[width=0.4\textwidth]{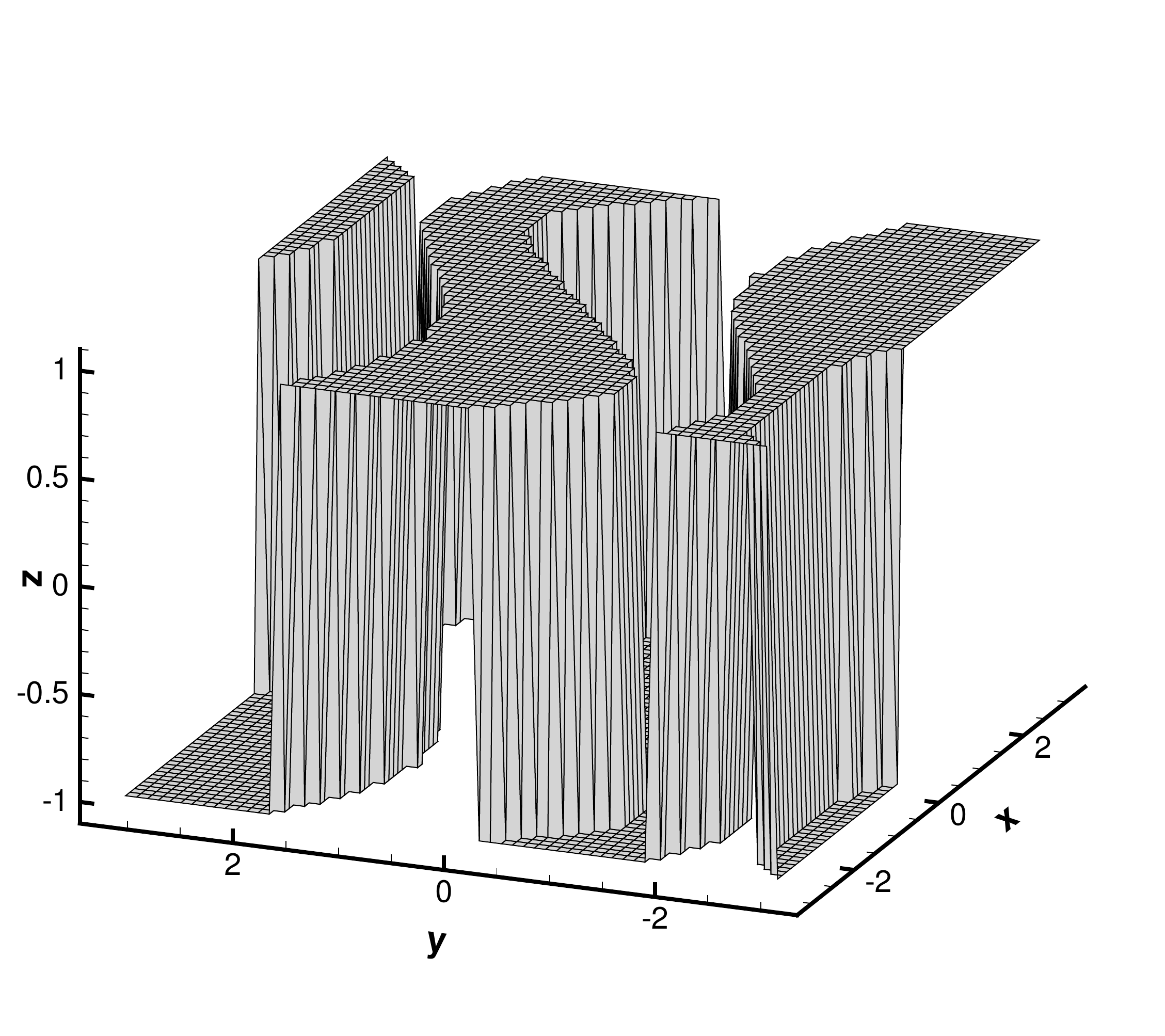}\label{Fig7.2}}
	\subfigure[Solution. CFL=2. ]{
		\includegraphics[width=0.4\textwidth]{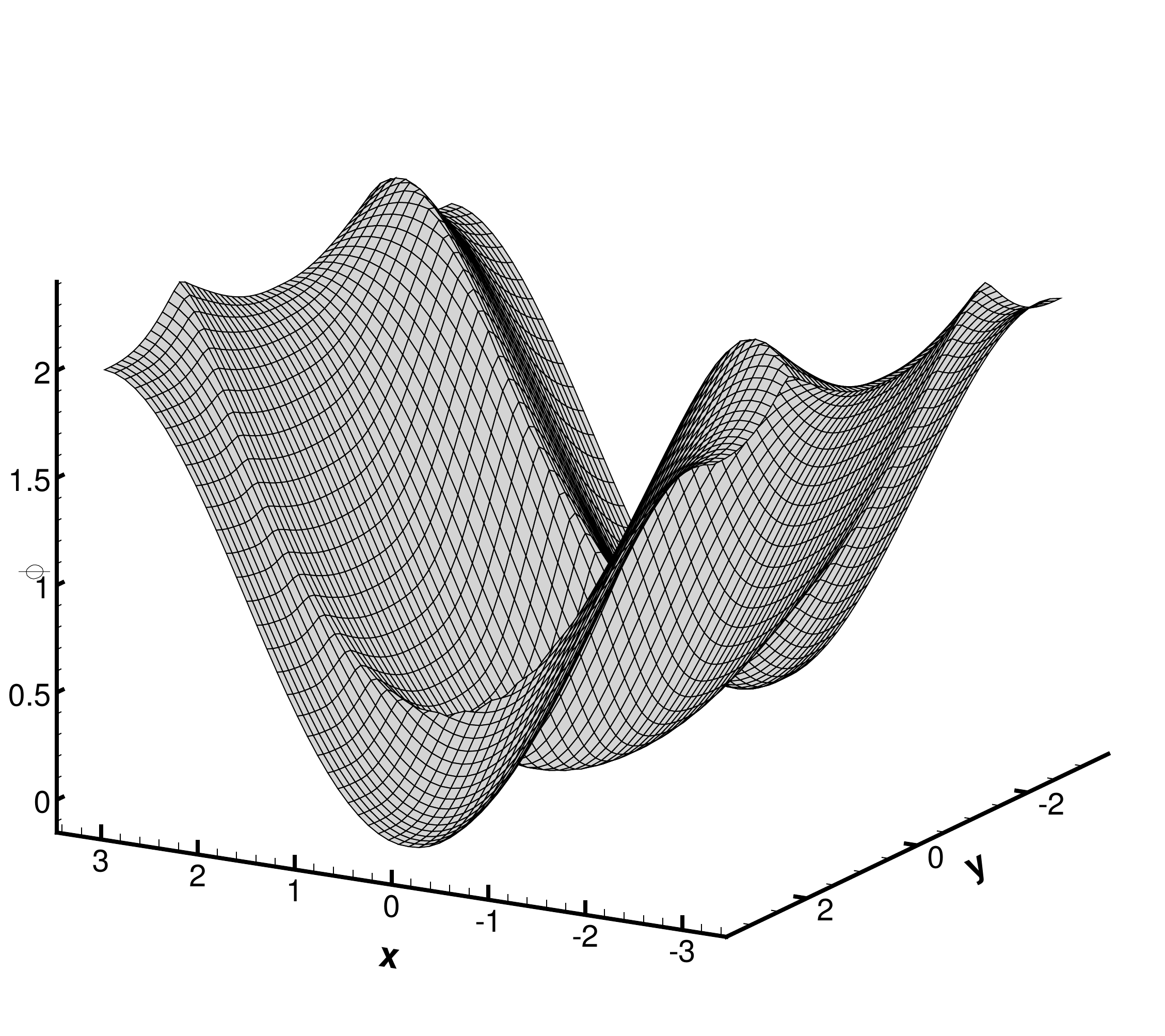}\label{Fig7.3}}
	\subfigure[Optimal control $sign(\phi_y)$. CFL=2. ]{
		\includegraphics[width=0.4\textwidth]{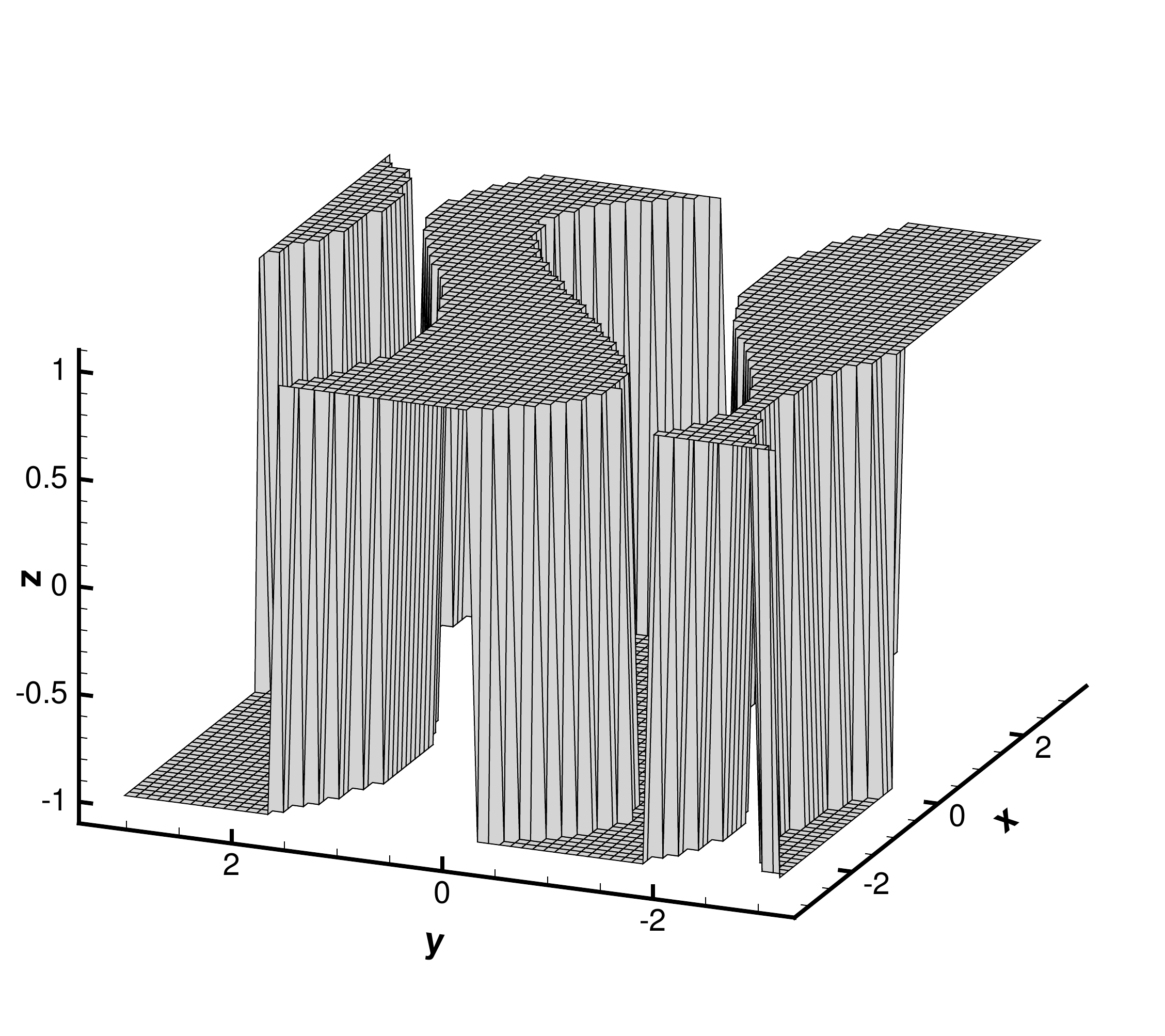}\label{Fig7.4}}
	\caption{\em  Example \ref{ex3}: 2D Control problem. $T=1$. $60\times60$ grid points. $k=3$ with $\beta=0.6$. }
	\label{Fig7}
\end{figure}

\begin{exa}\label{ex8}
In this example, we simulate a propagating surface on a rectangular domain by solving the following equation 
\begin{align}
\label{eq:propagting}
\left\{\begin{array}{ll}
\phi_{t} - (1-\epsilon K)\sqrt{\phi_{x}^2+\phi_{y}^2+1} =0, & 0\leq x, y\leq 1\\
\phi^{0}(x,y) = 1-\frac{1}{4}(\cos(2\pi x)-1)(\cos(2\pi y)-1)\\
\end{array}
\right.
\end{align}
where $K$ is the mean curvature defined by
\begin{align}
K=-\frac{\phi_{xx}(1+\phi_y^2)-2\phi_{xy}\phi_x\phi_y+\phi_{yy}(1+\phi_{x}^2)} {(\phi_{x}^2+\phi_{y}^2+1)^{3/2}},
\end{align}
and $\epsilon$ is a small constant. Periodic boundary conditions are imposed. Note that if $\epsilon\neq0$, we need to discretize the second derivatives $u_{xx}$, $u_{xy}$ and $u_{yy}$ as well. The approach proposed in our previous work \cite{christlieb2017kernel} is employed for the second derivative discretization. In the simulation, a uniform mesh with $N_{x}=N_{y}=60$ is used.
In Figures \ref{Fig8a}-\ref{Fig8b}, we plot the snapshots of numerical solutions for $\epsilon=0$ (pure convection) at $t=0,\,0.3,\,0.6,\,0.9$ and for $\epsilon=0.1$ at $t=0,\,0.1,\,0.3,\,0.6$, respectively.
\end{exa}

\begin{figure}
	\centering
	\subfigure[Propagating surface. CFL=0.5. ]{
		\includegraphics[width=0.4\textwidth]{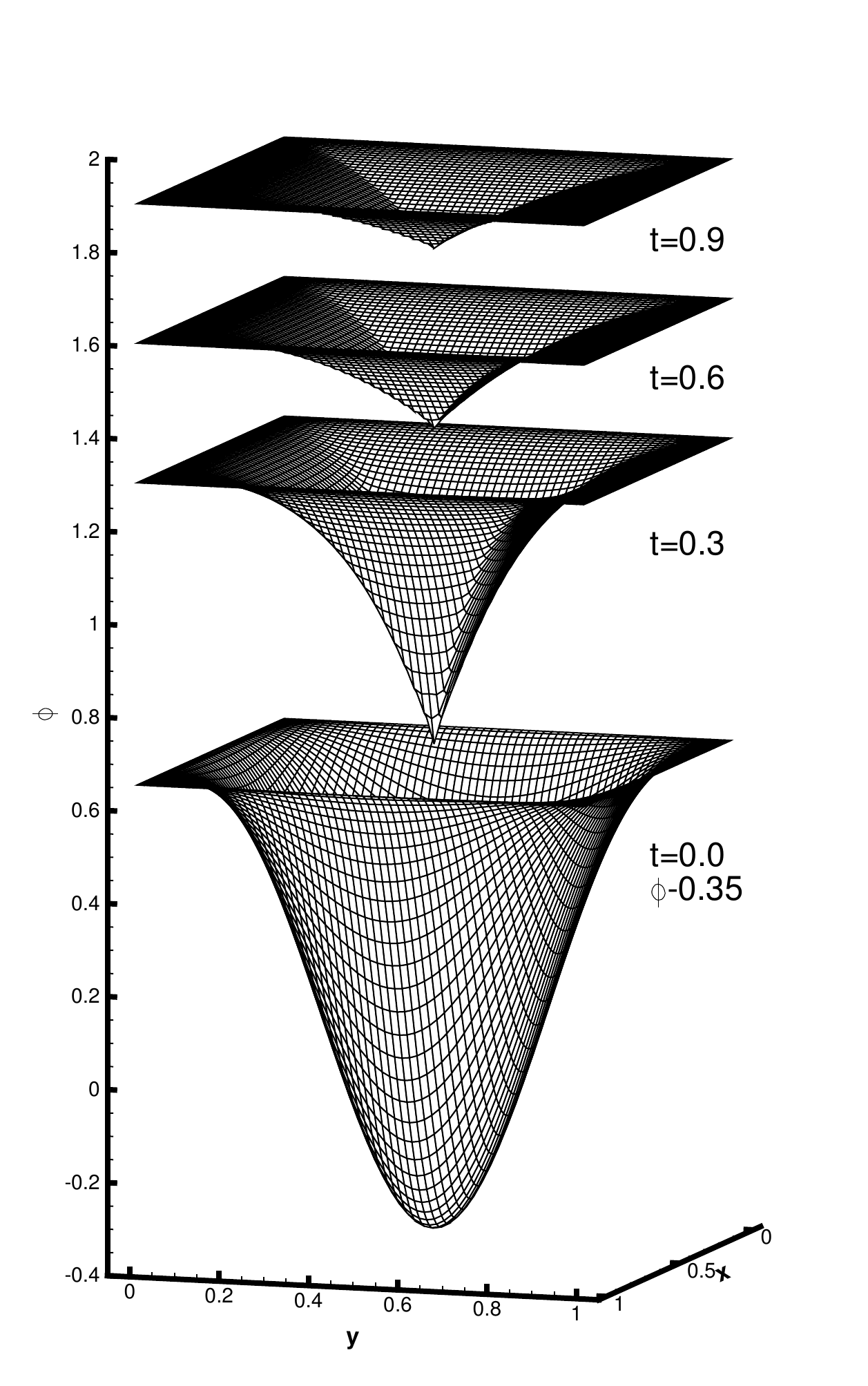}\label{Fig8a.1}}
	\subfigure[Propagating surface. CFL=2. ]{
		\includegraphics[width=0.4\textwidth]{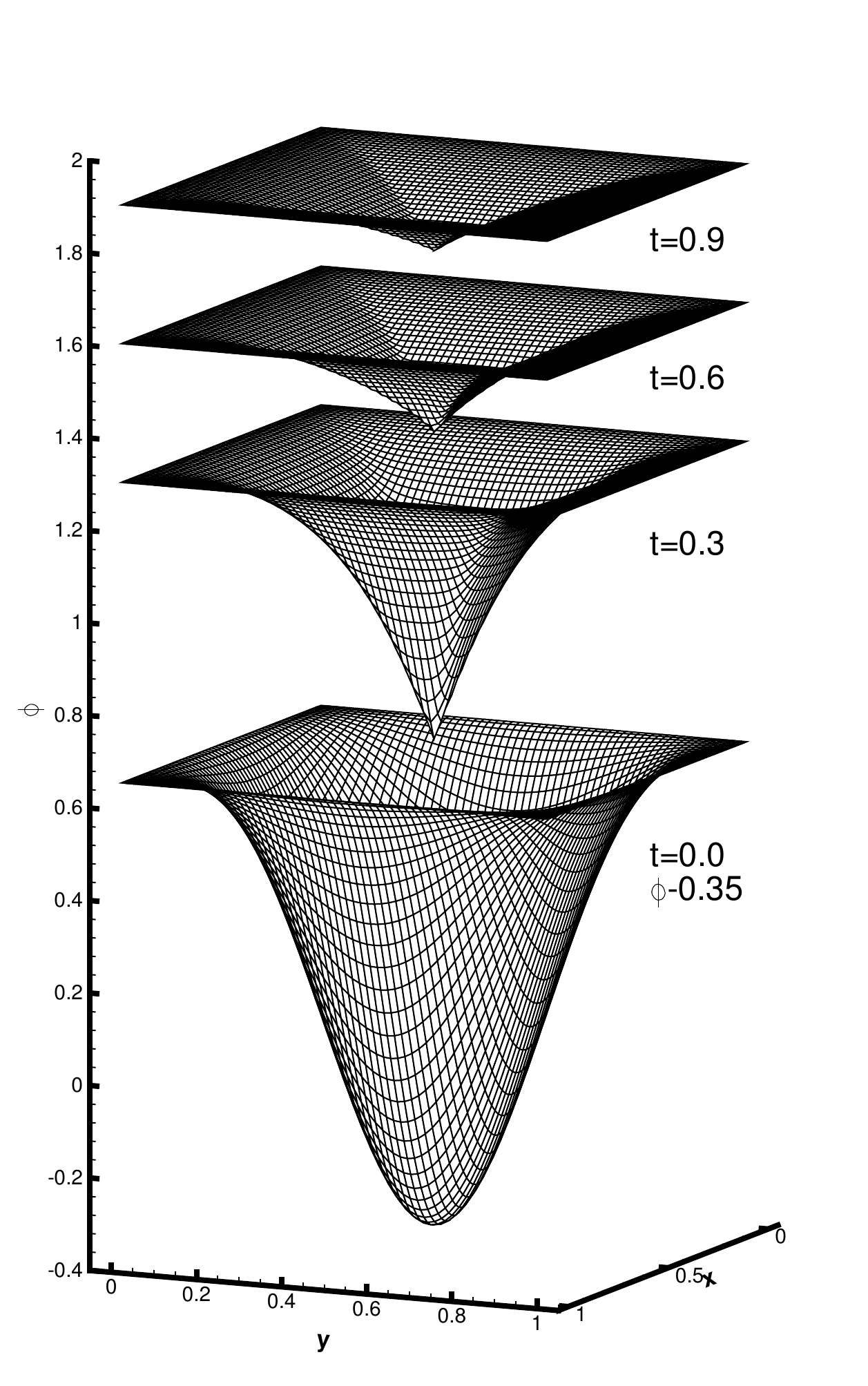}\label{Fig8a.2}}
	\caption{\em  Example \ref{ex8}: propagating surface. $\epsilon=0$. $60\times60$ grid points. $k=3$ with $\beta=0.6$. }
	\label{Fig8a}
\end{figure}

\begin{figure}
	\centering
	\subfigure[Propagating surface. CFL=0.5. ]{
		\includegraphics[width=0.4\textwidth]{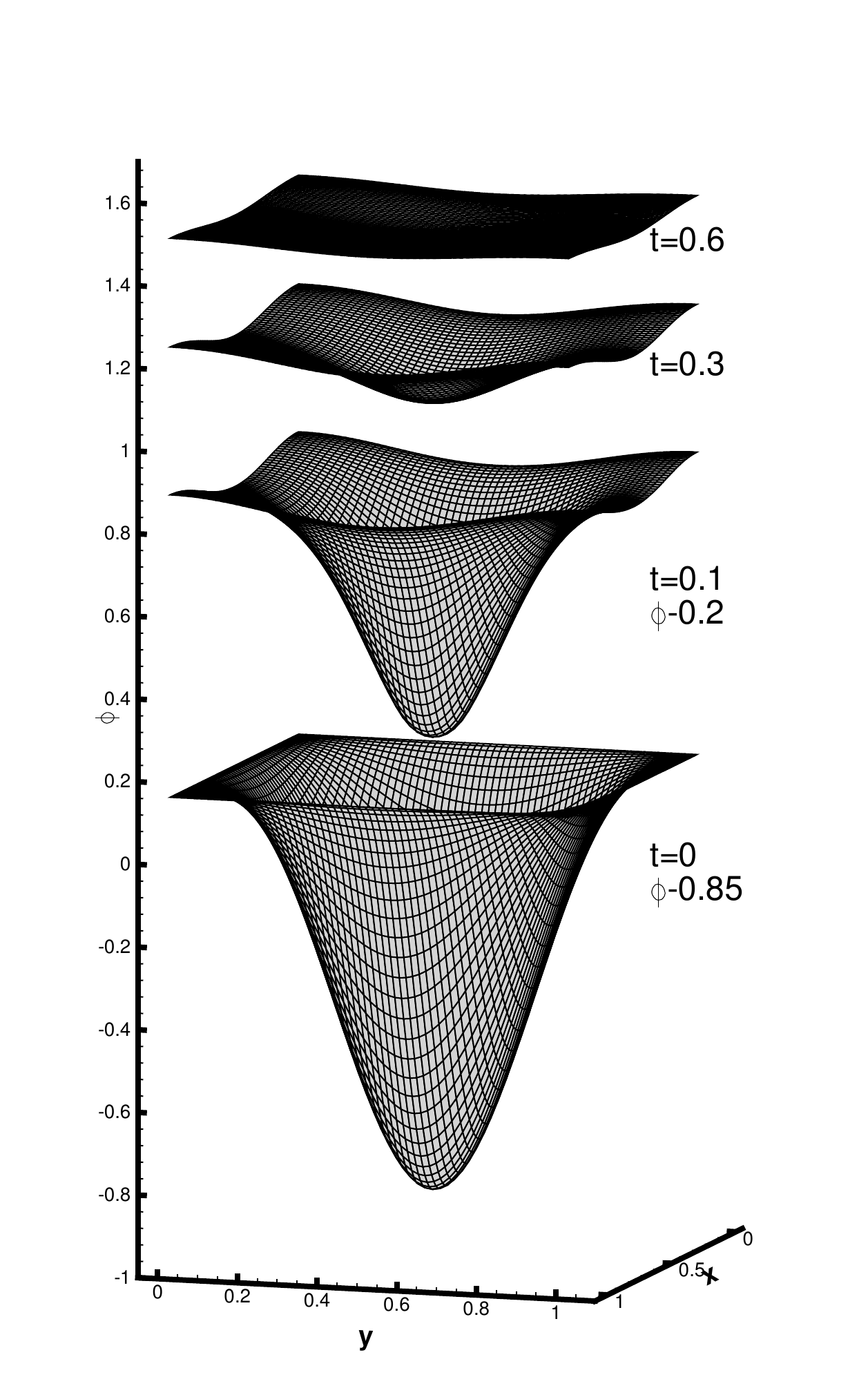}\label{Fig8b.1}}
	\subfigure[Propagating surface. CFL=2. ]{
		\includegraphics[width=0.4\textwidth]{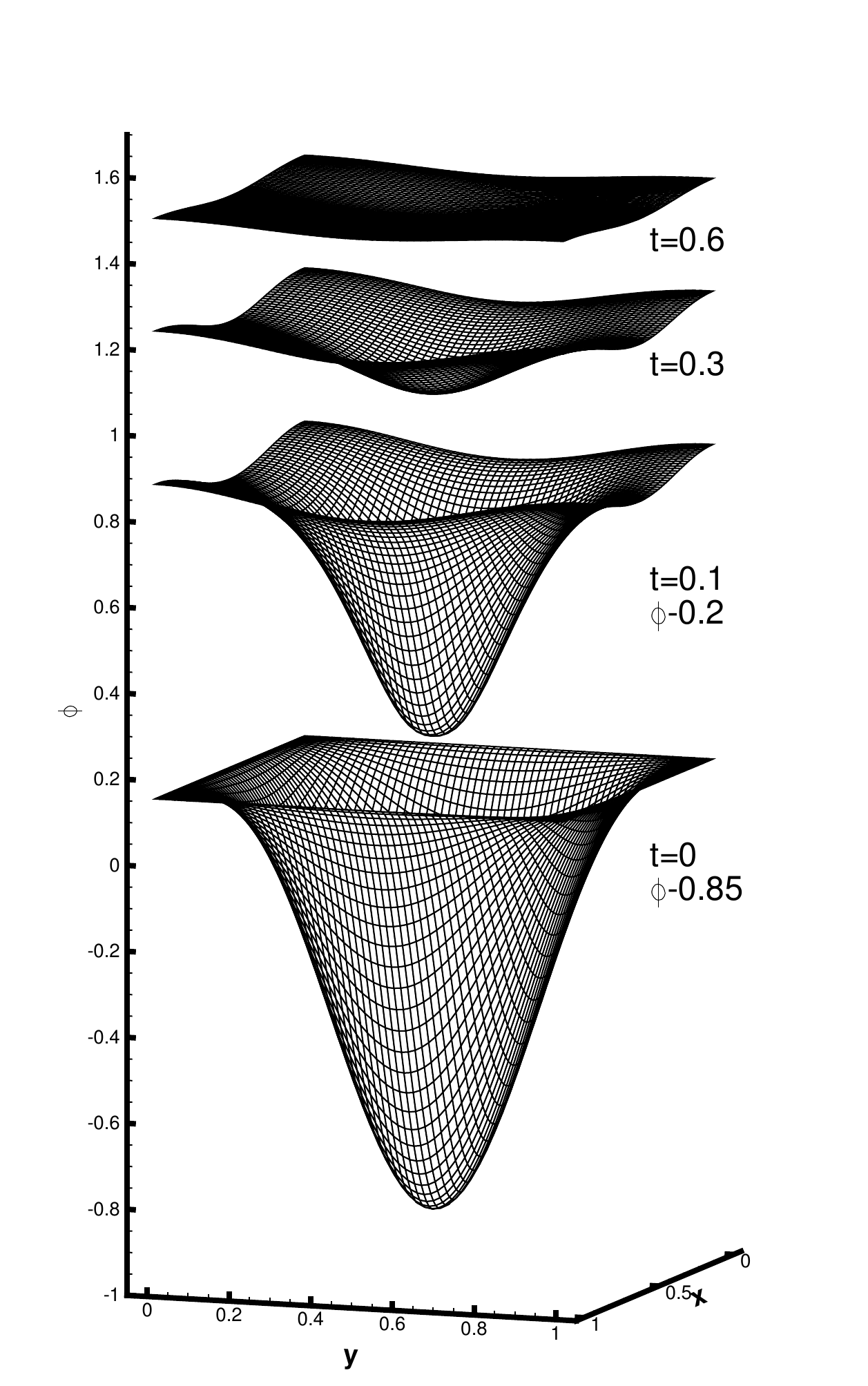}\label{Fig8b.2}}
	\caption{\em  Example \ref{ex8}: propagating surface. $\epsilon=0.1$. $60\times60$ grid points. $k=3$ with $\beta=0.4$. }
	\label{Fig8b}
\end{figure}

\begin{exa}\label{ex9}
As our last example, we consider the same problem of a propagating surface \eqref{eq:propagting} with $\epsilon=0$ but on the unit disk $x^2 + y^2 \leq 1$. The initial condition is 
$$\phi(x, y, 0) = \sin \left( \frac{\pi}{2}(x^2+y^2) \right). $$
A Dirichlet boundary condition $$\phi(x,y,t)=1+t, \quad \text{for all}\quad x^2+y^2=1$$
 is imposed. We would like to use this example to demonstrate the ability of the proposed methods to handle complex geometry and general boundary conditions. 
 The domain is discretized by embedding the boundary
 in a regular Cartesian mesh. The mesh points include all the interior points as well as the intersections of all possible Cartesian grid
 lines $x = x_i$ and $y = y_j$ with the boundary curve, as illustrated in Figure \ref{Fig9a}. In Figure \ref{Fig9b}, we plot several snapshots of the evolution of the surface. 
 %The numerical results agree with the benchmarks in the literature.
 
\end{exa}

\begin{figure}
	\centering
		\includegraphics[width=0.4\textwidth]{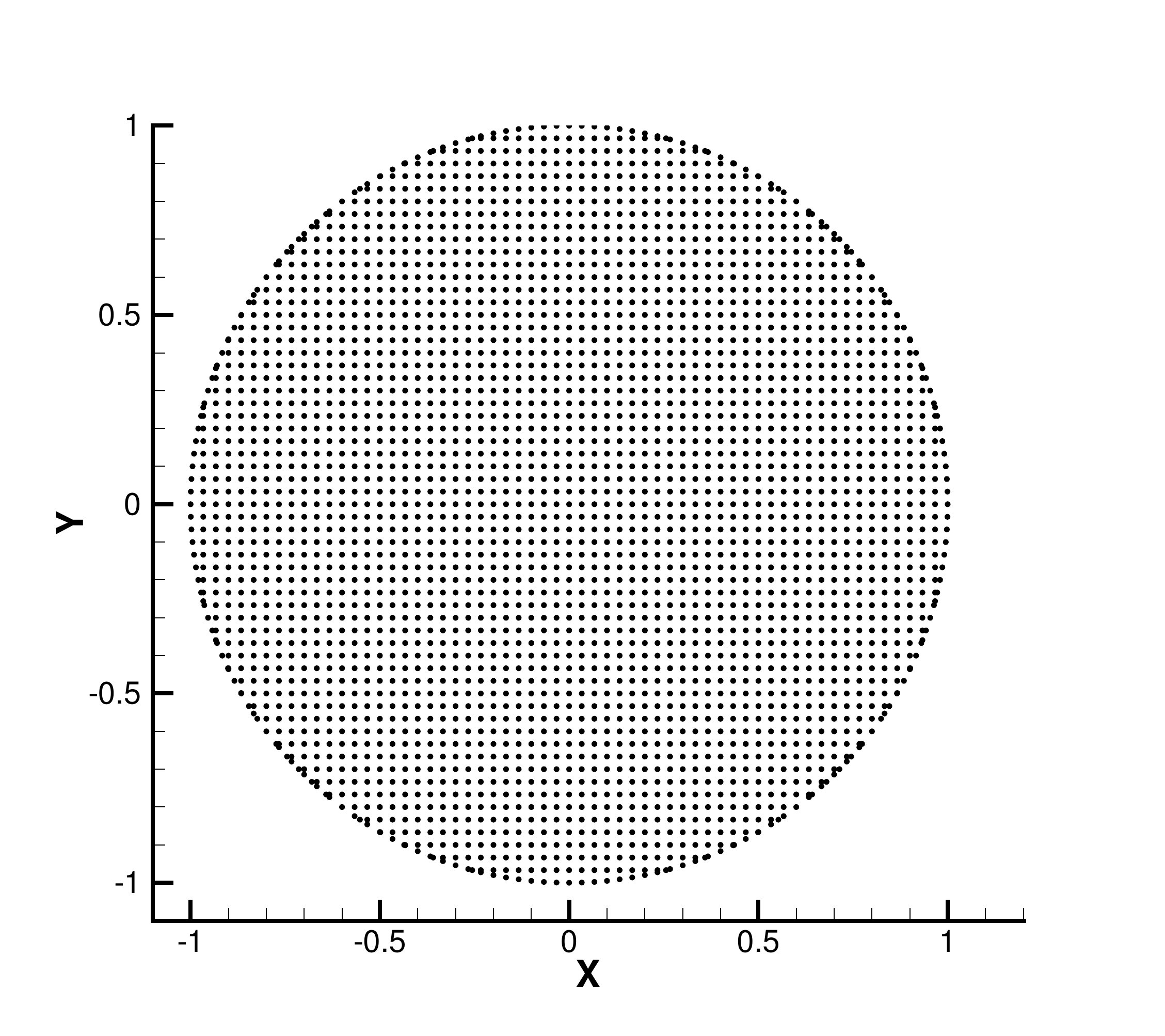}
	\caption{\em  Example \ref{ex9}: The discretization of the domain which is embedded in a Cartesian grid with $N_{x}=N_{y}=60$. }
	\label{Fig9a}
\end{figure}

\begin{figure}
	\centering
	\subfigure[Propagating surface. CFL=0.5. ]{
		\includegraphics[width=0.4\textwidth]{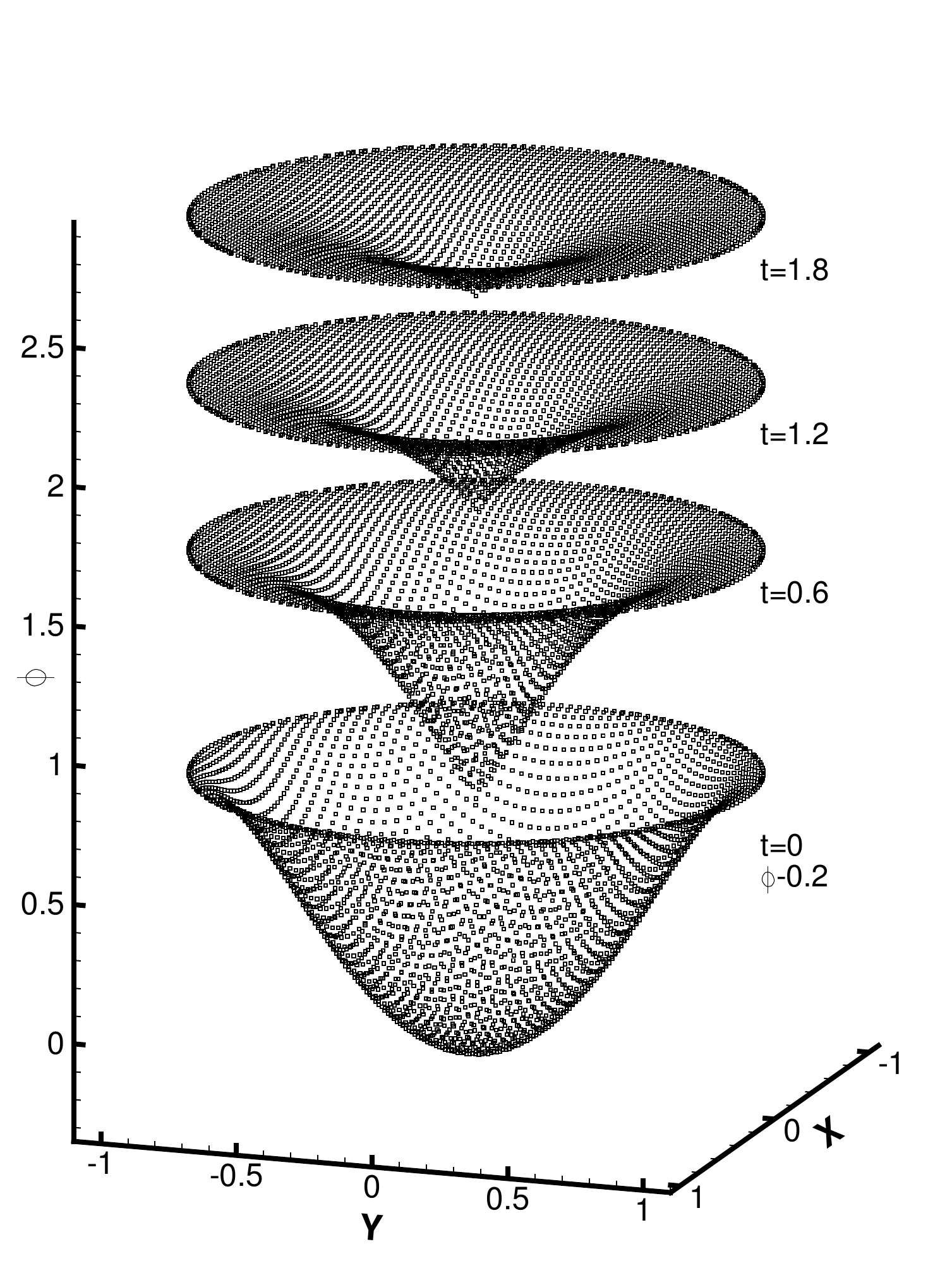}\label{Fig9.1}}
	\subfigure[Propagating surface. CFL=2. ]{
		\includegraphics[width=0.4\textwidth]{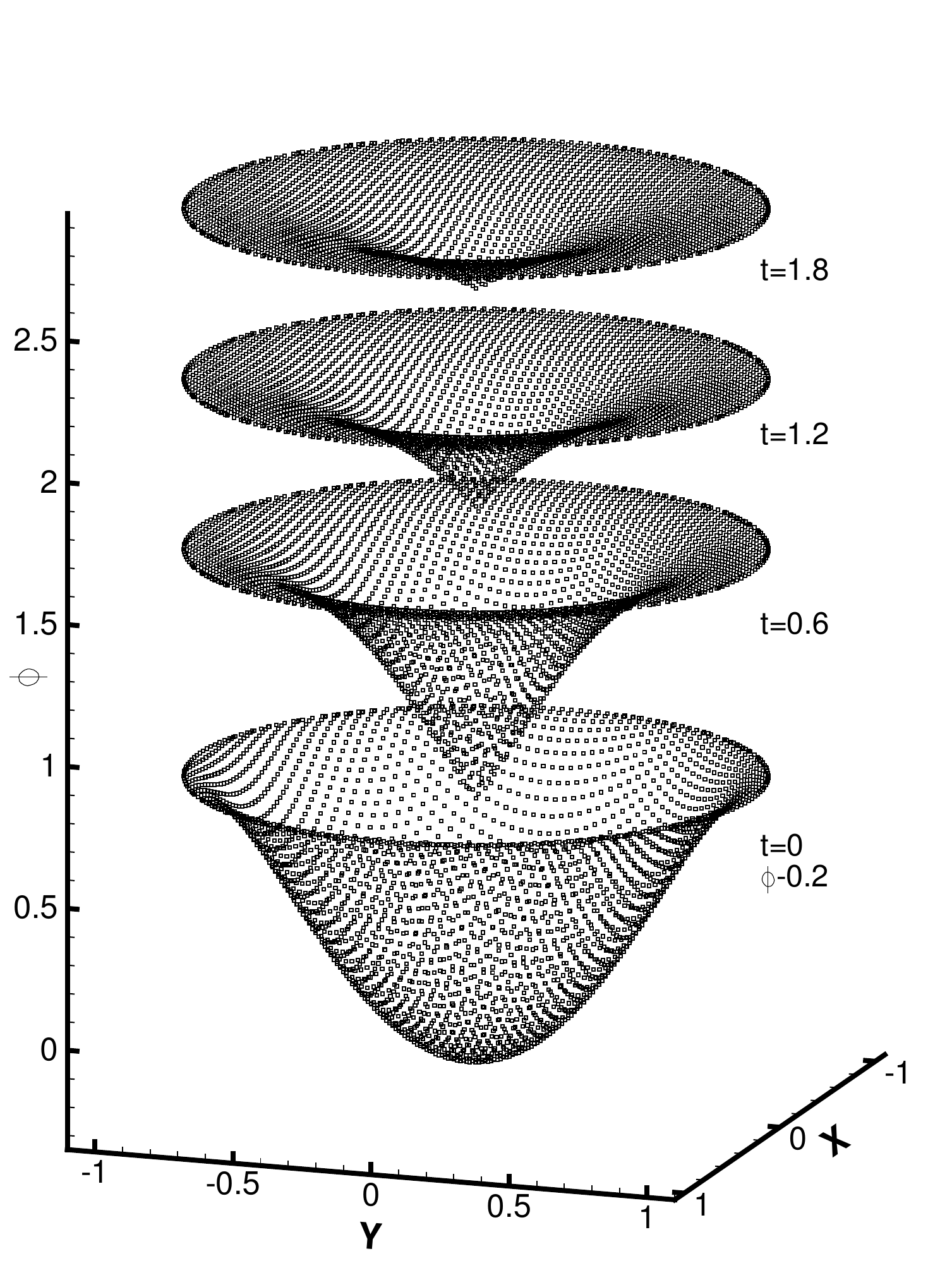}\label{Fig9.2}}
	\caption{\em  Example \ref{ex9}: propagating surface on the unit disk. $k=3$ with $\beta=0.6$. }
	\label{Fig9b}
\end{figure}

\begin{comment}

\textbf{Example 8 (steady state):}\label{ex8}
A problem from computer vision
\begin{align}
\left\{\begin{array}{ll}
\phi_{t} + I(x,y)\sqrt{\phi_{x}^2+\phi_{y}^2+1} - 1=0, & -1\leq x, y\leq 1\\
\phi(x,y,0) = 0\\
\end{array}
\right.
\end{align}
with $\phi=0$ as the boundary condition.

Firstly, we take
$$I(x, y) = 1/\sqrt{1+(1-|x|)^2+(1-|y|)^2}.$$
The exact steady solution is $\phi(x, y, \infty) = (1-|x|)(1-|y|)$.
\end{comment}

\section{Conclusion}

In this paper, we developed a class of high order non-oscillatory numerical schemes with unconditional stability for solving the nonlinear Hamilton-Jacobi (H-J) equations. The development of the schemes was based on our previous work \cite{christlieb2017kernel}, in which the spatial derivatives of a function were represented as a special kernel-based formulation. By incorporating a robust WENO methodology as well as a novel nonlinear filter, the proposed methods are able to suppress spurious oscillations and capture the correct viscosity solution.  The high order explicit strong-stability-preserving Runge-Kutta methods are employed for time integration, which do not require matrix inversion or solving nonlinear algebraic equations. As the most notable  advantage, the proposed schemes attain    unconditional stability, high order accuracy (up to third order accuracy), and the essentially non-oscillatory property at the same time.  In addition, the methods are able to handle complicated geometry and general boundary conditions if we further incorporate the inverse Lax-Wendroff boundary treatment technique. In the future, we plan to extend our schemes to solve other types of time-dependent problems, such as hyperbolic conservation laws. 
%{\color{red} We also would like to prove the convergence for the schemes towards the entropy solution for solving H-J equations.}

\bibliographystyle{abbrv}
\bibliography{ref}

\begin{thebibliography}{10}

\bibitem{abgrall1996numerical}
R.~Abgrall.
\newblock Numerical discretization of the first-order {H}amilton--{J}acobi
  equation on triangular meshes.
\newblock {\em Communications on pure and applied mathematics},
  49(12):1339--1373, 1996.

\bibitem{barnes1986hierarchical}
J.~Barnes and P.~Hut.
\newblock {A hierarchical $O(N \log N)$ force-calculation algorithm}.
\newblock {\em Nature}, 324:446--449, 1986.

\bibitem{borges2008improved}
R.~Borges, M.~Carmona, B.~Costa, and W.~S. Don.
\newblock An improved weighted essentially non-oscillatory scheme for
  hyperbolic conservation laws.
\newblock {\em Journal of Computational Physics}, 227(6):3191--3211, 2008.

\bibitem{causley2017method}
M.~Causley, H.~Cho, and A.~Christlieb.
\newblock Method of lines transpose: Energy gradient flows using direct
  operator inversion for phase-field models.
\newblock {\em SIAM Journal on Scientific Computing}, 39(5):B968--B992, 2017.

\bibitem{causley2014method}
M.~Causley, A.~Christlieb, B.~Ong, and L.~Van~Groningen.
\newblock Method of lines transpose: An implicit solution to the wave equation.
\newblock {\em Mathematics of Computation}, 83(290):2763--2786, 2014.

\bibitem{causley2016method}
M.~F. Causley, H.~Cho, A.~J. Christlieb, and D.~C. Seal.
\newblock Method of lines transpose: High order {L}-stable $\mathcal{O}({N})$
  schemes for parabolic equations using successive convolution.
\newblock {\em SIAM Journal on Numerical Analysis}, 54(3):1635--1652, 2016.

\bibitem{causley2013method}
M.~F. Causley, A.~J. Christlieb, Y.~Guclu, and E.~Wolf.
\newblock Method of lines transpose: A fast implicit wave propagator.
\newblock {\em arXiv preprint arXiv:1306.6902}, 2013.

\bibitem{cheng2017asymptotic}
Y.~Cheng, A.~J. Christlieb, W.~Guo, and B.~Ong.
\newblock An asymptotic preserving {M}axwell solver resulting in the darwin
  limit of electrodynamics.
\newblock {\em Journal of Scientific Computing}, 71(3):959--993, 2017.

\bibitem{cheng2007discontinuous}
Y.~Cheng and C.-W. Shu.
\newblock A discontinuous {G}alerkin finite element method for directly solving
  the {H}amilton--{J}acobi equations.
\newblock {\em Journal of Computational Physics}, 223(1):398--415, 2007.

\bibitem{cheng2014new}
Y.~Cheng and Z.~Wang.
\newblock A new discontinuous {G}alerkin finite element method for directly
  solving the {H}amilton--{J}acobi equations.
\newblock {\em Journal of Computational Physics}, 268:134--153, 2014.

\bibitem{chow2017algorithm}
Y.~Chow, J.~Darbon, S.~Osher, and W.~Yin.
\newblock {Algorithm for overcoming the curse of dimensionality for
  time-dependent non-convex {H}amilton--{J}acobi equations arising from optimal
  control and differential games problems}.
\newblock {\em Journal of Scientific Computing}, 73(2-3):617--643, 2017.

\bibitem{christlieb2016weno}
A.~Christlieb, W.~Guo, and Y.~Jiang.
\newblock {A WENO-based Method of Lines Transpose approach for Vlasov
  simulations}.
\newblock {\em Journal of Computational Physics}, 327:337--367, 2016.

\bibitem{christlieb2017kernel}
A.~Christlieb, W.~Guo, and Y.~Jiang.
\newblock Kernel based high order ``explicit" unconditionally-stable scheme for
  nonlinear degenerate advection-diffusion equations.
\newblock {\em arXiv preprint arXiv:1707.09294}, 2017.

\bibitem{crandall1984some}
M.~G. Crandall, L.~C. Evans, and P.-L. Lions.
\newblock Some properties of viscosity solutions of {H}amilton-{J}acobi
  equations.
\newblock {\em Transactions of the American Mathematical Society},
  282(2):487--502, 1984.

\bibitem{crandall1983viscosity}
M.~G. Crandall and P.-L. Lions.
\newblock Viscosity solutions of {H}amilton-{J}acobi equations.
\newblock {\em Transactions of the American Mathematical Society},
  277(1):1--42, 1983.

\bibitem{crandall1984two}
M.~G. Crandall and P.-L. Lions.
\newblock Two approximations of solutions of {H}amilton--{J}acobi equations.
\newblock {\em Mathematics of Computation}, 43(167):1--19, 1984.

\bibitem{darbon2016algorithms}
J.~Darbon and S.~Osher.
\newblock {Algorithms for overcoming the curse of dimensionality for certain
  {H}amilton--{J}acobi equations arising in control theory and elsewhere}.
\newblock {\em Research in the Mathematical Sciences}, 3(1):19, 2016.

\bibitem{gottlieb2005high}
S.~Gottlieb.
\newblock On high order strong stability preserving {R}unge--{K}utta and multi
  step time discretizations.
\newblock {\em Journal of Scientific Computing}, 25(1):105--128, 2005.

\bibitem{gottlieb2001strong}
S.~Gottlieb, C.-W. Shu, and E.~Tadmor.
\newblock Strong stability-preserving high-order time discretization methods.
\newblock {\em SIAM review}, 43(1):89--112, 2001.

\bibitem{greengard1987fast}
L.~Greengard and V.~Rokhlin.
\newblock A fast algorithm for particle simulations.
\newblock {\em Journal of Computational Physics}, 73(2):325--348, 1987.

\bibitem{hu1999discontinuous}
C.~Hu and C.-W. Shu.
\newblock A discontinuous {G}alerkin finite element method for
  {H}amilton--{J}acobi equations.
\newblock {\em SIAM Journal on Scientific computing}, 21(2):666--690, 1999.

\bibitem{huang2008numerical}
L.~Huang, C.-W. Shu, and M.~Zhang.
\newblock Numerical boundary conditions for the fast sweeping high order {WENO}
  methods for solving the {E}ikonal equation.
\newblock {\em Journal of Computational Mathematics}, pages 336--346, 2008.

\bibitem{jia2008krylov}
J.~Jia and J.~Huang.
\newblock Krylov deferred correction accelerated method of lines transpose for
  parabolic problems.
\newblock {\em Journal of Computational Physics}, 227(3):1739--1753, 2008.

\bibitem{jiang2000weighted}
G.-S. Jiang and D.~Peng.
\newblock Weighted {ENO} schemes for {H}amilton--{J}acobi equations.
\newblock {\em SIAM Journal on Scientific computing}, 21(6):2126--2143, 2000.

\bibitem{kropinski2011fast}
M.~C.~A. Kropinski and B.~D. Quaife.
\newblock Fast integral equation methods for rothe�s method applied to the
  isotropic heat equation.
\newblock {\em Computers \& Mathematics with Applications}, 61(9):2436--2446,
  2011.

\bibitem{lafon1996high}
F.~Lafon and S.~Osher.
\newblock High order two dimensional nonoscillatory methods for solving
  {H}amilton--{J}acobi scalar equations.
\newblock {\em Journal of Computational Physics}, 123(2):235--253, 1996.

\bibitem{lepsky2000analysis}
O.~Lepsky, C.~Hu, and C.-W. Shu.
\newblock Analysis of the discontinuous {G}alerkin method for
  {H}amilton--{J}acobi equations.
\newblock {\em Applied Numerical Mathematics}, 33(1-4):423--434, 2000.

\bibitem{osher1988fronts}
S.~Osher and J.~A. Sethian.
\newblock Fronts propagating with curvature-dependent speed: algorithms based
  on {H}amilton-{J}acobi formulations.
\newblock {\em Journal of computational physics}, 79(1):12--49, 1988.

\bibitem{osher1991high}
S.~Osher and C.-W. Shu.
\newblock High-order essentially nonoscillatory schemes for
  {H}amilton--{J}acobi equations.
\newblock {\em SIAM Journal on numerical analysis}, 28(4):907--922, 1991.

\bibitem{qiu2007hermite}
J.~Qiu.
\newblock Hermite {WENO} schemes with lax-wendroff type time discretizations
  for {H}amilton--{J}acobi equations.
\newblock {\em Journal of Computational Mathematics}, pages 131--144, 2007.

\bibitem{qiu2005hermite}
J.~Qiu and C.-W. Shu.
\newblock Hermite {WENO} schemes for {H}amilton--{J}acobi equations.
\newblock {\em Journal of Computational Physics}, 204(1):82--99, 2005.

\bibitem{salazar2000theoretical}
A.~Salazar, M.~Raydan, and A.~Campo.
\newblock Theoretical analysis of the exponential transversal method of lines
  for the diffusion equation.
\newblock {\em Numerical Methods for Partial Differential Equations},
  16(1):30--41, 2000.

\bibitem{schemann1998adaptive}
M.~Schemann and F.~A. Bornemann.
\newblock An adaptive rothe method for the wave equation.
\newblock {\em Computing and Visualization in Science}, 1(3):137--144, 1998.

\bibitem{shu2002survey}
C.-W. Shu.
\newblock A survey of strong stability preserving high order time
  discretizations.
\newblock {\em Collected lectures on the preservation of stability under
  discretization}, 109:51--65, 2002.

\bibitem{shu2007high}
C.-W. Shu.
\newblock High order numerical methods for time dependent {H}amilton--{J}acobi
  equations.
\newblock {\em Mathematics and Computation in Imaging Science and Information
  Processing, Lect. Notes Ser. Inst. Math. Sci. Natl. Univ. Singap}, 11:47--91,
  2007.

\bibitem{shu2009high}
C.-W. Shu.
\newblock High order weighted essentially nonoscillatory schemes for convection
  dominated problems.
\newblock {\em SIAM review}, 51(1):82--126, 2009.

\bibitem{tan2010inverse}
S.~Tan and C.-W. Shu.
\newblock Inverse {L}ax-{W}endroff procedure for numerical boundary conditions
  of conservation laws.
\newblock {\em Journal of Computational Physics}, 229(21):8144--8166, 2010.

\bibitem{tan2012efficient}
S.~Tan, C.~Wang, C.-W. Shu, and J.~Ning.
\newblock Efficient implementation of high order inverse {L}ax--{W}endroff
  boundary treatment for conservation laws.
\newblock {\em Journal of Computational Physics}, 231(6):2510--2527, 2012.

\bibitem{tao2017dimension}
Z.~Tao and J.~Qiu.
\newblock Dimension-by-dimension moment-based central {H}ermite {WENO} schemes
  for directly solving {H}amilton--{J}acobi equations.
\newblock {\em Advances in Computational Mathematics}, 43(5):1023--1058, 2017.

\bibitem{xiong2010fast}
T.~Xiong, M.~Zhang, Y.-T. Zhang, and C.-W. Shu.
\newblock Fast sweeping fifth order {WENO} scheme for static
  {H}amilton--{J}acobi equations with accurate boundary treatment.
\newblock {\em Journal of Scientific Computing}, 45(1-3):514--536, 2010.

\bibitem{yan2011local}
J.~Yan and S.~Osher.
\newblock A local discontinuous {G}alerkin method for directly solving
  {H}amilton--{J}acobi equations.
\newblock {\em Journal of Computational Physics}, 230(1):232--244, 2011.

\bibitem{zhang2003high}
Y.-T. Zhang and C.-W. Shu.
\newblock High-order {WENO} schemes for {H}amilton--{J}acobi equations on
  triangular meshes.
\newblock {\em SIAM Journal on Scientific Computing}, 24(3):1005--1030, 2003.

\bibitem{zheng2017finite}
F.~Zheng, C.-W. Shu, and J.~Qiu.
\newblock Finite difference {H}ermite {WENO} schemes for the
  {H}amilton--{J}acobi equations.
\newblock {\em Journal of Computational Physics}, 337:27--41, 2017.

\bibitem{zhu2013hermite}
J.~Zhu and J.~Qiu.
\newblock Hermite {WENO} schemes for {H}amilton--{J}acobi equations on
  unstructured meshes.
\newblock {\em Journal of Computational Physics}, 254:76--92, 2013.

\end{thebibliography}

\end{document}